\documentclass[11pt,a4paper]{article}
\usepackage[english]{babel}
\usepackage{amssymb}
\usepackage{amsmath}
\usepackage{amsthm}
\usepackage{authblk}
\usepackage{graphicx}
\usepackage{todonotes}
\usepackage{appendix}
\usepackage{comment}
\usepackage{tikz}
\usepackage{overpic}
\usepackage[nottoc,numbib]{tocbibind} 

\theoremstyle{plain}
\newtheorem*{theorem*}{Theorem}
\newtheorem{theorem}{Theorem}[section] 
\newtheorem{lemma}[theorem]{Lemma}
\newtheorem{proposition}[theorem]{Proposition}
\newtheorem{corollary}[theorem]{Corollary}

\newtheorem{assumption}[theorem]{Assumption}

\theoremstyle{definition}
\newtheorem{definition}[theorem]{Definition}
\newtheorem{remark}[theorem]{Remark}

\makeatletter
\@addtoreset{equation}{section}
\makeatother
\numberwithin{equation}{section}

\DeclareMathOperator{\im}{Im}
\DeclareMathOperator{\re}{Re}
\renewcommand{\Re}{\re}
\renewcommand{\Im}{\im}
\DeclareMathOperator{\diag}{diag}

\newcommand{\ds}{\displaystyle}

\DeclareMathOperator{\sgn}{sgn}

\title{The two periodic Aztec diamond and matrix valued orthogonal
polynomials}

\author[1]{Maurice Duits}
\author[2]{Arno B.J. Kuijlaars}
\affil[1]{Department of Mathematics, 
Royal Institute of Technology (KTH),
Stockholm, Sweden. Email: duits@kth.se}
\affil[2]{Department of Mathematics, KU Leuven - University of Leuven,
Belgium, 
Email: arno.kuijlaars@kuleuven.be}

\date{\today}

\usepackage{hyperref}
\AtBeginDocument{} 

\begin{document}

\maketitle


\begin{abstract}
We analyze domino tilings of the two-periodic Aztec diamond
by means of matrix valued orthogonal polynomials
that we obtain from a reformulation of the Aztec diamond 
as  a non-intersecting path model with periodic 
transition matrices. In a more general framework we express the
correlation kernel for the underlying determinantal point
process as a double contour integral that contains 
the reproducing kernel of matrix valued orthogonal 
polynomials. We use the Riemann-Hilbert problem to 
simplify this formula for the case of the two-periodic Aztec diamond.
 
In the large size limit we recover the three phases
of the model known as solid, liquid and gas. 
We describe fine asymptotics for the gas phase
and at the cusp points of the liquid-gas boundary, 
thereby complementing and extending results of
Chhita and Johansson. 

\end{abstract}
\tableofcontents

\section{Introduction} \label{sec:intro} 
We study domino tilings of the Aztec diamond with
a two periodic weighting. This model  falls into a class of models 
for which existing techniques for studying fine asymptotics are not adequate and only recently first 
important progress has been made \cite{BCJ,CJ}.  
We  introduce a new approach  based on matrix valued 
orthogonal polynomials that allows us to compute the determinantal
correlations at finite size and their asymptotics as the size of
the diamond gets large in a rather orderly way. 
We strongly believe that  this approach will also prove to be 
a good starting point for other tiling models with a periodic weighting. 

Random tilings of planar domains have been studied intensively in 
the past decade. Such models exhibit a rich structure 
including a limit shape and fluctuations that are expected 
to fall in  various important universality classes 
(see  \cite{BKMM,CEP,CKP,J05,J17,K,KO,KOS} for general references on the topic
and \cite{BK,Pe1,Pe2} for recent contributions).
When the correlation structure is determinantal, there is hope 
to understand the fine  asymptotic structure  by studying 
the asymptotic behavior of the correlation kernel. 
An important source of  examples of such models is the Schur 
process \cite{OR1}. For these models the correlation kernel 
can be explicitly computed in terms of a  double integral
representation, opening up to the possibility of performing an 
asymptotic analysis by means of classical steepest decent 
(or stationary phase) techniques. 

Of course, the  Schur process is rather special and many models 
of interest fall  outside this class. In particular this is 
true for random tilings or dimer models  with doubly periodic weightings.
Yet, these models have exciting new features and have 
therefore been discussed in the physics and mathematics literature
\cite{BCJ,CJ,CY,FR,KOS,NHB}. An important feature is the 
appearance of a so-called gas phase. For instance, in the 
two periodic weighting  for domino tiling of the Aztec diamond, 
the diamond can be partitioned into three regions: 
the solid, liquid and gas region \cite{OR1} 
(as we will see in Figure \ref{fig:AlgCurve} below). 
The gas region has not been observed in models that are in the Schur class. 
The $2$-point correlations  (for an associated particle process) 
in  the gas region behave differently when compared to the 
liquid regions. Indeed, the correlation kernel decays exponentially 
with the distance $d$ between the points, 
instead of $\sim 1/d$ in the liquid region. At the 
liquid-solid boundary one expects the Airy process to appear, 
but the situation at the gas-liquid boundary is far more 
complicated \cite{BCJ,CJ}. 

To the best of our knowledge, the two periodic Aztec diamond is 
the only  model  with periodic weightings for which rigorous 
results on fine asymptotics exist \cite{BCJ, CJ}. Inspired by  
a formula for  the Kasteleyn matrix found by Chhita and Young
\cite{CY}, Chhita and Johansson \cite{CJ} found a way to 
compute the asymptotic  behavior of the Kasteleyn matrix as the 
size of the Aztec diamond goes to infinity.  
We will follow  a different approach to studying such models with periodic weightings.

As we will recall in Section \ref{sec:paths}, the Aztec diamond 
can be described by non-intersecting paths  (we refer to \cite{J17} 
and the references therein for more background on the relation 
between dimers, tilings, non-intersecting paths and all that). 
For a general class of discrete non-intersecting paths with  
$p$-periodic transition matrices
(which includes $p$-periodic weightings for domino tilings of 
the Aztec diamond and $p$-periodic weightings for lozenge tilings 
of the hexagon), we show in Section \ref{sec:MVOP} how the 
correlation kernel can be written as a double integral 
formula involving matrix valued polynomials that satisfy 
a non-hermitian orthogonality.

We believe that this general setup has a high potential 
for a rigorous asymptotic analysis. The key fact is that 
these matrix valued orthogonal polynomials can be characterized 
in terms of the solution of  a $2p \times 2p$ matrix valued Riemann-Hilbert 
problem. With the highly developed Riemann-Hilbert toolkit 
at hand, we may thus hope to compute the asymptotic behavior 
of the polynomials, and more importantly the correlation
kernel. The formalism will be worked out in Section \ref{sec:MVOP}.
It provides a new perspective even on the classical examples 
of uniform domino tilings of  the Aztec diamond and 
lozenge tilings of a hexagon, as we will discuss briefly in 
Sections~\ref{subsec:uniformAD} and \ref{subsec:hexagon}.

The main focus of the paper is to show how the Riemann-Hilbert
approach can be exploited to find an asymptotic analysis for 
the two periodic Aztec diamond. Remarkably, in this case the 
result of the Riemann-Hilbert analysis is a surprisingly simple 
double integral formula for the correlation kernel. 
It is not an asymptotic result, but an exact formula valid for 
fixed finite $N$.
This representation also appears to be more elementary than the 
one given in Chhita and Johansson \cite{CJ}. 
The Riemann-Hilbert analysis is given in Section \ref{sec:RHP}. 
We analyze the double integral formula for the  kernel 
asymptotically using classical steepest descent techniques 
in Section \ref{sec:asymp}. 

The model of the two periodic Aztec diamond is explained and the main results are summarized in the next section.

\section{Statement of results} \label{sec:results}
In this section we will introduce the two periodic Aztec diamond and state our main results for this model. 

\subsection{Definition of the model}

The Aztec diamond is a region on the square lattice with a sawtooth boundary that can
be covered by $2 \times 1$ and $1 \times 2$ rectangles, called dominos. 
The squares have a black/white checkerboard coloring and a possible
tiling of the Aztec diamond of size $4$ is shown in Figure \ref{fig:AztecTiling}. 
There are four types of dominos, namely North, West, East, and South, 
that are also shown in the figure. 
The Aztec diamond model was first introduced in \cite{EKLP}.

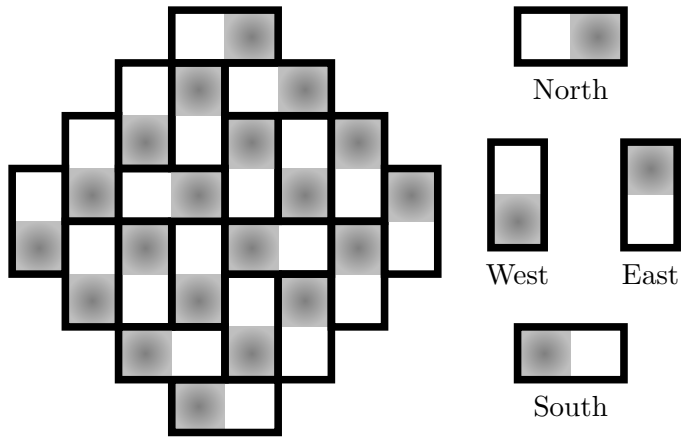
\begin{figure}[t]
\begin{center}
\begin{tikzpicture}[scale=0.7]


\foreach \x/\y in {-4/0,-3/1,-3/-1,-2/2,-1/-1,0/-2,1/1}
{ \draw[outer color=lightgray,inner color=gray]
(\x,\y)--(\x,\y-1)--(\x+1,\y-1)--(\x+1,\y); 
\draw [line width = 1mm] (\x,\y-1) rectangle (\x+1,\y+1);
}

\foreach \x/\y in {-2/-1,-1/2,0/1,1/-2,2/-1,2/1,3/0}
{  \draw[outer color=lightgray,inner color=gray]
(\x,\y)--(\x,\y+1)--(\x+1,\y+1)--(\x+1,\y);
\draw [line width = 1mm] (\x,\y-1) rectangle (\x+1,\y+1);
}

\foreach \x/\y in {0/-3,-1/-2,1/0}
{  \draw[outer color=lightgray,inner color=gray]
(\x,\y-1)--(\x-1,\y-1)--(\x-1,\y)--(\x,\y);
\draw [line width = 1mm] (\x-1,\y-1) rectangle (\x+1,\y);
}

\foreach \x/\y in {-1/1,0/4,1/3}
{  \draw[outer color=lightgray,inner color=gray]
(\x,\y-1)--(\x+1,\y-1)--(\x+1,\y)--(\x,\y);
\draw [line width=1mm] (\x-1,\y-1) rectangle (\x+1,\y);
}

 \draw[outer color=lightgray,inner color=gray]
(5,0.5)--(5,-0.5)--(6,-0.5)--(6,0.5);
\draw [line width = 1mm] (5,-0.5) rectangle (6,1.5);
\draw (5.5,-1) node {West};

 \draw[outer color=lightgray,inner color=gray]
(6.5,3)--(7.5,3)--(7.5,4)--(6.5,4);
\draw[line width = 1mm] (5.5,3) rectangle (7.5,4);
\draw (6.5,2.5) node {North};

 \draw[outer color=lightgray,inner color=gray]
(6.5,-3)--(5.5,-3)--(5.5,-2)--(6.5,-2);
\draw[line width = 1mm] (5.5,-3) rectangle (7.5,-2);
\draw (6.5,-3.5) node {South};

 \draw[outer color=lightgray,inner color=gray]
(7.5,0.5)--(7.5,1.5)--(8.5,1.5)--(8.5,0.5);
\draw[line width = 1mm] (7.5,-0.5) rectangle (8.5,1.5);
\draw (8,-1) node {East};
\end{tikzpicture}

\caption{Possible tiling of a $4 \times 4$ Aztec diamond with four kinds of dominos.
\label{fig:AztecTiling} } 
\end{center}
\end{figure}

In the two periodic Aztec diamond we assign a weight to each
domino in a tiling, depending on its shape (horizontal or vertical)
and its location in the Aztec diamond. We assume the Aztec diamond
is of even size.

To describe the two periodic weighting we introduce a coordinate system where 
$(0,0)$ is at the center of the Aztec diamond. 
The center of a horizontal domino has coordinates $(x,y + 1/2)$ with
$x,y \in \mathbb Z$. We then say that the horizontal
domino is in column $x$.
The center of a vertical domino has coordinates $(x+1/2,y)$ with 
$x,y \in \mathbb Z$, and we say that the
vertical domino is in row $y$. The row and column numbers 
run from $-N+1$ to $N-1$, where $2N$ is the size of the Aztec diamond.

We fix two positive numbers $a$ and $b$ and define the weights   as follows.
\begin{definition} The weight of a domino $D$ in a tiling $\mathcal T$ of
the Aztec diamond is
\begin{equation} \label{eq:weightD} 
	w(D) =
	 \begin{cases} a, & \text{if $D$ is a horizontal domino in an even column}, \\
	  b, & \text{if $D$ is a horizontal domino in an odd column}, \\
	  b, & \text{if $D$ is a vertical domino in an even row}, \\
	  a, & \text{if $D$ is a vertical domino in an odd row}. 
	\end{cases}	 
	 \end{equation}

The weight of the tiling $\mathcal T$ is 
\begin{equation} \label{eq:weightT}
	w(\mathcal T) = \prod_{D \in \mathcal T}  w(D),  
	\end{equation}
and the probability for $\mathcal T$ is
\begin{equation} \label{eq:probT} 
	\mathcal P(\mathcal T) =   \frac{w(\mathcal T)}{Z_N}, 
	\end{equation}
where $Z_N = \sum\limits_{\mathcal T'} w(\mathcal T')$ (sum over all
possible tilings $\mathcal T'$ of an Aztec diamond of size $2N$) 
is the partition function.
\end{definition}

In the example from Figure \ref{fig:AztecTiling} the weights are
shown in Figure \ref{fig:AztecWeights}. The weight of the tiling
is $w(\mathcal T) = a^{10} b^{10}$.

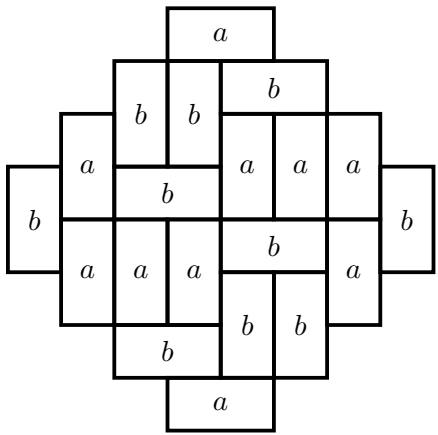
\begin{figure}[t]
\begin{center}
\begin{tikzpicture}[scale=.7]


\foreach \x/\y in {-4/0,-3/1,-3/-1,-2/2,-1/-1,0/-2,1/1}
{ 
\draw [line width = 0.5mm] (\x,\y-1) rectangle (\x+1,\y+1);
}

\foreach \x/\y in {-2/-1,-1/2,0/1,1/-2,2/-1,2/1,3/0}
{ 
\draw [line width = 0.5mm] (\x,\y-1) rectangle (\x+1,\y+1);
}

\foreach \x/\y in {0/-3,-1/-2,1/0}
{  
\draw [line width = 0.5mm] (\x-1,\y-1) rectangle (\x+1,\y);
}

\foreach \x/\y in {-1/1,0/4,1/3}
{ 
\draw [line width=0.5mm] (\x-1,\y-1) rectangle (\x+1,\y);
}

\foreach \x in {-3.5,3.5} \draw (\x,0) node {$b$};
\foreach \x in {-2.5,0.5,1.5,2.5} \draw (\x,1) node {$a$}; 
\foreach \x in {-2.5,-1.5,-0.5,2.5} \draw (\x,-1) node {$a$};
\foreach \x in {-1.5,-0.5} \draw (\x,2) node {$b$}; 
\foreach \x in {0.5,1.5} \draw (\x,-2) node {$b$};

\foreach \y in {-3.5,3.5} \draw(0,\y) node {$a$};
\foreach \y in {-0.5,2.5} \draw(1,\y) node {$b$};
\foreach \y in {-2.5,0.5} \draw(-1,\y) node {$b$};
\end{tikzpicture}
   \caption{Two periodic weights of dominos in a tiling of 
   the Aztec diamond. The vertical dominos
   in an even row and the horizontal dominos in an
   odd column have weight $a$. Other dominos have weight $b$. 
\label{fig:AztecWeights}}
\end{center}
\end{figure}

The model is homogeneous in the sense that the probabilities \eqref{eq:probT} 
do not change
if we multiply $a$ and $b$ by a common factor. We may and do assume $a b = 1$.
In what follows it will be more convenient to work with
\begin{equation} \label{eq:defalphabeta}
	\alpha = a^2, \qquad \beta = b^2 
	\end{equation}
instead of $a$ and $b$. We have $\alpha \beta = 1$, and without loss of
generality we assume $\alpha \geq 1$. If $\alpha = \beta = 1$ then the model
reduces to the uniform weighting on domino tilings, and so the true 
interest is in the case $\alpha > 1$, and this is what we assume from now on.

\subsection{Particle system and determinantal point process}

By putting a particle in the black square of the West and South dominos, we obtain a random particle system. In our running example
the particle systems is shown in Figure \ref{fig:AztecParticles}. 


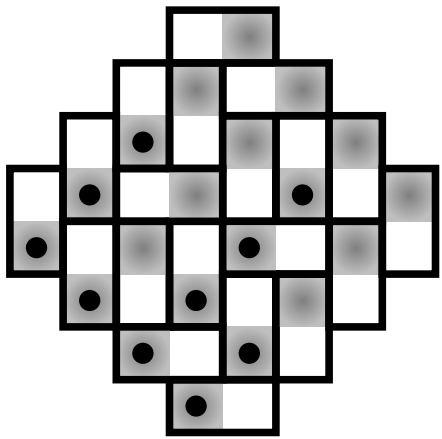
\begin{figure}[t]
\begin{center}
\begin{tikzpicture}[scale=.7]


\foreach \x/\y in {-4/0,-3/1,-3/-1,-2/2,-1/-1,0/-2,1/1}
{ \draw[outer color=lightgray,inner color=gray]
(\x,\y)--(\x,\y-1)--(\x+1,\y-1)--(\x+1,\y); 
\draw [line width = 1mm] (\x,\y-1) rectangle (\x+1,\y+1);
\fill (\x+1/2,\y-1/2) circle [radius=2mm];
}

\foreach \x/\y in {-2/-1,-1/2,0/1,1/-2,2/-1,2/1,3/0}
{  \draw[outer color=lightgray,inner color=gray]
(\x,\y)--(\x,\y+1)--(\x+1,\y+1)--(\x+1,\y);
\draw [line width = 1mm] (\x,\y-1) rectangle (\x+1,\y+1);
}

\foreach \x/\y in {0/-3,-1/-2,1/0}
{  \draw[outer color=lightgray,inner color=gray]
(\x,\y-1)--(\x-1,\y-1)--(\x-1,\y)--(\x,\y);
\draw [line width = 1mm] (\x-1,\y-1) rectangle (\x+1,\y);
\fill (\x-1/2,\y-1/2) circle [radius=2mm];
}

\foreach \x/\y in {-1/1,0/4,1/3}
{  \draw[outer color=lightgray,inner color=gray]
(\x,\y-1)--(\x+1,\y-1)--(\x+1,\y)--(\x,\y);
\draw [line width=1mm] (\x-1,\y-1) rectangle (\x+1,\y);
}
\end{tikzpicture}

\caption{Particles in a domino tiling.
\label{fig:AztecParticles} }
\end{center}
\end{figure}

We rotate the picture over 45 degrees in clockwise direction
and we change the coordinate system so that 
black squares are identified with the product set
\begin{equation} \label{calBN} 
 \mathcal B_N = \{0, \ldots, 2N \} \times \{0, \ldots, 2N-1 \} 
 \end{equation}
Any possible tiling of the Aztec diamond gives rise to a subset 
$\mathcal X = \mathcal X(\mathcal T)$  of $\mathcal B_N$ containing
the squares that are occupied by a particle.

We use $(m,n) \in \mathcal B_N$ to denote an element $\mathcal B_N$ 
and we will refer to $m$ as the level in $\mathcal B_N$. 
Any $\mathcal X$ that comes from a tiling will have $2N-m$
particles at level $m$ for each $m=0,1, \ldots, 2N$. Therefore the
cardinality is
\[ | \mathcal X | =  N(2N+1). \]
There are also interlacing conditions that are satisfied when
comparing the particles at level $m$ with those at level $m+1$.

The probability measure \eqref{eq:probT} on tilings gives rise to a
probability measure on subsets $\mathcal X$, that turns out to be
determinantal. This means that there exists  a kernel 
\begin{equation} \label{eq:AztecKernel} 
	K_N : \mathcal B_N \times \mathcal B_N \to \mathbb R 
	\end{equation}
with the property that for any subset $S \subset \mathcal B_N$
\[ \mathcal P \left[ S \subset \mathcal X \right] = 
	 \det \left[ K_N(x,y) \right]_{x,y \in S}. \]
This is a discrete determinantal point process \cite{B}.

We found an explicit double contour integral formula for the kernel $K_N$. 
We take $(m,n), (m', n') \in \mathcal B_N$, and instead of 
$K_N((m,n), (m',n'))$ we write $K_N(m,n; m',n')$. We
collect $K_N(m,n;m',n')$ with some of its neighbors in
a $2 \times 2$ matrix 
\begin{equation} \label{eq:KNmatrix}
	\mathbb K_N(m,n; m',n')
	=
	\begin{pmatrix} 
	K_N(m,n; m',n') & K_N(m,n+1; m',n') \\
	K_N(m,n; m', n'+ 1) & K_N(m,n+1; m', n'+1) \end{pmatrix}
\end{equation}
and this matrix  appears in our formula \eqref{eq:theo12}.

\begin{theorem} \label{thm:formula}
Assume $N$ is even and $(m,n) \in \mathcal B_N$, $(m',n') \in \mathcal B_N$ are such that
$m+n$ and $m'+ n'$ are even. Then 
\begin{multline} \label{eq:theo12} 
	\mathbb K_N(m,n; m',n') =
	 - \frac{\chi_{m > m'}}{2\pi i} \oint_{\gamma_{0,1}} A^{m-m'}(z)  z^{(m'+ n')/2 - (m+n)/2} \frac{dz}{z} 
	\\
	+ \frac{1}{(2\pi i)^2} \oint_{\gamma_{0,1}} \frac{dz}{z} 
		\oint_{\gamma_{1}} \frac{dw}{z-w} A^{N-m'}(w) F(w) A^{-N+m}(z) \\
		\times
		\frac{z^{N/2} (z-1)^N}{w^{N/2} (w-1)^N} \frac{w^{(m'+n')/2}}{z^{(m+n)/2}} 
		\end{multline}
	where
\begin{equation} \label{eq:defA}
	A(z) = \frac{1}{z-1} \begin{pmatrix} 2 \alpha z & \alpha (z+1) \\
		\beta z (z+1) & 2 \beta z \end{pmatrix} \end{equation}
and
\begin{equation} \label{eq:defF}
	F(z) = \frac{1}{2} I_2 
		+ \frac{1}{2\sqrt{z(z+\alpha^2)(z+\beta^2)}} \begin{pmatrix} (\alpha-\beta) z & \alpha (z+1) \\
			\beta z (z+1) & - (\alpha-\beta) z \end{pmatrix},
			\end{equation}
where $I_2$ denotes the $2 \times 2$ identity matrix and we use
the principal branch of the square root in \eqref{eq:defF}. The notation 
$\chi$ in \eqref{eq:theo12} denotes the indicator function, $\chi_{m>m'}=1 $ if $m>m'$ and $\chi_{m>m'}=0$ otherwise. 

The contour $\gamma_{0,1}$ in \eqref{eq:theo12} is a simple closed contour going
around $0$ and $1$ in positive direction. The contour $\gamma_1$ is a simple closed
contour in the right half-plane that goes around  $1$ in positive direction, and it
lies in the interior of $\gamma_{0,1}$, see Figure \ref{fig:Contours1}.		
\end{theorem}

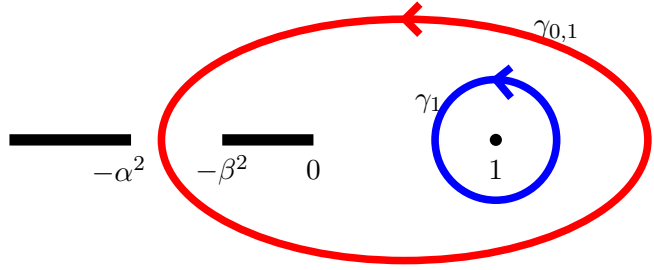
\begin{figure}[t]
\begin{center}
\begin{tikzpicture}[scale=0.8]

\draw [line width=1.5mm] (-6,0)--(-4,0);
\draw (-4.2,-0.5) node {$-\alpha^2$};

\draw [line width= 1.5mm] (-2.5,0)--(-1,0);
\draw (-2.5,-0.5) node {$-\beta^2$};
\draw (-1,-0.5) node {$0$};

\fill (2,0) circle [radius=1mm];
\draw (0.9,0.6) node {$\gamma_1$};
\draw (2,-0.5) node {$1$};
\draw [blue,line width =1mm] (2,0) circle (1cm);
\draw [blue,line width = 1mm] (2.25,0.7)--(2,1)--(2.25,1.2);

\draw (3,1.8) node {$\gamma_{0,1}$};
\draw [red,line width=1mm,] (0.5,0) ellipse (4cm and 2cm);
\draw [red,line width=1mm] (0.75,1.75)--(0.5,2)--(0.75,2.25);
\end{tikzpicture} 
\end{center}
\caption{Contours $\gamma_{0,1}$ and $\gamma_1$ used in
the definition of $\mathbb K_N$ in \eqref{eq:theo12} in
Theorem \ref{thm:formula}.
\label{fig:Contours1}}
\end{figure}

\begin{remark} \label{rem:RS}
The square root factor in \eqref{eq:defF} is defined and analytic for
$z \in \mathbb C \setminus ((-\infty,-\alpha^2] \cup [-\beta^2, 0])$ with
the branch that is positive for real $z > 0$.
This is one sheet of the Riemann surface $\mathcal R$ associated with the cubic
equation 
\begin{equation} \label{eq:RSeq} 
	y^2 =  z (z+\alpha^2) (z + \beta^2), 
	\end{equation}
that will play an important role in what follows. It is a two sheeted 
surface consisting of two sheets
$\mathbb C \setminus ((-\infty,-\alpha^2] \cup [-\beta^2,0])$ glued together
along the two cuts $(-\infty,-\alpha^2]$ and $[-\beta^2,0]$ in the usual
crosswise manner. The surface has genus $1$ unless $\alpha = \beta = 1$ 
in which case the genus drops to $0$.

The matrix valued function $F(z)$ from \eqref{eq:defF} is considered on the first sheet.
Its analytic continuation to the second sheet is given by $I_2 - F(z)$.
\end{remark}

\begin{remark}
The eigenvalues of $A(z)$, see \eqref{eq:defA}, are equal to $\frac{\rho_{1,2}(z)}{z-1}$
where 
	\begin{equation} \label{eq:rho12}
	\rho_{1,2}(z) = (\alpha + \beta) z \pm \sqrt{z(z+\alpha^2)(z+\beta^2)}.
	\end{equation}
are the eigenvalues of  $\begin{pmatrix} 2 \alpha z & \alpha (z+1) \\
		\beta z (z+1) & 2 \beta z \end{pmatrix}$. 
Thus \eqref{eq:rho12} are the two branches of the meromorphic function 
$\rho = (\alpha + \beta) z + y$ on the Riemann surface $\mathcal R$ 
associated with \eqref{eq:RSeq}, with $\rho_{1}$ on the first sheet and $\rho_2$ on the second sheet.

The eigenvectors can also be considered on the Riemann surface.
There is a matrix $E(z)$ whose columns are the eigenvectors such that
\begin{equation} \label{eq:Adecomp}
	A(z) = \frac{1}{z-1} E(z) \begin{pmatrix} \rho_1(z) & 0 \\ 0 & \rho_2(z) \end{pmatrix} E^{-1}(z).
	\end{equation}
See \eqref{eq:defE} below for the precise formula for $E(z)$.
It turns out that (see \eqref{eq:defF} for the definition of $F(z)$)
\begin{equation}  \label{eq:Fdecomp}
  F(z) = E(z) \begin{pmatrix} 1 & 0 \\ 0 & 0 \end{pmatrix} E^{-1}(z).
\end{equation}
Thus $F(z)$ has eigenvalues $0$ and $1$, and $F(z)$ commutes with $A(z)$.
	
We will also work with
\begin{equation} \label{eq:defW}
	W(z) = \frac{A^2(z)}{z} 
	\end{equation}
which in view of \eqref{eq:Adecomp} has eigenvalue decomposition
\begin{equation} \label{eq:Wdecomp} 
	W(z) = E(z) \Lambda(z) E^{-1}(z),
	\qquad \Lambda(z) = \begin{pmatrix} \lambda_1(z) & 0 \\ 0 & \lambda_2(z) \end{pmatrix} \end{equation}
with 
\begin{equation} \label{eq:lambda12}
	\lambda_{1,2}(z) = \frac{\rho_{1,2}^2(z)}{z(z-1)^2}.
	\end{equation}
These eigenvalues are the two branches of a meromorphic function $\lambda$ on the
Riemann surface $\mathcal R$, with $\lambda_{1}$ defined on the first 
sheet and $\lambda_2$ on the second sheet.
\end{remark}

\begin{remark} 
The only singularity for the $w$ integral in \eqref{eq:theo12} 
that is inside
the contour $\gamma_1$ is the pole at $w=1$. Because of \eqref{eq:defA} we see that
$A^{N-m'}(w)$ has a pole of order $N-m'$ (it is a zero if $m' > N$) 
and therefore the integrand in the double integral of \eqref{eq:theo12} has a 
pole of order $N-m'+ N = 2N-m'$ at $w=1$. 
There is no pole if $m'=2N$, and thus it follows that \eqref{eq:theo12} 
vanishes identically for $m'  = 2N$. This is in agreement with the fact 
that there are no particles at level~$2N$.

Level $0$ is full of particles, which means that 
$K_N(0,n; 0,n) = 1$ for all $n = 0, \ldots, 2N-1$. 
This can be seen from formula \eqref{eq:theo12} as follows.
For $n$ even, the formula gives for $\mathbb K_N(0,n;0,n)$,
\begin{multline} \label{eq:KN0n0n}
	 \frac{1}{(2\pi i)^2}
	\oint_{\gamma_{0,1}} \frac{dz}{z} \oint_{\gamma_1} \frac{dw}{z-w}
	A^N(w) F(w) A^{-N}(z) \frac{z^{N/2} (z-1)^N}{w^{N/2} (w-1)^N}
	\frac{w^{n/2}}{z^{n/2}} \end{multline}

In view of \eqref{eq:Fdecomp}, \eqref{eq:defW}, and \eqref{eq:Wdecomp}
we have \begin{align} \label{eq:AFproduct1}
 A^N(w) F(w) & = w^{N/2} F(w) \lambda_1^{N/2}(w), \\
 A^N(w)(I_2-F(w)) & = w^{N/2} (I_2-F(w)) \lambda_2^{N/2}(w).
 \label{eq:AFproduct2}
\end{align}
We will see in Lemma \ref{lem:lambdarho} (b) that $\lambda_1(w)$
has a double pole and $\lambda_2(w)$ has a double zero at $w=1$.
There are no other zeros and poles.  Thus \eqref{eq:AFproduct1}
has a pole of order $N$ at $w=1$, and \eqref{eq:AFproduct2}
has a zero of order $N$ at $w=1$. It follows that
the $w$-integrand in \eqref{eq:KN0n0n} has a pole of
order $2N$ at $w=1$. However, if we replace $F(w)$
by $I_2-F(w)$, then the integrand does not have a pole
at $w=1$  anymore,
and thus by Cauchy's theorem the double integral is zero. 

It follows that the value of the double integral \eqref{eq:KN0n0n}
remains the same if we remove the factor $F(w)$. Then the 
integrand that remains is rational in $w$, with a simple pole at
$w=z$ and a decay $O(w^{-1-N+n/2})$ as $w \to \infty$. 
We apply the residue theorem to the exterior of $\gamma_1$
and the only contribution comes from the pole at $w=z$. 
The $z$-integral that remains in \eqref{eq:KN0n0n} reduces to
\[ \frac{1}{2\pi i} \oint_{\gamma_{0,1}} \frac{dz}{z}  I_2 = I_2
\]
and so by looking at the diagonal entries, see \eqref{eq:KNmatrix},
we get $K_N(0,n;0,n) = K_N(0,n+1;0,n+1) = 1$, as claimed.
\end{remark}

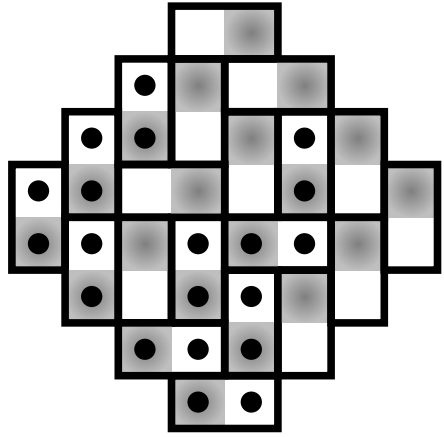
\begin{figure}[t]
\begin{center}
\begin{tikzpicture}[scale=.7]


\foreach \x/\y in {-4/0,-3/1,-3/-1,-2/2,-1/-1,0/-2,1/1}
{ \draw[outer color=lightgray,inner color=gray]
(\x,\y)--(\x,\y-1)--(\x+1,\y-1)--(\x+1,\y); 
\draw [line width = 1mm] (\x,\y-1) rectangle (\x+1,\y+1);
\fill (\x+1/2,\y-1/2) circle [radius=2mm];
\fill (\x+1/2,\y+1/2) circle [radius=2mm];
}

\foreach \x/\y in {-2/-1,-1/2,0/1,1/-2,2/-1,2/1,3/0}
{  \draw[outer color=lightgray,inner color=gray]
(\x,\y)--(\x,\y+1)--(\x+1,\y+1)--(\x+1,\y);
\draw [line width = 1mm] (\x,\y-1) rectangle (\x+1,\y+1);
}

\foreach \x/\y in {0/-3,-1/-2,1/0}
{  \draw[outer color=lightgray,inner color=gray]
(\x,\y-1)--(\x-1,\y-1)--(\x-1,\y)--(\x,\y);
\draw [line width = 1mm] (\x-1,\y-1) rectangle (\x+1,\y);
\fill (\x-1/2,\y-1/2) circle [radius=2mm];
\fill (\x+1/2,\y-1/2) circle [radius=2mm];
}

\foreach \x/\y in {-1/1,0/4,1/3}
{  \draw[outer color=lightgray,inner color=gray]
(\x,\y-1)--(\x+1,\y-1)--(\x+1,\y)--(\x,\y);
\draw [line width=1mm] (\x-1,\y-1) rectangle (\x+1,\y);
}
\end{tikzpicture}

\caption{Extended particle system that is equivalent to a domino tiling
of the Aztec diamond.
\label{fig:AztecParticlesExtended} }
\end{center}
\end{figure}

\begin{remark}
Observe that the particle system on the black squares is not 
equivalent to the tiling of the Aztec diamond. However, it can 
be extended to a larger system where we also also put a particle 
in the white square of the West and South dominos, as shown in
Figure \ref{fig:AztecParticlesExtended}. This extended
process is equivalent to the tiling. It is also determinantal and 
we will now give a double contour integral formula for the kernel.  
We  chose to 
work with the particle system determined by the black squares 
only since the formulas and their asymptotics  stated in the 
next paragraphs take an easier form.  The analogous asymptotics 
for the extended processes are straightforward from these results 
and the formula \eqref{eq:KNextended4} below.

So in addition to the set of black squares \eqref{calBN} 
we also consider the white squares
\begin{equation} \label{calWN} 
	\mathcal W_N = \{\tfrac{1}{2}, \ldots, 2N-\tfrac{1}{2} \} \times
	\{-\tfrac{1}{2}, \tfrac{1}{2}, \ldots, 2N-\tfrac{1}{2} \}. 
	\end{equation}
Then the extended particle system has a correlation kernel
(that we continue to denote by $K_N$)
\begin{equation} \label{eq:KNextended1} 
	K_N : \left(\mathcal B_N \cup \mathcal W_N \right)
	\times \left(\mathcal B_N \cup \mathcal W_N \right)
 	\to \mathbb R. \end{equation}
To write down the extended kernel we use the notation that will
also be used later in \eqref{eq:Ammh} and \eqref{eq:Amhmp},
namely, for an integer $m \in \mathbb Z$,
\begin{equation}  \label{eq:KNextended2} 
A_{m,m+1/2}(z) = \begin{pmatrix} \alpha & \alpha \\ \beta z & \beta \end{pmatrix},
	\quad
	A_{m+1/2,m+1}(z) = \frac{1}{z-1} \begin{pmatrix} z & 1 \\ z & z \end{pmatrix},
	\end{equation}
and for $m, m' \in \frac{1}{2} \mathbb Z$, 
\begin{equation} \label{eq:KNextended3} 
	A_{m,m'}(z) =  A_{m,m+1/2}(z) A_{m+1/2,m+1}(z) \cdots 
	A_{m'-1/2,m'}(z), \quad  \text{ if } m < m', 
	\end{equation}
$ A_{m,m'}(z) = A_{m',m}(z)^{-1}$ if $m > m'$, and $A_{m,m}(z) = I_2$. 
	
Then the extended kernel \eqref{eq:KNextended1} is as follows.
We assume $N$ is even $(m,n) \in \mathcal B_N \cup \mathcal W_N$
and $(m',n') \in \mathcal B_N \cup \mathcal W_N$ with both
$m+n$ and $m'+n'$ even. Then
\begin{multline} \label{eq:KNextended4} 
	\begin{pmatrix} 
	K_N(m,n; m',n') & K_N(m,n+1; m',n') \\
	K_N(m,n; m', n'+ 1) & K_N(m,n+1; m', n'+1) \end{pmatrix}
	\\ =
	 - \frac{\chi_{m > m'}}{2\pi i} \oint_{\gamma_{0,1}} A_{m',m}(z)  
	 z^{(m'+ n')/2 - (m+n)/2} \frac{dz}{z} 
	\\
	+ \frac{1}{(2\pi i)^2} \oint_{\gamma_{0,1}} \frac{dz}{z} 
		\oint_{\gamma_{1}} \frac{dw}{z-w} A_{m',N}(w) F(w) A_{N,m}(z) \\
		\times
		\frac{z^{N/2} (z-1)^N}{w^{N/2} (w-1)^N} \frac{w^{(m'+n')/2}}{z^{(m+n)/2}}. 
		\end{multline}
Observe that by \eqref{eq:defA} and \eqref{eq:KNextended1} 
\[ A_{m,m+1/2}(z) A_{m+1/2,m+1}(z) = A(z), \qquad m \in \mathbb Z, \]
and so by the definition \eqref{eq:KNextended3}, we have
$A_{m',m}(z) = A^{m-m'}(z)$, $A_{m',N}(w) = A^{N-m'}(w)$,
$A_{N,m}(z) = A^{-N+m}(z)$, if $m,m'$ are integers.
In this way we recover \eqref{eq:theo12} from \eqref{eq:KNextended4}.
\end{remark}

\subsection{Matrix valued orthogonal polynomials}
The starting point of our approach to Theorem \ref{thm:formula} 
is the non-intersecting path reformulation of the Aztec diamond
and the Lindstr\"om-Gessel-Viennot lemma. This will be developed
in Section \ref{sec:paths}.
The novel ingredient in the further analysis is the use of matrix
valued orthogonal polynomials (MVOP).

A matrix valued polynomial of degree $k$ and size $d$ is a function
\[ P(z) = C_0 z^k + C_{1} z^{k-1} + \cdots + C_k \]
where $C_0, \ldots, C_k$ are matrices of size $d \times d$.
Suppose $w(z)$ is a $d\times d$ weight matrix on a set $\Gamma$ in
the complex plane. 

\begin{definition}
Suppose $P_N$ is a matrix valued polynomial of degree $N$ 
with an invertible leading coefficient. 
Then $P_N$ is a matrix valued orthogonal polynomial with
weight matrix $w$ on $\Gamma$ if 
\begin{equation} \label{eq:MVOPdef} 
	\int_{\Gamma} P_N(z) w(z) Q^t(z) dz = 0_d, 
	\end{equation}
for all matrix polynomials $Q$ of degree $\leq N-1$, where $Q^t$ denotes the matrix transpose.
\end{definition}
The integral in \eqref{eq:MVOPdef} 
is to be taken entrywise, and $0_d$ denotes the $d \times d$ zero matrix. 
We note that the order of the factors in the integrand in 
\eqref{eq:MVOPdef} is important since we are dealing with matrices.

For us the weight matrix will be $\frac{1}{2\pi i} W^N(z)$ 
on  the closed contour  $\gamma_{1}$ around $1$. 
Thus $d=2$, and $W^N$ denotes the $N$th power of  $W$,
so that  the weight matrix is varying with $N$. 
Recall that $W$ is defined in \eqref{eq:defW}, and explicitly we have
\begin{equation} \label{eq:Wexplicit} 
	W(z) = \frac{1}{(z-1)^2} \begin{pmatrix} (z+1)^2 + 4 \alpha^2 z & 
	2\alpha (\alpha+\beta)(z+1) \\
		2\beta (\alpha+\beta) z(z+1) & (z+1)^2 + 4 \beta^2 z \end{pmatrix}. 
	\end{equation}
The existing literature on MVOP mostly deals with the case of
orthogonality on an interval of the real line, 
with a positive definite weight matrix $w$ with all existing moments.
In such a case the MVOP  exists for every degree $n$,
and they can be normalized in such a way that 
\begin{equation} \label{eq:MVOPdef2}
	\int_{\Gamma}	P_n(z) w(z) P_n^t(z) dz = I_d, 
	\qquad j=0, 1, \ldots, n-1.
	\end{equation}
However, it is interesting to note that MVOP first appeared
in connection with prediction theory where the orthogonality 
is on the unit circle, see \cite{Bing} for a recent survey.
The interest in MVOP on the real line has been steadily
growing since the early 1990s.  
The analytic theory of MVOP on the real line is surveyed in 
\cite{DPS} with \cite{AN} as one of the pioneering works.
MVOP satisfy recurrence relations \cite{DvA} and special cases 
satisfy differential equations \cite{DG}.  Interesting 
examples of 
MVOP come from  matrix valued spherical functions, see  \cite{GPT,KvPR} as well as many other papers.
	
We deviate from the usual set-up of MVOP in several ways
\begin{itemize}
\item $\Gamma$ is a closed contour in the complex plane,
\item the weight matrix $w(z) = \frac{1}{2\pi i} W^N(z)$
is complex-valued on $\Gamma$ and varies with $N$,
\item the weight matrix is not symmetric or Hermitian (let
alone positive definite), or have any other property that 
would imply existence and uniqueness of the MVOP.
\end{itemize}
Since there is no complex conjugation in \eqref{eq:MVOPdef},
we are thus dealing with non-Hermitian matrix valued orthogonality  
with varying weights on a closed contour in the plane.

As already noted, existence and uniqueness of the MVOP are not 
guaranteed in this general setting. However, for the weight 
$\frac{1}{2\pi i} W^N$ we can show that the 
monic MVOP up to degrees $N$ all exist and are unique. 
However, since the weight matrix is not symmetric, we cannot
normalize to obtain orthonormal MVOP as in \eqref{eq:MVOPdef2}.
In our case the MVOP of degrees $> N$ do not exist.

In Section \ref{sec:MVOP} we consider a situation that is
more general than the two periodic Aztec diamond. It deals with
a multi-level particle system that is determinantal, and
transitions between the levels are periodic. See Assumptions
\ref{ass:multilevel} and \ref{ass:blocksymbols} 
for the precise assumptions. In this general setting we make 
a connection with matrix valued (bi)orthogonal polynomials
and our main result in Section \ref{sec:MVOP} 
is Theorem \ref{thm:Kernelblock}
that expresses the correlation kernel 
as a double contour integral containing a reproducing
kernel for the matrix polynomials.

In the special situation of the two periodic Aztec diamond it
 gives that the matrix 
 $\mathbb K_N$  with the correlation kernels as 
 in \eqref{eq:theo12} is given by 
\begin{multline} \label{eq:AztecKernelwithCDSum}  
	- \frac{\chi_{m > m'}}{2\pi i} \oint_{\gamma_{0,1}} A^{m-m'}(z)  
	z^{(m'+ n')/2 - (m+n)/2} \frac{dz}{z} 
	\\
	+ \frac{1}{(2\pi i)^2} \oint_{\gamma_{0,1}} \frac{dz}{z} 
		\oint_{\gamma_{0,1}} \frac{dw}{z-w} 
		A^{2N-m'}(w) R_N(w,z)  
		A^{m}(z) \frac{w^{(m'+n')/2}}{z^{(m+n)/2} w^N}. 
		\end{multline}
where $R_N(w,z)$ is the reproducing kernel associated with
the matrix polynomials of degrees $\leq N-1$. That is,
$R_N(w,z)$ is a bivariate matrix valued polynomial of 
degree $N-1$ in both $w$ and $z$, such that
\[ \frac{1}{2\pi i} \oint_{\gamma_1}
	R_N(w,z) W^N(z) Q^t(z) dz = Q^t(w) \]
holds for every matrix valued polynomial $Q$ of degree $\leq N-1$.
	
The MVOP of degree $N$ is characterized by a Riemann-Hilbert problem 
and the reproducing kernel $R_N(w,z)$ can be expressed in terms
of the solution of the Riemann-Hilbert problem. 
This is known from work of Delvaux \cite{D} and  
we recall it in Section \ref{subsec:RHP}. Then we perform 
an analysis of the Riemann-Hilbert problem, and 
quite remarkably this produces the exact formula \eqref{eq:theo12}.

\subsection{Classification of phases}
\label{subsec:phases}

The explicit formula \eqref{eq:theo12} in Theorem \ref{thm:formula}
is suitable for asymptotic analysis as $N \to \infty$. See Figure \ref{fig:sample} for a sampling of a large 2-periodic Aztec diamond. In this figure  three  regions emerge where the tiling  appears to have different statistical behavior.
We first describe how we can distinguish these three phases (solid, liquid and gas) 
in the model. This classification will depend on the location of saddle points for 
the double integral in \eqref{eq:theo12}.

\begin{figure}[t]
\begin{center}
\includegraphics[angle=135,scale=.5]{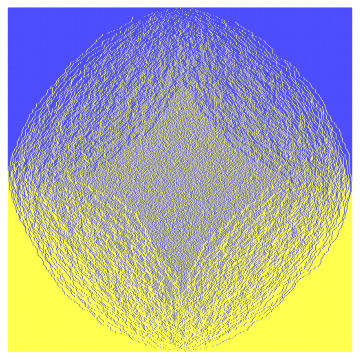}
\end{center}
\caption{A sample of a two-periodic Aztec diamond of size $500$. We colored the West and South dominos blue. The East and North dominons are colored yellow.
The gas phase is visible in the middle. 
This figure is generated by a code that was kindly provided to us by Sunil Chhita.}\label{fig:sample}
\end{figure}

We fix coordinates $-1 < \xi_1 < 1$ and $-1 < \xi_2 < 1$
and choose $ m, m' \approx (1+\xi_1)N$ and $n,n'\approx (1+\xi_2) N$.
Then from the formula \eqref{eq:theo12}, we see that the $z$ integral of
the double contour integral is dominated as $N \to \infty$ by the expression 
\[  A(z)^{\xi_1 N} z^{N/2} (z-1)^N  z^{-(1 + \frac{\xi_1}{2} + \frac{\xi_2}{2})N}
	= W^{\xi_1 N/2}(z)
		(z-1)^N z^{-(1 + \xi_2) N/2}
	\]
where $W$ is given by \eqref{eq:defW}. 
In view of the eigenvalue decomposition \eqref{eq:Wdecomp} this is
\[ E(z) \begin{pmatrix} e^{N \Phi_1(z)/2} & 0 \\ 0 & e^{N \Phi_2(z)/2} \end{pmatrix}
	E(z)^{-1} \]
with $\Phi_{j}(z) = 2 \log (z-1) - (1+ \xi_2) \log z + \xi_1 \log \lambda_j(z)$,
$j=1,2$. Hence  we are led to consider
\begin{equation} \label{eq:defPhi} 
	\Phi(z) = 2 \log(z-1) - (1+\xi_2) \log z + \xi_1 \log \lambda (z) 
	\end{equation}
as a function on the Riemann surface $\mathcal R$, depending on parameters
$\xi_1$ and $\xi_2$. It is multi-valued, but its differential
\begin{equation} \label{eq:MeroDiff} 
	\Phi'(z) dz = \left( \frac{2}{z-1} -  \frac{1+\xi_2}{z} + 
	\xi_1 \frac{\lambda'(z)}{\lambda(z)} \right) dz 
\end{equation}
is a single-valued meromorphic differential with simple poles at 
$z = 1^{(1)}$, $z = 1^{(2)}$, $z = 0$ and $z=\infty$, see also 
Section \ref{subsec:saddles}.
(For $j=1,2$, we use $1^{(j)}$ to denote the value $1$ 
on the $j$th sheet of the Riemann surface.)
There are also four zeros, counting multiplicities, since the genus is $1$.

\begin{definition} \label{def:saddles}
The saddle points are the zeros of $\Phi'(z) dz$.
\end{definition}

The real part $\mathcal R_r$ of the Riemann surface consists of all real tuples $(z,y)$ satisfying the algebraic equation \eqref{eq:RSeq} together with 
the point at infinity. The real part is the union of two 
cycles,
\begin{equation} \label{eq:RSreal} 
	\mathcal R_r = \mathcal C_1 \cup \mathcal C_2 
	\end{equation}
where $\mathcal C_1$ is the union of the intervals $[-\alpha^2, -\beta^2]$ 
on the two sheets,  and $\mathcal C_2$ is the union of
the two intervals $[0,\infty]$ on both sheets.

It turns out that there are always at least two distinct saddle 
points on the cycle $\mathcal C_1$, see Proposition \ref{prop:saddleC1}
below.  The location of the other two saddle points determines the phase.

\begin{definition} \label{def:phases}
Let $-1 < \xi_1, \xi_2 < 1$.
\begin{itemize}
\item[(a)] If two simple saddles are in $\mathcal C_2$, then
$(\xi_1,\xi_2)$ is in the \textbf{solid phase}, and we write 
$(\xi_1, \xi_2) \in \mathfrak{S}$.
\item[(b)] If two saddles are outside of the real part of 
the Riemann surface, then $(\xi_1, \xi_2)$ is in the \textbf{liquid phase}, and
we write $(\xi_1, \xi_2) \in \mathfrak{L}$ .
\item[(c)] If all four saddles are simple and belong to 
$\mathcal C_1$, then $(\xi_1, \xi_2)$ is in the 
\textbf{gas phase}, and  we write
$(\xi_1, \xi_2) \in \mathfrak{G}$.
\end{itemize}

Transitions between phases take place when two or more saddle points coalesce.
\begin{itemize}
\item[(d)] If there is a double saddle point on $\mathcal C_2$, then 
$(\xi_1, \xi_2)$ is on the \textbf{solid-liquid transition}.
\item[(e)] If there is a double or triple saddle point on 
$\mathcal C_1$ then $(\xi_1,\xi_2)$ is on the \textbf{liquid-gas transition}. 
\end{itemize}
\end{definition}
It is not possible to have a double saddle point outside of
the real part of the Riemann surface.

\begin{figure}[t]
\begin{center}
\includegraphics[scale=0.25]{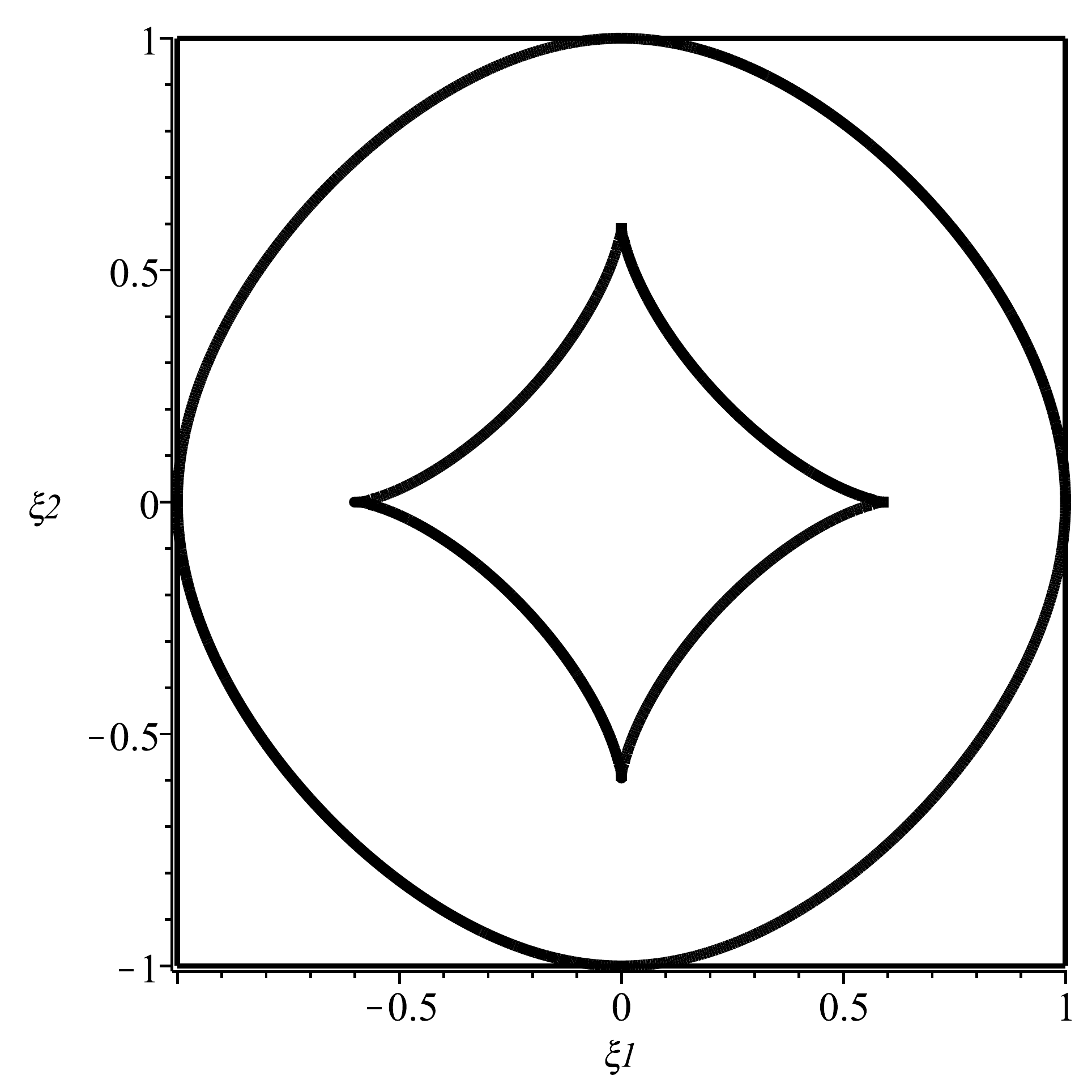} \qquad
\includegraphics[scale=0.25]{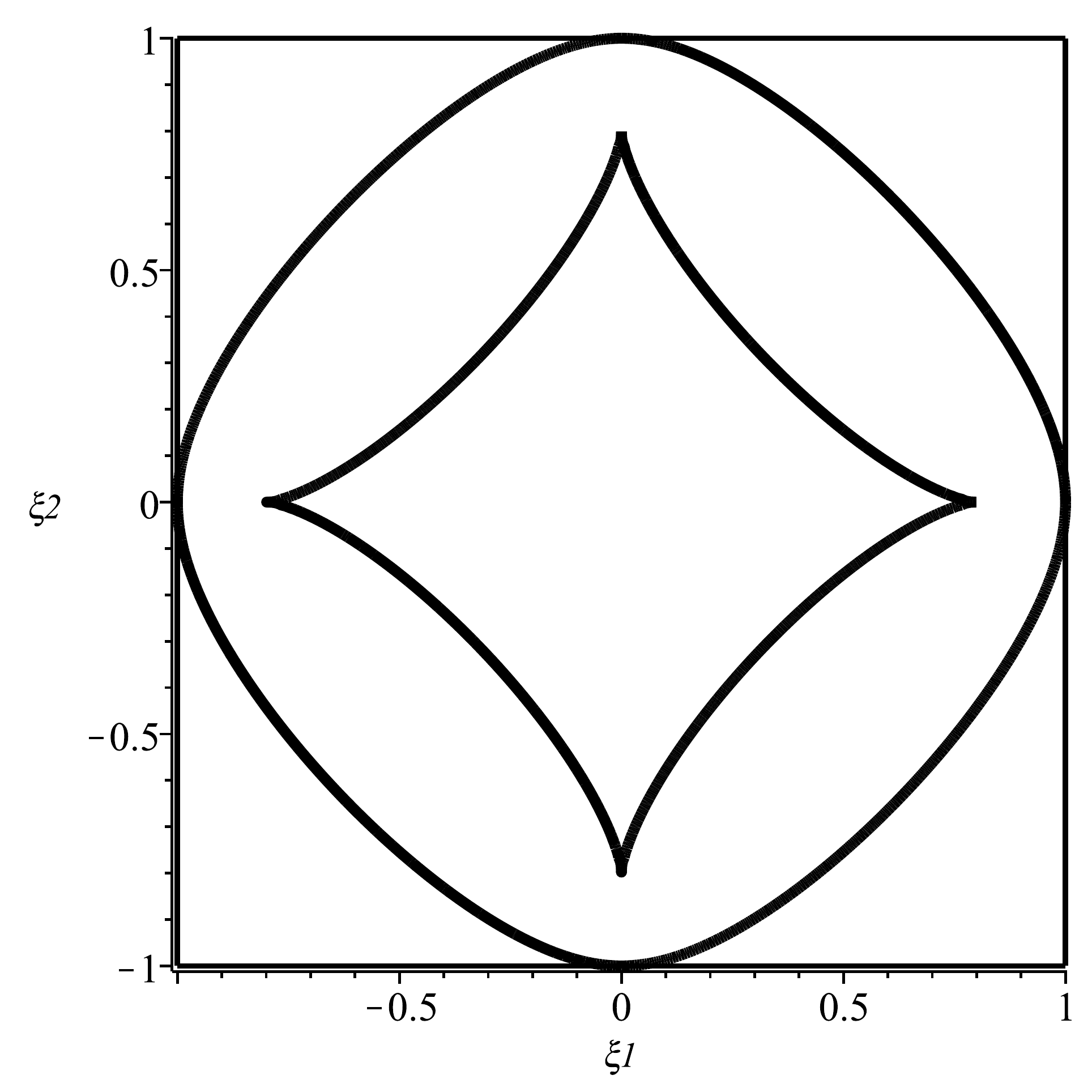}
\caption{Real section of the degree 8 algebraic curve 
in $\xi_1$-$\xi_2$ plane for the cases $\alpha=2$ (left) and 
$\alpha =3$ (right). The outer
component is the boundary between the solid and liquid phases
and the inner component is the boundary between the liquid and
gas phases.
\label{fig:AlgCurve}}
\end{center}
\end{figure} 

The condition for coalescing saddle points leads to an algebraic equation
of degree $8$ for $\xi_1, \xi_2$ and it precisely coincides with the 
equation listed in the appendix of \cite{CJ}, see also 
\eqref{eq:AlgEq} below.

The real section of the degree 8 algebraic equation has two components
in case $\alpha > 1$, as shown in Figure \ref{fig:AlgCurve}. Both components
are contained in the square $-1 \leq \xi_1, \xi_2 \leq 1$.
The outer component is a smooth closed curve that touches the square in
the points $(\pm 1,0)$, and $(0,\pm 1)$. It is the boundary between
the solid and liquid phases.

The inner component is the boundary between the liquid and gas phases.
It is a closed curve with four cusps at locations
$\left(\pm \tfrac{\alpha-\beta}{\alpha+\beta}, 0\right)$, and 
$\left(0,\pm \tfrac{\alpha-\beta}{\alpha+\beta}\right)$.

It can indeed be checked that for $\xi_1=0$, the eight solutions
of the degree $8$ equation for $\xi_2$ are explicit, 
namely $\xi_2 = \pm 1$ with multiplicity $1$ and 
$\xi_2 = \pm \frac{\alpha-\beta}{\alpha+\beta}$, both with multiplicity $3$. 

The intersections of the algebraic curve with the diagonal lines $\xi_2 = \pm \xi_1$
are explicit as well, namely $(\xi_1,\xi_2) = \left( \pm  \tfrac{\alpha+1}{2\sqrt{\alpha^2+1}},
		\pm  \tfrac{\alpha+1}{2\sqrt{\alpha^2+1}} \right)$
are on the outer component, and
$(\xi_1,\xi_2) = \left( \pm  \tfrac{\alpha-1}{2\sqrt{\alpha^2+1}},
		\pm  \tfrac{\alpha-1}{2\sqrt{\alpha^2+1}} \right)$		
are on the inner component.

\subsection{Gas phase}

Our next result gives the limit of $\mathbb K_N$ in the gas phase.
Recall that $\mathbb K_N$ is defined by \eqref{eq:KNmatrix}.

\begin{theorem} \label{thm:gaslimit}
Assume $(\xi_1, \xi_2) \in \mathfrak{G}$. Suppose $m, m', n, n'$ are
integers that vary with $N$ in such a way that 
\begin{align} \label{eq:mnscaling}
	m = (1+\xi_1) N + o(N), \quad n = (1+\xi_2)N + o(N), 
	\end{align}
as $N \to \infty$, while 
\begin{align} \label{eq:mndiff}
	m' - m = \Delta m, \qquad n' - n = \Delta n
	\end{align}
are fixed. Also assume that $m+n$ and $m' + n'$ are even.
Then we have for $N$ even,
\begin{equation} \label{eq:gaslimit}
	\lim_{N \to \infty} \mathbb K_N(m,n; m', n') 
	=	 \mathbb K_{gas}(m,n; m',n') 
	\end{equation}
with
\begin{equation} \label{eq:gaskernel}
	 \mathbb K_{gas}(m,n;m',n') =
	\frac{1}{2\pi i} \oint_{\gamma}   
	  \left(F(z) - \chi_{\Delta m < 0} I_2 \right) A^{-\Delta m}(z) z^{(\Delta m + \Delta n)/2} 
	   \frac{dz}{z} 
\end{equation}
where $\gamma$ is a closed contour in 
$\mathbb C \setminus((-\infty,-\alpha^2] \cup [-\beta^2,0])$ going around the interval $[-\beta^2,0]$
(which we can also view as a closed loop on the first sheet of the Riemann surface).
\end{theorem}
The limit \eqref{eq:gaskernel} does not depend on $\xi_1$ and $\xi_2$ as long as these are in the gas region. 
The kernel $\mathbb K_{gas}$ still depends on the parameters $\alpha$ and $\beta$ 
(cf.\ Remark \ref{rem:diagonal}) and therefore the 2-periodic structure is still present in this limit. 

The proof of Theorem \ref{thm:gaslimit} is in Section \ref{subsec:gaslimit}.

\begin{remark} \label{rem:decay}
In \eqref{eq:gaskernel} we can see the exponential decay of correlations
that is characteristic for the gas phase as follows.
We combine the factor $z^{(\Delta m)/2}$ with $A^{-\Delta m}(z)$, since by
\eqref{eq:defW}
\[ z^{(\Delta m)/2} A^{-\Delta m}(z) = W^{-(\Delta m)/2}(z). \]
We  find from this and \eqref{eq:defF} and \eqref{eq:defW} that
\begin{align*} 
	F(z) z^{(\Delta m)/2} A^{-\Delta m}(z) 
	& = F(z) \lambda_1^{-(\Delta m)/2}(z)
	\end{align*} 
and
\begin{align*}
	(F(z) - I_2) z^{(\Delta m)/2} A^{-\Delta m}(z) 
	& 	= (F(z) - I_2) \lambda_2^{-(\Delta m)/2}(z) \\
	& = (F(z) - I_2) \lambda_1^{(\Delta m)/2}(z) 
	\end{align*}
since $\lambda_2 = \lambda_1^{-1}$, see part 
(d) of Lemma \ref{lem:lambdarho} below.
Thus  \eqref{eq:gaskernel} can be written as 
\begin{equation} \label{eq:gaskernel2}
	\mathbb K_{gas}(m,n;m',n') =
		\begin{cases} \ds	\frac{1}{2\pi i} \oint_{\gamma}   
	  F(z) \lambda_1^{-(\Delta m)/2}(z) z^{(\Delta n)/2} 
	   \frac{dz}{z}, & \text{ if } \Delta m \geq 0,  \\[10pt]
	\ds \frac{1}{2\pi i} \oint_{\gamma}   
	  (F(z) - I_2) \lambda_1^{(\Delta m)/2}(z) z^{(\Delta n)/2} 
	   \frac{dz}{z}, & \text{ if } \Delta m < 0.   
	  \end{cases} \end{equation}

By analyticity and Cauchy's theorem, we have the freedom to deform 
the contour $\gamma$ as long as it goes
around the interval $[-\beta^2,0]$ and does not intersect $(-\infty,-\alpha^2]$.
Since $\beta < 1 < \alpha$ we can deform it to a circle centered at zero
of radius $<1$ or to a circle of radius $>1$.
If $\Delta n \geq 0$ we deform to a circle $|z| = r < 1$ and if
$\Delta n < 0$ we deform to a circle $|z| = r > 1$.
In both cases the factor $z^{(\Delta n)/2}$ is exponentially small as
$|\Delta n| \to +\infty$.	

Since $|\lambda_1| > 1$ on $\gamma$ (as will follow from 
parts (d) and (f) of Lemma \ref{lem:lambdarho} below) the
factor $\lambda_1^{\pm (\Delta m)/2}(z)$ is also exponentially small as $|\Delta m| \to +\infty$.
It follows that the gas kernel \eqref{eq:gaskernel2}
decays exponentially as $|\Delta m| + |\Delta n| \to \infty$.
\end{remark}

\begin{remark} \label{rem:diagonal}
Let's see what we have for the gas kernel \eqref{eq:gaskernel}
on the diagonal, i.e., for $\Delta m = \Delta n = 0$.
then we obtain from \eqref{eq:gaskernel} and \eqref{eq:defF}
if $m+n$ is even
\begin{multline} \label{eq:Kgasdiag} \mathbb K_{gas}(m,n; m,n) =
 \frac{1}{2\pi i} \oint_{\gamma} F(z) \frac{dz}{z} \\
 = \frac{1}{2} I_2 + \frac{1}{4\pi i}
	\oint_{\gamma} \begin{pmatrix} (\alpha-\beta) z  & \alpha(z+1) \\
		\beta z (z+1) & - (\alpha-\beta) z \end{pmatrix}
	 \frac{dz}{z \sqrt{z(z+\alpha^2)(z+\beta^2)}}. 
	 \end{multline}
	 
The diagonal entries of \eqref{eq:Kgasdiag} are $K_{gas}(m,n; m,n)$
and $K_{gas}(m,n+1; m, n+1)$, see also \eqref{eq:KNmatrix}. Thus for
arbitary parity of $m+n$, 
\begin{align} \nonumber
K_{gas}(m,n;m,n) & =  
	\frac{1}{2}  + (-1)^{m+n} \frac{\alpha-\beta}{4\pi i} \oint_{\gamma}
	\frac{dz}{\sqrt{z(z+\alpha^2)(z+\beta^2)}} \\
	& = \frac{1}{2} + (-1)^{m+n} \frac{\alpha-\beta}{2\pi} \int_0^{\beta^2}
	\frac{dt}{\sqrt{t(\beta^2-t)(\alpha^2-t)}}. \label{eq:Kgasdiag2}
	 \end{align}
where \eqref{eq:Kgasdiag2} follows from deforming the contour $\gamma$ to 
the interval $[-\beta^2,0]$ and making a change of variables. 
It is easy  to check that
\[ 0 < \frac{\alpha-\beta}{2\pi} \int_0^{\beta^2}
	\frac{dt}{\sqrt{t(\beta^2-t)(\alpha^2-t)}} < \frac{1}{2} \]
	and so \eqref{eq:Kgasdiag2} is between
	$0$ and $1$, as it is the particle density of the gas phase. Note also that the density depends on the parameters $\alpha$ and $\beta$. 
\end{remark}

\subsection{Cusp points} \label{subsec:cusps}

On the boundary between the gas phase and the liquid phase,
the gas kernel \eqref{eq:gaskernel} is still the dominant
contribution. This phenomenon was already observed by Chhita and Johansson \cite{CJ} and further investigated by Beffara, Chhita and Johansson \cite{BCJ}, who looked
at the diagonal point $\xi_1 = \xi_2$ on the boundary and
proved that after averaging there is an Airy like behavior
in the first subleading term.

We consider the cusp points, and show explicitly the
appearance of a Pearcey like behavior in the subleading term
of the kernel $\mathbb K_N$.

The four cusp points are located at positions 
$(\xi_1,\xi_2) = \left(\pm \tfrac{\alpha-\beta}{\alpha+\beta},0\right)$ 
and $(\xi_1,\xi_2) = \left(0,\pm \tfrac{\alpha-\beta}{\alpha+\beta}\right)$
in the phase diagram.
We focus on the top cusp point with coordinates 
$(\xi_1^*, \xi_2^*) = \left(0,\tfrac{\alpha-\beta}{\alpha+\beta}\right)$.
At this cusp point the triple saddle point is located at the
branch point $-\alpha^2$.

\begin{theorem} \label{thm:cusplimit} 
Suppose $N$ and $m+n$ and $m'+n'$ are even. Write
$m = (1+\xi_1)N$, $n = (1+\xi_2)N$, $m' = (1+\xi_1')N$, $n' = (1+\xi_2')N$
and assume 
\begin{equation} \label{eq:scaling1}
\begin{aligned}
	 N^{3/4}\xi_1 & \to c_1 u, & 
	 N^{1/2}( \xi_2 - \xi_2^*) & \to c_2 v,
	& \xi_2^* & = \frac{\alpha-\beta}{\alpha+\beta}, \\
	N^{3/4} \xi_1' & \to c_1 u', & 
	N^{1/2} (\xi_2' - \xi_2^*) & \to c_2 v',
	\end{aligned}
	\end{equation}
	as $N \to \infty$, with fixed $u, u', v, v'$ and
with constants
\begin{equation} \label{eq:constants1}
c_1 = \frac{2^{1/4}}{\sqrt{\alpha-\beta}}, \quad c_2 = \frac{\sqrt{2}}{\alpha+\beta}.
\end{equation}
Then, in case $m$ is even, 	
\begin{multline} \label{eq:PearceyLimit1}
\lim_{N \to \infty} N^{1/4}  (-1)^{(\Delta n- \Delta m)/2} \alpha^{-\Delta n}  
	\left( \mathbb K_N(m,n;m',n') - \mathbb K_{gas}(m,n;m',n') \right)
	\\ = 
	\frac{\sqrt{\alpha-\beta}}{2^{1/4}} \begin{pmatrix} 1 & 1 \\ -1 & -1 \end{pmatrix}
\frac{1}{(2\pi i)^2} \int_{\Sigma \cup (-\Sigma)} \int_{i\mathbb R}
	\frac{e^{\frac{1}{4} s^4 + \frac{1}{2} v s^2 + us}}
	{e^{\frac{1}{4} t^4 + \frac{1}{2} v' t^2 + u't}} \frac{dsdt}{t-s}
	\end{multline}   
with contours $\Sigma$ and $-\Sigma$ as shown in Figure \ref{fig:PearceyIntegrals}.

In case $m$ is odd, we have
\begin{multline} \label{eq:PearceyLimit2}
\lim_{N \to \infty} N^{1/4}  (-1)^{(\Delta n- \Delta m)/2} \alpha^{-\Delta n}  
	\left( \mathbb K_N(m,n;m',n') - \mathbb K_{gas}(m,n;m',n') \right)
	\\ = 
	\frac{\sqrt{\alpha-\beta}}{2^{1/4}} \begin{pmatrix} 1 & 1 \\ -1 & -1 \end{pmatrix}
\frac{1}{(2\pi i)^2} \int_{\Sigma \cup (-\Sigma)^{-1}} \int_{i\mathbb R}
	\frac{e^{\frac{1}{4} s^4 + \frac{1}{2} v s^2 + us}}
	{e^{\frac{1}{4} t^4 + \frac{1}{2} v' t^2 + u't}} \frac{dsdt}{t-s}
	\end{multline}  
where $(-\Sigma)^{-1}$ indicates that the orientation on $(-\Sigma)$
is reversed. 
\end{theorem}

\begin{figure}[t]
	\begin{center}
		\begin{tikzpicture}
		\draw[help lines] (-2,0)--(2,0);
		\draw[line width=1mm,->,blue] (0,-2)--(0,1);
		\draw[line width=1mm,-,blue] (0,1)--(0,2);
		\draw[line width=1mm,red,->] (2,-2)--(1,-1);
		\draw[line width=1mm, red,->] (1,-1)--(.15,-.15)--(.15,.15)--(1.2,1.2);
		\draw[line width=1mm, red,-] (1,1)--(2,2);
		\draw[line width=1mm,red,->] (-2,2)--(-1,1);
		\draw[line width=1mm, red,->] (-1,1)--(-.15,.15)--(-.15,-.15)--(-1.2,-1.2);
		\draw[line width=1mm, red,-] (-1,-1)--(-2,-2);	
		\draw (1.5,1) node {$\Sigma$};
		\draw (-1.7,-1.2) node {$-\Sigma$};
		\draw (-0.25,1.7) node {$i \mathbb R$};	
		\end{tikzpicture}
		\caption{The contours of integration for the Pearcey integrals in \eqref{eq:PearceyLimit1} and \eqref{eq:PearceyLimit2}.
		\label{fig:PearceyIntegrals}} 
	\end{center}
\end{figure}
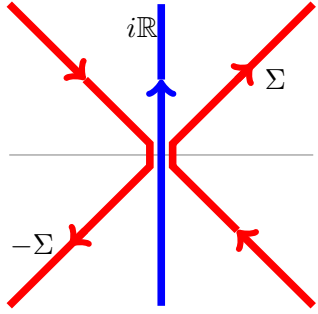

The proof of Theorem \ref{thm:cusplimit} is in Section \ref{subsec:cusplimit}.

The double integral in \eqref{eq:PearceyLimit1}
\begin{equation} \label{eq:Pearcey} \frac{1}{(2\pi i)^2} \int_{\Sigma \cup (-\Sigma)} \int_{i\mathbb R}
	\frac{e^{\frac{1}{4} s^4 + \frac{1}{2} v s^2 + us}}
	{e^{\frac{1}{4} t^4 + \frac{1}{2} v' t^2 + u't}} \frac{dsdt}{t-s} 
\end{equation}
is, up to a gaussian, known as the Pearcey kernel. It is one of the canonical kernels from
random matrix theory that arises typically as a scaling limit near a cusp
point. It was first described by Br\'ezin and Hikami \cite{BH} in
the context of random matrices with an external source, see also \cite{BK2}.
The Pearcey process was given in  \cite{OR2,TW}.
More recent contributions are for example \cite{ACvM1,BC,HHN}. Note that the actual Pearcey kernel includes a gaussian in addition to the double integral in \eqref{eq:Pearcey}. Remarkably, this term is hidden in the gas kernel and can be retrieved by a steepest descent analysis of that kernel.

\begin{theorem}\label{thm:gaussianlimit} 
Under the same assumptions as in Theorem \ref{thm:cusplimit} we have 
\begin{multline} \label{eq:gaussianlimit}
\lim_{N \to \infty} N^{1/4}(-1)^{(\Delta n -|\Delta m|)/2} 
\alpha^{-\Delta n}  \mathbb K_{gas}(m,n;m',n')\\
= 
\begin{cases} \ds
\frac{\sqrt{\alpha-\beta}}{2^{1/4} }\begin{pmatrix} 1& 1\\ -1 & -1\end{pmatrix} \frac{1}{\sqrt{2 \pi (v-v') }}e  ^{-\frac{(u-u')^2}{2(v-v')}},& \text{ if } v>v',\\
0, & 
\begin{array}{l} \text{if }v < v' \text{ or} \\
	\text{if } v = v' \text{ and } u \neq u',
	\end{array}
\end{cases}
\end{multline}
as $N\to \infty$. 
\end{theorem}

It is very curious that the double integral part of the Pearcey kernel appears in the scaling limit at the cusp point, but only in the subleading term. A similar phenomenon was already observed in \cite{CJ} 
on the smooth parts of
the liquid-gas boundary. The gas phase is dominant with a subleading
Airy behavior. Also here the gaussian part of the Airy kernel is hiding in the gas kernel \cite[\S 3.2]{CJ}.  With some effort we can also find this from our approach.

\section{Non-intersecting paths} \label{sec:paths}
We discuss the non-intersecting paths on the Aztec diamond. What follows
in this section is not new, and can be found in several places,
see e.g.\ \cite{J02,J05,J06} and the recent works \cite{BDS,J17}. 
Note however, that we use a (random) double Aztec diamond to extend the
paths, see Section \ref{subsec:doubleAD} below, instead of a deterministic
extension in \cite{J02,J05}. 

\subsection{Non-intersecting paths} \label{subsec:paths}

The South, West and East dominos are marked by line segments as shown in
Figure \ref{fig:AztecDiamondPaths}. The North dominos have no marking.
There are also particles on the West and South dominos, but these will only
play a role later on. 
We look at the line segment as part of paths that go from left to right
and either go up (in a West domino), down (in an East domino), or go
horizontal (in a South domino). Each segment enters a domino in the black
square, and exits it from the white square within the domino.

We include the marking of the dominos into the tiling we obtain
non-intersecting paths, starting at the lower left side of the 
Aztec diamond
and ending at the lower right side. In the pictures that follow we
forget about the black/white shading of the dominos.

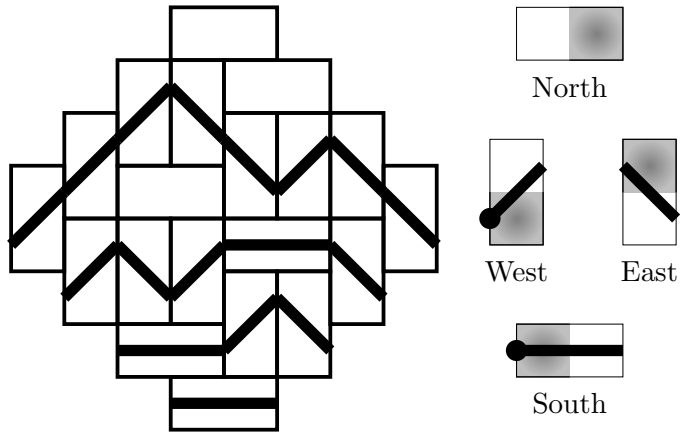
\begin{figure}[h]
\begin{center}
\begin{tikzpicture}[scale=0.7]


\foreach \x/\y in {-4/0,-3/1,-3/-1,-2/2,-1/-1,0/-2,1/1}
{ 
\draw [line width = 0.5mm] (\x,\y-1) rectangle (\x+1,\y+1);
	\draw [line width = 1.5mm] (\x,\y-0.5)--(\x+1,\y+0.5);
}

\foreach \x/\y in {-2/-1,-1/2,0/1,1/-2,2/-1,2/1,3/0}
{ 
\draw [line width = 0.5mm] (\x,\y-1) rectangle (\x+1,\y+1);
	\draw [line width = 1.5mm] (\x,\y+0.5)--(\x+1,\y-0.5);
}

\foreach \x/\y in {0/-3,-1/-2,1/0}
{  
\draw [line width = 0.5mm] (\x-1,\y-1) rectangle (\x+1,\y);
	\draw [line width = 1.5mm] (\x-1,\y-0.5)--(\x+1,\y-0.5);
}

\foreach \x/\y in {-1/1,0/4,1/3}
{ 
\draw [line width=0.5mm] (\x-1,\y-1) rectangle (\x+1,\y);
}

 \draw[outer color=lightgray,inner color=gray]
(5,0.5)--(5,-0.5)--(6,-0.5)--(6,0.5);
\draw [line width = 0mm] (5,-0.5) rectangle (6,1.5);
	\draw [line width = 1.5mm] (5,0)--(6,1);
	\fill (5, 0) circle [radius = 2mm]; 
\draw (5.5,-1) node {West};

 \draw[outer color=lightgray,inner color=gray]
(6.5,3)--(7.5,3)--(7.5,4)--(6.5,4);
\draw[line width = 0mm] (5.5,3) rectangle (7.5,4);
\draw (6.5,2.5) node {North};

 \draw[outer color=lightgray,inner color=gray]
(6.5,-3)--(5.5,-3)--(5.5,-2)--(6.5,-2);
\draw[line width = 0mm] (5.5,-3) rectangle (7.5,-2);
	\draw [line width = 1.5mm] (5.5,-2.5)--(7.5,-2.5);
	\fill (5.5, -2.5) circle [radius = 2mm]; 
\draw (6.5,-3.5) node {South};

 \draw[outer color=lightgray,inner color=gray]
(7.5,0.5)--(7.5,1.5)--(8.5,1.5)--(8.5,0.5);
\draw[line width = 0mm] (7.5,-0.5) rectangle (8.5,1.5);
	\draw [line width = 1.5mm] (7.5,1)--(8.5,0);
\draw (8,-1) node {East};
\end{tikzpicture}

\caption{Line segments and particles on the dominos, that
lead to non-intersecting paths in a domino tiling of 
the Aztec diamond.
\label{fig:AztecDiamondPaths}}
\end{center}
\end{figure}

Each path ends at the same height as where it started. Thus along each path
the number of West dominos is the same as the number of East dominos.
Also  each path has in each row the same number of West dominos
as the number of East dominos. Since we need this property, we
state it in a separate lemma.

\begin{lemma} \label{lem:EastWest}
In any domino tiling of the Aztec diamond, the following holds:
\begin{itemize}
\item[\rm (a)] In each row, the number of West dominos is the same
as the number of East dominos.
\item[\rm (b)] In each column, the number of North dominos is the
same as the number of South dominos.
\end{itemize}
\end{lemma} 
\begin{proof}
(a) This is immediate, since each path has the same number of 
West and East dominos in each row.
\medskip

(b) This follows from part (a), since the Aztec diamond model
is symmetric under 90 degrees rotation. 
\end{proof}

\subsection{Double Aztec diamond} \label{subsec:doubleAD}

The lengths of the paths vary greatly. To obtain a more symmetric picture,
which will be useful for what follows,
we attach to the right bottom side of the original Aztec diamond 
of size $2N$ another one of size $2N-1$ as in  
Figure~\ref{fig:DoublePathsParticles}.

\begin{lemma}
Any domino tiling of the double Aztec diamond splits into a domino 
tiling of the original Aztec diamond of size $2N$ and a domino tiling 
of the attached Aztec diamond of  size $2N-1$.

In other words: there are no dominos that are partly in the original 
Aztec diamond and partly in the newly attached Aztec diamond.
\end{lemma}

\begin{proof}
The smaller Aztec diamond is attached to the original Aztec diamond along
its south-east boundary. In the checkerboard coloring there are only white
squares adjacent to this boundary as is seen in Figure \ref{fig:AztecTiling}. 

If some dominos are partly in the original Aztec diamond
and partly in the new one, then these dominos would cover a number
of white squares in the original Aztec diamond, and no black squares. 
That would leave us with more black squares than white squares 
that should be covered by dominos. 
This is impossible, since each domino covers exactly one black and one white square. 
\end{proof}

\begin{figure}[t]
\begin{center}
\begin{tikzpicture}[scale=0.7]

\foreach \x/\y in {-4/0,-3/1,-3/-1,-2/2,-1/-1,0/-2,1/1,5/-5,5/-3,6/-4}
{ 
\draw [line width = 0.5mm] (\x,\y-1) rectangle (\x+1,\y+1);
	\draw [line width = 1.5mm] (\x,\y-0.5)--(\x+1,\y+0.5);
	\fill (\x,\y-0.5) circle [radius = 2mm]; 
}

\foreach \x/\y in {-2/-1,-1/2,0/1,1/-2,2/-1,2/1,3/0,1/-4,2/-3,4/-5}
{ 
\draw [line width = 0.5mm] (\x,\y-1) rectangle (\x+1,\y+1);
	\draw [line width = 1.5mm] (\x,\y+0.5)--(\x+1,\y-0.5);
}

\foreach \x/\y in {-1/-2,0/-3,1/0,3/-4,4/-3,4/-1}
{  
\draw [line width = 0.5mm] (\x-1,\y-1) rectangle (\x+1,\y);
	\draw [line width = 1.5mm] (\x-1,\y-0.5)--(\x+1,\y-0.5);
	\fill (\x-1,\y-0.5) circle [radius = 2mm]; 	
}

\foreach \x/\y in {-1/1,0/4,1/3,3/-5,4/-2,4/-6}
{ 
\draw [line width=0.5mm] (\x-1,\y-1) rectangle (\x+1,\y);
}


\fill (4, -0.5) circle [radius = 2mm]; 
\fill (5, -1.5) circle [radius = 2mm]; 
\fill (6, -2.5) circle [radius = 2mm]; 
\fill (7, -3.5) circle [radius = 2mm]; 
\end{tikzpicture}

\caption{A tiling of the double Aztec diamond with non-intersecting
paths and particles along the paths.
\label{fig:DoublePathsParticles}}
\end{center}
\end{figure}
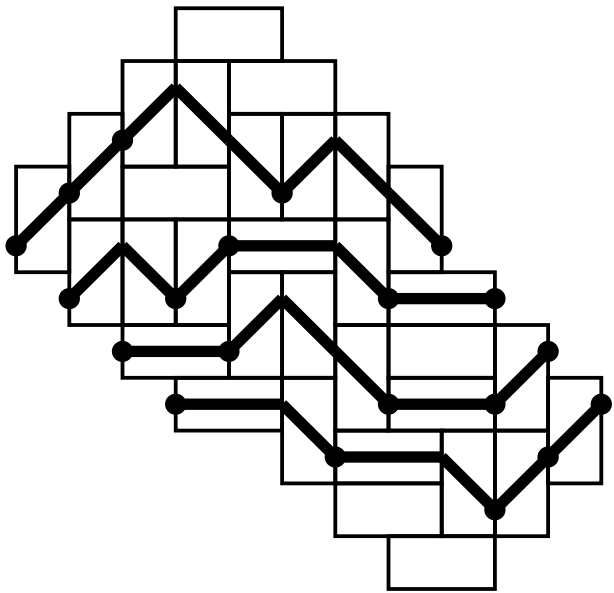

We thus cover the second Aztec diamond with dominos, independently 
from what we did in the original Aztec diamond. We also
put in the markings with line segments. Note however that a horizontal
domino at the very bottom of the second Aztec diamond is now a ``North domino".
A possible tiling is shown in Figure \ref{fig:DoublePathsParticles}
together with the corresponding non-intersecting paths and the particle system 
that we discuss in Section \ref{subsec:particles}.

The double Aztec diamond with partial overlap is considered in \cite{ACvM2,AJvM},
where the phenomenon of a tacnode is studied. For us, the two Aztec diamonds
do not overlap and there is no tacnode phenomenon. 

\subsection{Particle system} \label{subsec:particles}

Next we put particles on the paths. On each West and South domino we put
a particle at the left-most part of its marking as already shown in Figure \ref{fig:AztecDiamondPaths}.  
We do not put any particle
on the East and North dominos, although there may be a particle on the right
edge of an East (or West or South) domino if it is connected to a
West or South domino.  This a slight deviation from what we
did before. We now put the particles on the boundary of the West
and South dominos and not in the center of the black square.
We also put a particle at the end of each path, see 
Figure \ref{fig:DoublePathsParticles}.

Each path contains $2N+1$ particles.  The particles
are interlacing if we consider them along diagonal lines going from
north-west to south-east. There are
$2N+1$ diagonal lines and each diagonal line contains $2N$ particles.

To find a clearer picture, we forget about the dominos and  we rotate 
the figure over 45 degrees in clockwise direction. We further  introduce 
a shear transformation so that the starting and ending points of each path remain at the same height.
In a formula: $(x,y)$ is mapped to $(x+y,y)$.

Then we obtain a figure as in Figure \ref{fig:ParticleVerticalLines} where each
vertical line contains $2N$ particles.  The paths now consist of 
diagonal (right-up) parts, horizontal parts and vertical parts.
The diagonal parts come from the West dominos.
The horizontal parts come from the South dominos
and the vertical parts come from the East dominos.
Each path contains $2N+1$ particles.

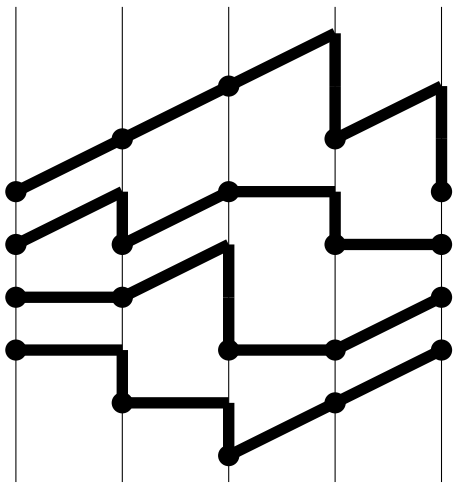
\begin{figure}[t]
\begin{center}
\begin{tikzpicture}[scale=0.7]
\foreach \x/\y in {-4/0,-3/1,-3/-1,-2/2,-1/-1,0/-2,1/1,5/-5,5/-3,6/-4}
{ 
	\draw [line width = 1.5mm] (\x+\y-0.5,\y-0.5)--(\x+\y+1.5,\y+0.5);
	\fill (\x+\y-0.5,\y-0.5) circle [radius = 2mm]; 
}

\foreach \x/\y in {-2/-1,-1/2,0/1,1/-2,2/-1,2/1,3/0,1/-4,2/-3,4/-5}
{ 
	
	\draw [line width = 1.5mm] (\x+\y+0.5,\y+0.5)--(\x+\y+0.5,\y-0.5);
}

\foreach \x/\y in {-1/-2,0/-3,1/0,3/-4,4/-3,4/-1}
{  
	\draw [line width = 1.5mm] (\x+\y-1.5,\y-0.5)--(\x+\y+0.5,\y-0.5);
	\fill (\x+\y-1.5,\y-0.5) circle [radius = 2mm]; 
}

\foreach \x/\y in {-1/1,0/4,1/3,3/-5,4/-2,4/-6}
{ 
}

\fill (3.5, -0.5) circle [radius = 2mm]; 
\fill (3.5, -1.5) circle [radius = 2mm]; 
\fill (3.5, -2.5) circle [radius = 2mm]; 
\fill (3.5, -3.5) circle [radius = 2mm]; 

\foreach \x in {-4.5,-2.5,-0.5,1.5,3.5}
\draw (\x,-6) -- (\x,3);
\end{tikzpicture}

\caption{Particle system on vertical lines
\label{fig:ParticleVerticalLines}}
\end{center}
\end{figure}

\subsection{Modified paths on a graph}

We are going to modify the paths, in such a way that each particle
is preceeded by a horizontal step of a half unit (except for the initial particles).

The particles are on the integer lattice $\mathbb Z  \times \mathbb Z$. 
We put coordinates so that the initial vertices are at $(0,j)$ for $j=0, \ldots, 2N-1$
and the ending vertices are at $(2N,j)$ for $j=0, \ldots, 2N-1$.
Each path consists of $2N$ parts, where the $m$th part goes from
$(m-1,k)$ to $(m,l)$ for some $k, l \in \mathbb Z$ with $m \leq k+1$. We modify
this part by an affine transformation that results in a path from
$(m-1,k)$ to $(m-1/2,l)$, followed by a horizontal step from
$(m-1/2,l)$ to $(m,l)$.

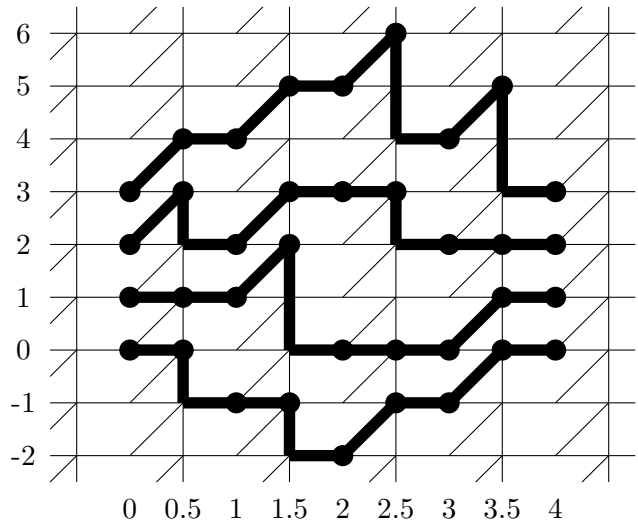
\begin{figure}[h]
\begin{center}
\begin{tikzpicture}[scale=0.7]

\foreach \x/\y in {-4/0,-3/1,-3/-1,-2/2,-1/-1,0/-2,1/1,5/-5,5/-3,6/-4}
{ 
	\draw [line width = 1.5mm] (\x+\y-0.5,\y-0.5)--(\x+\y+0.5,\y+0.5);
	\fill (\x+\y-0.5,\y-0.5) circle [radius = 2mm]; 
}
\foreach \x/\y in {-3/1,-2/2,-1/-1,0/-2,1/1,5/-5,5/-3,6/-4}
\draw [line width = 1.5mm] (\x+\y-1.5,\y-0.5)--(\x+\y-0.5,\y-0.5);

\foreach \x/\y in {-2/-1,-1/2,0/1,1/-2,2/-1,2/1,3/0,1/-4,2/-3,4/-5}
{ 
	
	\draw [line width = 1.5mm] (\x+\y-0.5,\y+0.5)--(\x+\y-0.5,\y-0.5);
}

\foreach \x/\y in {-1/-2,0/-3,1/0,3/-4,4/-3,4/-1}
{  
	\draw [line width = 1.5mm] (\x+\y-1.5,\y-0.5)--(\x+\y-0.5,\y-0.5);
	\fill (\x+\y-1.5,\y-0.5) circle [radius = 2mm]; 
}
\foreach \x/\y in {1/0,3/-4,4/-3,4/-1}
\draw [line width = 1.5mm] (\x+\y-2.5,\y-0.5)--(\x+\y-1.5,\y-0.5);

\foreach \x/\y in {-1/1,0/4,1/3,3/-5,4/-2,4/-6}
{ 
}

\foreach \y in {-3,-2,-1,0}
{
\fill (3.5, \y-0.5) circle [radius = 2mm]; 
\draw [line width = 1.5mm] (2.5,\y-0.5)--(3.5,\y-0.5);
}

\fill (-3.5,-3.5) circle [radius = 2mm]; 
\fill (-3.5,-2.5) circle [radius = 2mm]; 
\fill (-3.5,-0.5) circle [radius = 2mm]; 
\fill (-3.5,0.5) circle [radius = 2mm]; 

\fill (-1.5,-4.5) circle [radius = 2mm]; 
\fill (-1.5,-1.5) circle [radius = 2mm]; 
\fill (-1.5,-0.5) circle [radius = 2mm]; 
\fill (-1.5,1.5) circle [radius = 2mm]; 

\fill (0.5,-4.5) circle [radius = 2mm]; 
\fill (0.5,-3.5) circle [radius = 2mm]; 
\fill (0.5,-0.5) circle [radius = 2mm]; 
\fill (0.5,2.5) circle [radius = 2mm]; 

\fill (2.5,-3.5) circle [radius = 2mm]; 
\fill (2.5,-2.5) circle [radius = 2mm]; 
\fill (2.5,-1.5) circle [radius = 2mm]; 
\fill (2.5,1.5) circle [radius = 2mm];

\foreach \x in {-6,-4,-2,-0, 2,4}
\draw (\x+0.5,-6) -- (\x+0.5,3);
\foreach \y in {-5,-4,-3,-2,-1,0,1,2,3}
\draw (-6,\y-0.5) -- (5,\y-0.5);
\foreach \y in {-5,-4,-3,-2,-1,0,1,2}
{ \foreach \x in {-4,-2,-0, 2,4}
\draw (\x-0.5,\y-0.5)--(\x+0.5,\y+0.5);
}

\foreach \y in {-5,-4,-3,-2,-1,0,1,2}
\draw(-6,\y)--(-5.5,\y+0.5);
\foreach \x in {-4,-2,-0, 2,4}
\draw (\x-0.5,2.5)--(\x,3);
\foreach \x in {-6,-4,-2,-0, 2,4}
{\draw (\x,-6)--(\x+0.5,-5.5);
}

\foreach \x in {0,0.5,1,1.5,2,2.5,3,3.5,4}
\draw (\x+\x-4.5,-6.5) node {\x};
\foreach \y in {-2,-1,0,1,2,3,4,5,6}
\draw (-6.5,\y-3.5) node {\y};

\end{tikzpicture}

\caption{Modified paths on a directed graph. Horizontal and diagonal
edges are oriented to the right and vertical edges are oriented downwards. 
\label{fig:PathsOnGraph}}
\end{center}
\end{figure}

We also extend the particle system by putting particles at
the new vertical lines. If there is no vertical part, then the path has a unique
intersection point with the vertical line, and we put a particle there. If there is
a vertical part of the path at that level, then we put the particle at the highest point,  
see Figure \ref{fig:PathsOnGraph}.
Now there are $4N+1$ particles on each path.

The new paths have a two-step structure.
Starting from an initial position, we either move horizontally to the right 
by half a unit and stay at the same height, or we go diagonally up by one unit
in the vertical direction and horizontally by half a unit.
We call this a Bernoulli step. In both cases we end at a particle on the line
with horizontal coordinate $1/2$.

Then we make a number of vertical down steps followed by a horizontal step 
by half a unit to the right. The number of down steps can be  
any non-negative integer, including zero.  We call this a geometric step.

Then we repeat the pattern. We do a Bernoulli step, a geometric step, a Bernoulli
step, etc. The final step in each path is a geometric step, which should take us to the
same height as where the path started.

Another requirement is that the resulting paths are non-intersecting.
Any such path structure is in one-to-one correspondence with a unique
domino tiling of the double Aztec diamond.

The geometric steps cannot be too far down, since each path has to return
to its initial height, and each up step is done by one unit only. 
 In the example, the largest geometric down step is by two units, but in a
larger size example, one could imagine that larger steps are possible.

The paths lie on an infinite directed graph that is also shown in Figure \ref{fig:PathsOnGraph}.
We call it the Aztec diamond graph. Its set of vertices is 
$(\frac{1}{2} \mathbb Z) \times \mathbb Z$.
From a vertex $(m,n) \in \mathbb Z^2$ there are two directed edges 
\begin{itemize}
\item From $(m,n)$ to $(m+ \frac{1}{2},n+1)$ (diagonal up step)
\item From $(m,n)$ to $(m+ \frac{1}{2},n)$ (horizontal step)
\end{itemize}
The transition from $m$ to $m+\frac{1}{2}$ is a Bernoulli step.

From $(m+\frac{1}{2},n)$ there are also two directed edges
\begin{itemize}
\item From $(m + \frac{1}{2},n)$ to $(m+ \frac{1}{2},n-1)$ (vertical down step)
\item From $(m + \frac{1}{2},n)$ to $(m+ 1,n)$ (horizontal step)
\end{itemize}
We go from level $m+\frac{1}{2}$ to level $m +1$ by making a number of vertical down
steps (possibly zero) and then making the horizontal step to the right.
The transition from $m+\frac{1}{2}$ to $m+1$ is a geometric step.

For a given $N$ the paths start at the vertices with
coordinates $(0,j)$, $j=0, \ldots, 2N-1$, and end at the vertices
with coordinates $(2N,j)$, $j=0, \ldots, 2N-1$. The paths lie on
the graph and are not allowed to intersect, that is, the set of
vertices for two different paths is disjoint.
Since the graph is  planar and paths cannot go to the left, the
paths maintain their relative ordering. The path from
$(0,j)$ to $(2N,j)$ stays below the path from $(0,j+1)$ to $(2N,j+1)$
at all levels. 

\subsection{Weights} \label{subsec:weights}

There is a one-to-one correspondence between tilings of the double
Aztec diamond and non-intersecting paths on the Aztec diamond graph
with prescribed initial and ending positions as described above.

In the two periodic Aztec diamond we assign a weight \eqref{eq:weightT}
to a tiling with the corresponding probability \eqref{eq:probT}.
To be able to transfer this to the paths, we recall that in a Bernoulli step
a diagonal up-step corresponds
to a West domino, and a horizontal step corresponds to a South domino.
The vertical steps in a geometric step correspond to East dominos. 
The horizontal step that closes a geometric step was added artificially
and does not correspond to a domino.

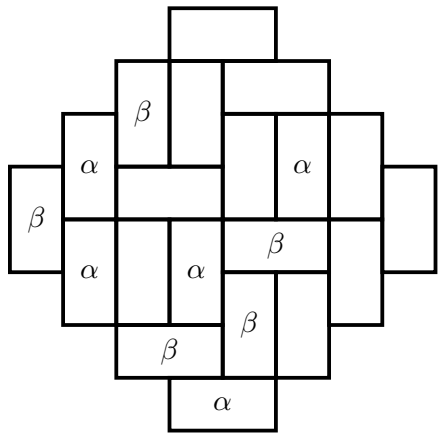
\begin{figure}[t]
\begin{center}
\begin{tikzpicture}[scale=0.7]


\foreach \x/\y in {-4/0,-3/1,-3/-1,-2/2,-1/-1,0/-2,1/1}
{ 
\draw [line width = 0.5mm] (\x,\y-1) rectangle (\x+1,\y+1);
}

\foreach \x/\y in {-2/-1,-1/2,0/1,1/-2,2/-1,2/1,3/0}
{ 
\draw [line width = 0.5mm] (\x,\y-1) rectangle (\x+1,\y+1);
}

\foreach \x/\y in {0/-3,-1/-2,1/0}
{  
\draw [line width = 0.5mm] (\x-1,\y-1) rectangle (\x+1,\y);
}

\foreach \x/\y in {-1/1,0/4,1/3}
{ 
\draw [line width=0.5mm] (\x-1,\y-1) rectangle (\x+1,\y);
}

\foreach \x in {-3.5} \draw (\x,0) node {$\beta$};
\foreach \x in {-2.5,1.5} \draw (\x,1) node {$\alpha$}; 
\foreach \x in {-2.5,-0.5} \draw (\x,-1) node {$\alpha$};
\foreach \x in {-1.5} \draw (\x,2) node {$\beta$}; 
\foreach \x in {0.5} \draw (\x,-2) node {$\beta$};

\foreach \y in {-3.5} \draw(0,\y) node {$\alpha$};
\foreach \y in {-0.5} \draw(1,\y) node {$\beta$};
\foreach \y in {-2.5} \draw(-1,\y) node {$\beta$};
\end{tikzpicture}

   \caption{Equivalent weighting of a tiling of the Aztec diamond. 
   The West dominos  in an odd row and the South dominos in an
   even column have weight $\alpha =a^2$. West dominos in an even row
   and South dominos in an odd column have weight $\beta = b^2$. 
   All North and East dominos have weight $1$.
\label{fig:AztecDiamondWeights}}
\end{center}
\end{figure}	 

We do not see the North dominos in the paths, and therefore we cannot
transfer the weights on the dominos to weights on path segments directly.
It is possible to assign weights to dominos in which North dominos have weight one 
and which is equivalent to \eqref{eq:weightD} as it leads to the 
same probabilities \eqref{eq:probT} on tilings. This was also 
done  in \cite{CJ}, but we present it in a different way here. 
Because of Lemma \ref{lem:EastWest} each column has the same number 
of North and South dominos, and these dominos all have
the same weight \eqref{eq:weightD}.
We obtain the same weight \eqref{eq:weightT} of a tiling if instead of assigning the same weight
$a$ or $b$ to all the horizontal dominos in a column, we assign $a^2$ or $b^2$ to
the South dominos and weight $1$ to the North dominos in that column.

For symmetry reasons we apply the same operation to  East and West dominos.
Then instead of \eqref{eq:weightD} we assign the following weight to a domino $D$,
where we recall that $\alpha = a^2$ and $\beta = b^2$,
\begin{equation} \label{eq:weightDhat} 
	\widehat{w}(D) =
	 \begin{cases} \alpha, & \text{if $D$ is a South domino in an even column}, \\
	  \beta, & \text{if $D$ is a South domino in an odd column}, \\
	  \beta, & \text{if $D$ is a West  domino in an even row}, \\
	  \alpha, & \text{if $D$ is a West domino in an odd row}, \\
	  1, & \text{if $D$ is a North or East domino}.
	\end{cases}	 
\end{equation}
In our running example, the new weights $\widehat{w}(D)$ are 
shown in Figure \ref{fig:AztecDiamondWeights}.

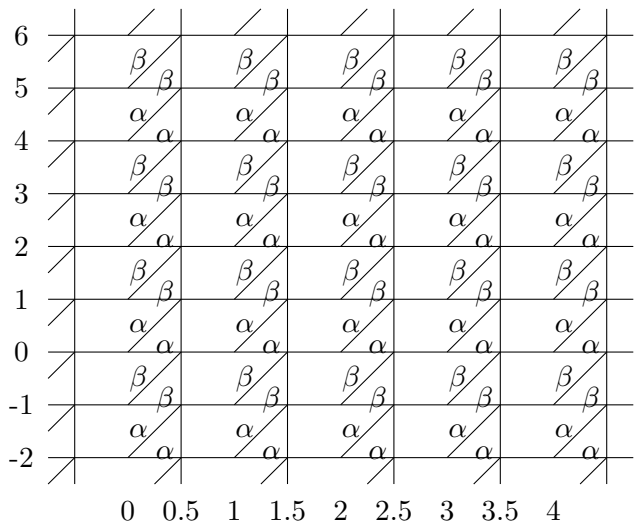
\begin{figure}[t]
\begin{center}
\begin{tikzpicture}[scale=0.7]
\foreach \x in {-6,-4,-2,-0, 2,4}
\draw (\x+0.5,-6) -- (\x+0.5,3);
\foreach \y in {-5,-4,-3,-2,-1,0,1,2,3}
\draw (-6,\y-0.5) -- (5,\y-0.5);
\foreach \y in {-5,-4,-3,-2,-1,0,1,2}
{ \foreach \x in {-4,-2,-0, 2,4}
\draw (\x-0.5,\y-0.5)--(\x+0.5,\y+0.5);
}
\foreach \y in {-5,-4,-3,-2,-1,0,1,2}
\draw(-6,\y)--(-5.5,\y+0.5);
\foreach \x in {-4,-2,-0, 2,4}
\draw (\x-0.5,2.5)--(\x,3);
\foreach \x in {-6,-4,-2,-0, 2,4}
{\draw (\x,-6)--(\x+0.5,-5.5);
}

\foreach \x in {0,0.5,1,1.5,2,2.5,3,3.5,4}
\draw (\x+\x-4.5,-6.5) node {\x};
\foreach \y in {-2,-1,0,1,2,3,4,5,6}
\draw (-6.5,\y-3.5) node {\y};


\foreach \y in {-5,-3,-1,1}
 \foreach \x in {-4,-2,-0,2,4} 
 {\draw (\x-0.3 ,\y) node {$\alpha$};
  \draw (\x+0.2, \y-0.4) node {$\alpha$};
  }
  
\foreach \y in {-4,-2,0,2}
 \foreach \x in {-4,-2,-0,2,4} 
 {\draw (\x-0.3 ,\y) node {$\beta$};
  \draw (\x+0.2, \y-0.4) node {$\beta$};
  }
\end{tikzpicture}

\caption{Weights on the edges of the Aztec diamond graph corresponding
to the weights \eqref{eq:weightDhat} on dominos. 
   	Horizontal edges
   and diagonal edges from level $m \in \mathbb Z$ to
   $m+1/2$ have weight $\alpha$ in even numbered rows
   and weight $\beta$ in odd numbered rows. 
	The vertical edges and  horizontal edges
	from level $m + 1/2$ to $m+1$ have weight $1$
	and these weights are not shown in the figure.
\label{fig:AztecGraphWeights}}
\end{center}
\end{figure}

Since North dominos have weight $1$ we can transfer the weights 
\eqref{eq:weightDhat} on dominos to weights on the edges of the  
Aztec diamond graph. The result is shown in Figure \ref{fig:AztecGraphWeights}.
The weights alternate per row. 

Horizontal edges from $(m,n) \in \mathbb Z^2$ to $(m+1/2,n)$ 
and diagonal edges from $(m,n)$ to $(m+1/2,n+1)$ have weight 
$\alpha$ if $n$ is even and weight $\beta$ if $n$ is odd. 
All other edges have weight $1$. 

\subsection{Transition matrices} \label{subsec:transition}

We use the layered structure of the Aztec diamond graph to introduce
transition matrices between levels. Here a level is just the horizontal 
coordinate. There are integer levels $m$ and half-integer levels
$m+1/2$ with $m \in \mathbb Z$. 

The transition from level $m$ to $m+1/2$ is  a Bernoulli step. 
Because of the weights, the transition matrix is,
with $x,y \in \mathbb Z$,
\begin{equation} \label{eq:Tmmh} 
	T_{m,m+1/2}(x,y) = \begin{cases} \alpha, & \text{if } 
	x \text{ is even and } y \in \{x, x+1\}, \\
	\beta, & \text{if } x \text{ is odd and } y \in \{x, x+1\}, \\
	0, & \text{otherwise}.
	\end{cases}
\end{equation}
Then $T_{m,m+1/2}$ is $2$-periodic, namely 
$T_{m,m+1/2}(x+2,y+2) = T_{m,m+1/2}(x,y)$ for all $x,y \in \mathbb Z$.
As a matrix it is a block Laurent matrix (i.e., a 
block Toeplitz matrix that is infinite in both directions)
with $2 \times 2$ blocks.
The diagonal block is 
$\begin{pmatrix} \alpha & \alpha \\ 0 & \beta \end{pmatrix}$,
the block on the first upper diagonal is
$\begin{pmatrix} 0 & 0 \\ \beta & 0 \end{pmatrix}$,
and all other diagonals are zero. The associated symbol \cite{BG} is
\begin{equation} \label{eq:Ammh} 
	A_{m,m+1/2}(z) = 
		\begin{pmatrix} \alpha & \alpha \\ 0 & \beta \end{pmatrix}
 + \begin{pmatrix} 0 & 0 \\ \beta & 0 \end{pmatrix} z 
= \begin{pmatrix} \alpha & \alpha \\ \beta z & \beta \end{pmatrix}
\end{equation}
with $z \in \mathbb C$.

To go from level $m+1/2$ to level $m+1$, we make a number of vertical 
down steps (possibly zero) and then a horizontal step.
All weights are $1$ and so the transition matrix is
\begin{equation} \label{eq:Tmhmp} 
	T_{m+1/2,m+1}(x,y) =
	 \begin{cases} 1, & \text{ if } y \leq x, \\
   	0, & \text{ if } y > x. \end{cases} \end{equation} 
This is a Laurent matrix, but we want to view it as a a  block Laurent
matrix with $2 \times 2$ blocks. The diagonal block is
$ \begin{pmatrix} 1 & 0 \\ 1 & 1 \end{pmatrix}$, all blocks below
the main diagonal are $\begin{pmatrix} 1 & 1 \\ 1 & 1 \end{pmatrix}$
and all blocks above the main diagonal are zero. The symbol is
\begin{align}  
	A_{m+1/2,m+1}(z)  & = \begin{pmatrix} 1 & 0 \\ 1 & 1 \end{pmatrix}
	+ \sum_{j=-\infty}^{-1}  \begin{pmatrix} 1 & 1 \\ 1 & 1 \end{pmatrix} z^j \label{eq:Amhmp}
	 = \frac{1}{z-1} \begin{pmatrix} z & 1 \\ z & z  \end{pmatrix},
	\end{align}
with $|z| > 1$.

Then (the product is matrix multiplication)
\begin{equation} \label{eq:defTmmp} 
	T = T_{m,m+1/2} T_{m+1/2,m+1} 
\end{equation} 
is the transition matrix from level $m$ to level $m+1$, and $T$ is two periodic.
The symbol for $T$  is easily seen to be the product of \eqref{eq:Ammh} 
and \eqref{eq:Amhmp} 
\begin{align} \nonumber
	A(z) & = A_{m,m+1/2}(z) A_{m+1/2,m+1}(z) \\
 	& = \frac{1}{z-1} \begin{pmatrix} 2\alpha z & \alpha(z+1) \\ 
	 	\beta z(z+1)& 2 \beta z \end{pmatrix} \label{eq:defA2}
			\end{align} 
which agrees with \eqref{eq:defA}.

More generally, for any integers $m < m'$ we have a transition matrix
$T^{m'-m}$ to go from level $m$ to level $m'$ with symbol $A^{m'-m}$.
In particular $T^{2N}$ is the transition matrix from level $0$
to $2N$ with symbol $A^{2N}$.

Now we want to invoke the Lindstr\"om-Gessel-Viennot lemma
\cite{GV,L}, see also \cite[Theorem 3.1]{J06} for a proof, which
gives an expression for the weighted number of non-intersecting
paths on the graph with prescribed  starting and ending positions.
For us, the starting positions are $(0,j)$, $j=0, \ldots, 2N-1$
and the ending positions $(2N, j)$, $j=0, \ldots, 2N-1$. 
Since $T^{2N}(j,k)$ is the sum of all weighted paths from 
$(0,j)$ to $(2N,k)$ and so by the Lindstr\"om-Gessel-Viennot lemma
the partition function is a determinant
\begin{equation} \label{eq:ZNdet} 
	Z_N = \det \left[ T^{2N}(j,k) \right]_{j,k=0, \ldots, 2N-1}. 
	\end{equation}

Because of the layered structure in the graph, we can also look at the positions
of the particles at intermediate levels $m$. We restrict
to integer values $m$, but we could also include the half-integer values.

Given an admissible $2N$-tuple of non-intersecting paths, we 
then find a point set configuration $(x_j^m)_{j=0, m=1}^{2N-1,2N-1}$ where $x_0^m < x_1^m < \cdots < x_{N-1}^m$
are the vertical coordinates of the particles at level $m$.
The probability measure on admissible tuples of non-intersecting
paths yields a probability measure on particle configurations
in $\{1, \ldots, 2N-1\} \times \mathbb Z^N $.

Then another application of the Lindstr\"om-Gessel-Viennot lemma
yields that the joint probability for the particle configuration
$(x_j^m)_{j=0, m=1}^{2N-1,2N-1}$ is
\begin{multline} \label{eq:multilevelAD} 
	\text{Prob} \left((x_j^m)_{j=0, m=1}^{2N-1,2N-1}\right) = 
 \frac{1}{Z_N} \prod_{m=0}^{N-1}  \det \left[ T(x_j^m, x_k^{m+1}) \right]_{j,k=0}^{2N-1}
 \\
 \text{with } x_j^0 = x_j^{2N} = j \text{ for } j=0, \ldots, 2N-1.
  \end{multline}
The point process \eqref{eq:multilevelAD} is determinantal. The correlation kernel
is given by the Eynard-Mehta theorem \cite{EM}, see also \cite{B,BR}.
\begin{proposition}
The correlation kernel is
\[ K(m, x ; m', y)
	= - \chi_{m > m'} T^{m'-m}(y,x)
		+ \sum_{i,j=0}^{2N-1} T^m(i,x) \left[\mathbf{G}^{-1}\right]_{j,i}
			T^{2N-m'}(y,j) \]
where $\mathbf{G} = \left(T^{2N}(i,j)\right)_{i,j=0}^{2N-1}$. 
		\end{proposition} 
In particular, if $m=m'$
\[ K_m(x,y) = K(m,x; m, y)
	=  \sum_{i,j=1}^{2N-1} T^m(i,x) \left[\mathbf{G}^{-1} \right]_{j,i}
			T^{2N-m}(y,j) \]
gives the correlation kernel for the particles at level $m$.

Note that $\mathbf{G}$ is a finite size submatrix of the two-sided infinite matrix $T^{2N}$
and $\mathbf{G}^{-1}$ is the inverse of this matrix. 
To handle the correlation kernel in the large $N$ limit, we 
need to find  a suitable way to invert the matrix $\mathbf{G}$. 
A fruitful idea is to biorthogonalize.
This can be done with matrix valued orthogonal polynomials,
and we will discuss this in greater generality in the next section.

\section{Determinantal point processes and MVOP} \label{sec:MVOP}

\subsection{The model} \label{subsec:model}

We analyze the following situation.
We take an integer $p \geq 1$, and we consider transition matrices
$T_m : \mathbb Z^2 \to \mathbb R$ for $m \in \mathbb Z$ 
that are $p$-periodic. This means that
\[ T_m(x+p,y+p) = T_m(x,y) \]
for every $m$ and $x,y \in \mathbb Z$. 

The model also depends on integers $N, L \in \mathbb N$ and $M \in \mathbb Z$. 
There will be $L+1$ levels numbered as $0,1, \ldots, L$. At each level $m$ 
there are $pN$ particles at integer positions denoted by 
\[ x_0^{m} < x_1^{m} < \cdots < x_{pN-1}^{m}. \]
The initial and ending positions (at levels $0$ and $L$) are deterministic
and are given by consecutive integers
\begin{equation} \label{eq:boundary} 
	x_j^0 = j, \qquad  x_j^L = pM +j, \qquad j = 0, \ldots, pN-1. 
	\end{equation}

Our assumption for this section is the following.
\begin{assumption} \label{ass:multilevel} 
$(x_j^m)_{j=0, m=0}^{pN-1,L}$ is a multi-level particle system with joint
probability 
\begin{multline} \label{eq:multilevel} 
	{\rm Prob} \left((x_j^m)_{j=0, m=0}^{pN-1,L} \right) 
	 = 
 \frac{1}{Z_N} 
 \det\left[ \delta_j(x_k^0)\right]_{j,k=0}^{pN-1} \\
  \cdot \left(\prod_{m=0}^{L-1}  \det \left[ T_m(x_j^m, x_k^{m+1}) \right]_{j,k=0}^{pN-1}
  \right)
 	\cdot \det\left[ \delta_{pM+j}(x_k^{L}) \right]_{j,k=0}^{pN-1}, 
 	\end{multline}
 where the transition matrices $T_m$ are $p$-periodic for every $m$.
The constant $Z_N$ in \eqref{eq:multilevel} is a normalizing constant and 
$\delta_j(x) = \delta_{j,x}$ is the Kronecker delta.
\end{assumption}
The determinantal factors with the delta-functions in \eqref{eq:multilevel} 
ensure the boundary conditions \eqref{eq:boundary}.

Assumption \ref{ass:multilevel} is satisfied for the two periodic Aztec
diamond by \eqref{eq:multilevelAD}, provided we take $p=2$, $L=2N$, $M=0$ and $T_m=T$ for each integer $m$, with $T$ given by \eqref{eq:defTmmp}, using \eqref{eq:Tmmh}, \eqref{eq:Tmhmp}. 	
There are three crucial assumptions contained in Assumption \ref{ass:multilevel}.
\begin{itemize}
\item  The transition matrices are $p$-periodic. 
It means that the $T_m$ are block Laurent matrices with $p \times p$ blocks.
\item The initial and ending positions of the particles are at
consecutive integers. This assumption allows to make a connection with
matrix valued polynomials. Note that we allow for a shift $pM$ in the
positions at level $L$ compared to the initial positions.
\item The transition matrices are such that \eqref{eq:multilevel} is
a probability. That is, \eqref{eq:multilevel} is always non-negative 
and we can find a normalization constant $Z_N$ such that all
probabilities \eqref{eq:multilevel} add up to $1$.
\end{itemize}
We made two other assumptions, namely
\begin{itemize}
\item the number $pN$ of particles at each level is a multiple of $p$, and
\item the shift $pM$ in the positions of the particles at the final level
is also a multiple of $p$, 
\end{itemize}
but these are less essential. They are made for convenience and ease of notation
and could be relaxed if needed. 

For future analysis, we also assume
\begin{assumption} \label{ass:blocksymbols}
The symbols for the block Laurent matrices $T_m$, $m \in \mathbb Z$,
are analytic in a  common annular domain $R_1 < |z| < R_2$ in the complex plane.
\end{assumption}

For $m < m'$ we use 
\begin{equation} \label{eq:Tmn} 
	T_{m,m'} = T_m \cdot T_{m+1} \cdot \cdots \cdot T_{m'-1}, 
\end{equation} 
for the transition matrix from level $m$ to level $m'$. The matrix multiplication
is well defined because of Assumption \ref{ass:blocksymbols}.
Every $T_{m,m'}$ is a block Laurent matrix with the same period $p$. 

The Eynard-Mehta theorem \cite{EM} applies to \eqref{eq:multilevel}.
We present  the Eynard-Mehta theorem as stated in \cite{B}.
We assume $\phi_j$, $\psi_j$ for $j=0,1, \ldots, N-1$ are given functions,
and for functions $\phi, \psi : \mathbb Z \to  \mathbb R$, and
$T : \mathbb Z \times \mathbb Z \to \mathbb R$ we use 
$\ds (\phi \ast T)(y) = \sum_{x \in \mathbb Z} \phi(x) T(x,y)$,
$\ds (T \ast \psi)(x) = \sum_{y \in \mathbb Z} T(x,y) \psi(y)$, and
$\ds  \phi \ast \psi = \sum_{x \in \mathbb Z} \phi(x) \psi(x)$.

\begin{theorem}[Eynard-Mehta] \label{thm:EMtheorem}
A multi-level particle system of the form
\begin{multline} \label{eq:multilevel2} 
	{\rm Prob} \left((x_j^m)_{j=0, m=0}^{N-1,L} \right) 
	 = 
 \frac{1}{Z_N} 
 \det\left[ \phi_j(x_k^0)\right]_{j,k=0}^{N-1} \\
  \cdot \left( \prod_{m=0}^{L-1}  \det \left[ T_m(x_j^m, x_k^{m+1}) \right]_{j,k=0}^{N-1}
  \right)
 	\cdot \det\left[ \psi_j (x_k^{L}) \right]_{j,k=0}^{N-1} 
\end{multline}
where $\phi_j$, $\psi_j$ for $j=0,1, \ldots, N-1$ are arbitrary given functions,
is determinantal with correlation kernel.
\begin{multline} \label{eq:EMtheorem} 
	K(m,x;m', y)  
	=  - \chi_{m> m'} T_{m',m}(y,x) \\
		+ \sum_{i,j=0}^{N-1} \left(\phi_i \ast T_{0,m} \right)(x) 
			\left[\mathbf{G}^{-t}\right]_{i,j}
			\left(T_{m',L} \ast \psi_j \right)(y) 
			\end{multline}
where the Gram matrix $\mathbf{G}$ is defined by
\begin{equation} \label{eq:Gram} 
	\mathbf{G} = \left( G_{i,j} \right)_{i,j=0}^{N-1},
	 \qquad
	 	G_{i,j} = \phi_i \ast T_{0,L} \ast \psi_j. 
\end{equation} 
	\end{theorem}	
We apply Theorem \ref{thm:EMtheorem} to \eqref{eq:multilevel} 
and we find the following correlation kernel \eqref{eq:EMkernel}.

\begin{corollary} \label{cor:EM2}
The multi-level particle system \eqref{eq:multilevel} is
determinantal with correlation kernel
\begin{multline} \label{eq:EMkernel} 
	K(m,x; m',y)  
	=  - \chi_{m> m'} T_{m',m}(y,x) \\
		+ \sum_{i,j=0}^{pN-1}  T_{0,m}(i,x) 
			\left[\mathbf{G}^{-t}\right]_{i,j} T_{m',L}(y,pM+j)  
			\end{multline}
with 
\begin{equation} \label{eq:Gramkernel} 
	\mathbf{G} = \left( G_{i,j} \right)_{i,j=0}^{pN-1},
	 \qquad
	 	G_{i,j} = T_{0,L}(i,pM+j). 
\end{equation} 
\end{corollary}
\begin{proof}
This follows from Theorem \ref{thm:EMtheorem} since
$ (\delta_i \ast T_{0,m})(x) = T_{0,m}(i,x)$ and
$(T_{m',L} \ast \delta_{pM+j})(y) = T_{m',L}(y,pM+j)$.  
\end{proof}

The Gram matrix $\mathbf{G}$ in \eqref{eq:Gram} is a finite section of the block
Laurent matrix $T_{0,L}$. It has size $pN \times pN$ and we also view
it as a block Toeplitz matrix of size $N \times N$ with blocks of size $p \times p$.
It is part of the conclusion of the Eynard Mehta Theorem \ref{thm:EMtheorem}
that $\mathbf{G}$ is invertible, and so it is in particular a consequence of
Assumptions \ref{ass:multilevel} and \ref{ass:blocksymbols}
that the matrix $\mathbf{G}$ is invertible.

\subsection{Symbols and matrix biorthogonality} \label{subsec:symbols}

Associated with the block Laurent matrices $T_m$ and $T_{m,m'}$ we  have 
the symbols $A_m(z)$ and $A_{m,m'}(z)$. According to Assumption \ref{ass:blocksymbols} 
all symbols are analytic in an annular domain
$R_1 < |z| < R_2$. We  have the identity
\begin{equation} \label{eq:Amn} 
	A_{m,n}(z) = A_m(z) A_{m+1}(z) \cdots A_{m'-1}(z), \qquad m < m', 
	\end{equation}
for $R_1 < |z| < R_2$. The series that define the symbol do not commute
(in general) and thus the  order of the factors in the product is important.

We let $\gamma$ be a circle of radius $R \in (R_1, R_2)$ with counterclockwise
orientation. By Cauchy's theorem, we can recover the Laurent matrix entries
from the symbols, and we have
\begin{equation} \label{eq:Tmnblock} 
	\left[ T_{m,m'}(px+i, py+j) \right]_{i,j=0}^{p-1}  
	 = \frac{1}{2\pi i} \oint_{\gamma}
	A_{m,m'}(z)  z^{x-y} \frac{dz}{z}. \end{equation}
In particular 
\begin{align}  \label{eq:T0Lblock}
	\left[T_{0,L}(px+i, pM+py+j) \right]_{i,j=0}^{p-1} 
	& = \frac{1}{2\pi i} \oint_{\gamma}
	A_{0,L}(z)  z^{x-y-M} \frac{dz}{z}. 
	\end{align}
and if we restrict to $0 \leq x, y \leq N-1$ then the
blocks \eqref{eq:T0Lblock} are the $p \times p$ blocks 
in the Gram matrix $\mathbf{G}$, see \eqref{eq:Gramkernel}.

We consider the matrix-valued weight
\begin{equation} \label{eq:defWgen} 
	W_{0,L}(z) = \frac{A_{0,L}(z)}{z^{M+N}}, \qquad z \in \gamma 
	\end{equation}
on the contour $\gamma$. Clearly $W_{0,L}$ also depends on $M$ and $N$
but we do not include this in the notation.
It introduces a bilinear pairing between $p \times p$ matrix valued functions
\begin{equation} \label{eq:pairing} 
	\langle P, Q \rangle
	= \frac{1}{2\pi i} \oint_{\gamma} P(z) W_{0,L}(z) Q^t(z) dz 
	\end{equation}
where $Q^t$ denotes the matrix transpose (no complex conjugation).
The integral is taken entrywise, and so $\langle P, Q \rangle$
is again a $p \times p$ matrix.

A matrix valued function is a polynomial of degree $\leq d$
if all its entries are polynomials of degree $\leq d$.

For invertible matrices $\mathbf P$ and $\mathbf Q$ of size 
$pN \times pN$ we define
\begin{equation} \label{eq:defPhij}
\begin{pmatrix} P_0(z) \\ P_1(z) \\ \vdots \\ P_{N-1}(z) \end{pmatrix}
	= {\mathbf P} \begin{pmatrix} I_p \\ z I_p \\ \vdots \\ z^{N-1} I_p \end{pmatrix} 
	\end{equation}
and
\begin{equation} \label{eq:defPsij}
\begin{pmatrix} Q_0(z) \\ Q_1(z) \\ \vdots \\ Q_{N-1}(z) \end{pmatrix}
	= {\mathbf Q} \begin{pmatrix} z^{N-1} I_p \\ \vdots \\ z I_p \\ I_p \end{pmatrix}. 	
	\end{equation}
Then $P_j$ and $Q_j$ for $j=0, 1, \ldots, N-1$ are matrix
valued polynomials of degrees $\leq N-1$.

\begin{proposition} Let $\mathbf P$ and $\mathbf Q$ be invertible $pN \times pN$ matrices,
and let $P_j, Q_j$ be the matrix valued polynomials as in \eqref{eq:defPhij}
and \eqref{eq:defPsij}.  Let $\mathbf{G}$ be the Gram matrix from  \eqref{eq:Gramkernel}. 
The following are equivalent.
\begin{enumerate}
\item[\rm (a)] ${\mathbf G}^{-1} = {\mathbf Q}^t {\mathbf P}$
\item[\rm (b)] For each $j,k = 0,1, \ldots, N-1$
\begin{equation} \label{eq:PsijPhikbiorthogonal} 
	\frac{1}{2\pi i} \oint_{\gamma}
	P_j(z) W_{0,L}(z) Q_k^t(z) dz = \delta_{j,k} I_p 
	\end{equation}
where $W_{0,L}(z) = \ds \frac{A_{0,L}(z)}{z^{M+N}}$ as in \eqref{eq:defWgen}.
\end{enumerate}
\end{proposition}
\begin{proof}
Consider
\[ X =
	\frac{1}{2\pi i} \oint_{\gamma} 
\begin{pmatrix} P_0(z) \\ P_1(z) \\ \vdots \\ P_{N-1}(z) \end{pmatrix}
 W_{0,L}(z) 
\begin{pmatrix} Q_0^t(z) & Q_1^t(z) & \cdots & Q_{N-1}^t(z) \end{pmatrix} dz.
\]
Then, by \eqref{eq:defWgen} and the definition \eqref{eq:defPhij}--\eqref{eq:defPsij} 
of the matrix valued polynomials, 
\begin{multline*} 
{\mathbf P}^{-1} X ({\mathbf Q}^{-1})^t = 
\frac{1}{2\pi i}
		\oint_{\gamma}
\begin{pmatrix} I_p  \\ z I_p  \\ \vdots \\  z^{N-1} I_p \end{pmatrix}
  	\frac{A_{0,L}(z)}{z^{M+N}}  \begin{pmatrix} z^{N-1} I_p  &  \cdots & z I_p & I_p \end{pmatrix} dz  \\
	= \frac{1}{2\pi i}
		\oint_{\gamma} 
\begin{pmatrix} I_p  \\ z I_p  \\ \vdots \\  z^{N-1} I_p \end{pmatrix}
  \frac{A_{0,L}(z)}{z^M} 
	\begin{pmatrix} I_p  & z^{-1} I_p & \cdots & z^{-N+1} I_p \end{pmatrix}  
	\frac{dz}{z} 
\end{multline*}
This is a a block Toeplitz matrix with $p \times p$ blocks.
For $0 \leq x, y \leq N-1$, the  $xy$-th block is
\[ \frac{1}{2 \pi i} \oint_{\gamma} A_{0,L}(z) z^{x-y-M} \frac{dz}{z} \]
which by \eqref{eq:Gramkernel} and  \eqref{eq:T0Lblock} is equal to the
$xy$-th block of $\mathbf{G}$. Thus
\[ {\mathbf P}^{-1} X ({\mathbf Q}^{-1})^t = \mathbf{G}, \]
which means that ${ \mathbf G}^{-1} = {\mathbf Q}^t {\mathbf P}$ if and only if $X= I_{pN}$.
This proves the proposition, since $X= I_{pN}$ is equivalent to the biorthogonality
\eqref{eq:PsijPhikbiorthogonal}.
\end{proof}

The property \eqref{eq:PsijPhikbiorthogonal} is a  matrix valued
biorthogonality between the two sequences $(P_j)_{j=0}^{N-1}$ 
and $(Q_j)_{j=0}^{N-1}$. The matrix valued biorthogonal polynomials
are clearly not unique but depend on the particular factorization of $\mathbf{G}^{-1}$.

\subsection{Reproducing kernel} \label{subsec:repkernel}
Let ${\mathbf G}^{-1} = {\mathbf Q}^t {\mathbf P}$ be any factorization of ${\mathbf G}^{-1}$
and let $P_j$, $Q_j$ be the matrix polynomials as in \eqref{eq:defPhij}
and \eqref{eq:defPsij}.
We consider
\begin{equation} 
\label{eq:RNwz} R_N(w,z) = \sum_{j=0}^{N-1} Q^t_j(w) P_j(z) 
\end{equation}
which is a bivariate polynomial of degree $\leq N-1$ in both $w$ and $z$.

\begin{lemma} \ \label{lem:RNkernel}
\begin{enumerate}
\item[\rm (a)]
For every matrix valued polynomial $P$ of degree $\leq N-1$ we have
\begin{equation} \label{eq:repkernel1}
	\frac{1}{2\pi i} \oint_{\gamma} P(w) W_{0,L}(w) R_N(w,z) dw = P(z). 	
	\end{equation}
\item[\rm (b)] 
For every matrix valued polynomial $Q$ of degree $\leq N-1$ we have
\begin{equation} \label{eq:repkernel2}
	\frac{1}{2\pi i} \oint_{\gamma} R_N(w,z) W_{0,L}(z) Q^t(z) dz = Q^t(w). 	
	\end{equation}
\item[\rm (c)] Either one of the properties (a) and (b) characterizes 
\eqref{eq:RNwz} in the sense that if a bivariate polynomial $\widehat{R}_N(w,z)$
of degree $\leq N-1$ in both $w$ and $z$ satisfies either (a) or (b),
then $\widehat{R}_N(w,z) = R_N(w,z)$ for every $w, z \in \mathbb C$.
\end{enumerate}
\end{lemma}
\begin{proof}
Parts (a) and (b) are immediate from the biorthogonality \eqref{eq:PsijPhikbiorthogonal}
and the fact that any matrix valued polynomials $P$ of degree $\leq N-1$ can be written as
$P(z) = \ds  \sum_{k=0}^{N-1} A_k P_k(z) = \sum_{k=0}^{N-1} B_k Q_k(z)$
for suitable constant matrices $A_0, \ldots, A_{N-1}, B_0, \ldots, B_{N-1}$.

\medskip

Let $\widehat{R}_N(w,z)$ be as in part (c), and suppose that 
\begin{equation} \label{eq:repkernelhat} 
	\frac{1}{2\pi i} \oint_{\gamma} \widehat{R}_N(w,z) W_{0,L}(z) Q^t(z) dz = Q^t(w) 
	\end{equation}
for every matrix valued polynomial $Q$ of degree $\leq N-1$.
For a fixed $w$ we note that $z \mapsto \widehat{R}_N(w,z)$ is a matrix
valued polynomial of degree $\leq N-1$ and it can be written as a linear
combination of $P_0, \ldots, P_{N-1}$ with matrix coefficients. The matrix
coefficients depend on $w$, and we get for some $A_j(w)$, $j=0, \ldots, N-1$,
\begin{equation} \label{eq:RNhatexpansion} 
	\widehat{R}_N(w,z) = \sum_{j=0}^{N-1} A_j(w) P_j(z). \end{equation}
Then taking \eqref{eq:repkernelhat} with $Q = Q_k$, and using the
biorthogonality \eqref{eq:PsijPhikbiorthogonal}, we get
\[
Q_k^t(w) = 
	\sum_{j=0}^{N-1} A_j(w) 
	\frac{1}{2\pi i} \oint_{\gamma} P_j(z)  W_{0,L}(z) Q_k^t(z) dz
	= A_k(w) \]
for every $k = 0, \ldots, N-1$. Thus $\widehat{R}_N(w,z) = R_N(w,z)$
by \eqref{eq:RNhatexpansion} and \eqref{eq:RNwz}. This proves part (c) in case the 
reproducing property of (b) is satisfied. The other case follows similarly.
\end{proof}

Because of  \eqref{eq:repkernel1} and \eqref{eq:repkernel2} 
we call $R_N(w,z)$ the reproducing kernel for the pairing
\eqref{eq:pairing}.

From part (c) of Lemma \ref{lem:RNkernel} we find in particular that
the sum in \eqref{eq:RNwz} does not depend on the particular choice of
factorization ${\mathbf G}^{-1} = {\mathbf Q}^t {\mathbf P}$. It does depend
on $\mathbf G$ as can be also seen from  the calculation using 
 \eqref{eq:defPhij}, \eqref{eq:defPsij}, and \eqref{eq:RNwz} 
\begin{align} \nonumber
	R_N(w,z) & =  \begin{pmatrix} w^{N-1} I_p & \cdots &w I_p & I_p \end{pmatrix} 
			{\mathbf Q}^t {\mathbf P}  \begin{pmatrix} I_p \\ z I_p \\ \vdots \\ z^{N-1} I_p \end{pmatrix} 	\\
	& = \label{eq:RNwz2} 
		  \begin{pmatrix} w^{N-1} I_p & \cdots &w I_p & I_p \end{pmatrix} 
		{\mathbf G}^{-1} \begin{pmatrix} I_p \\ z I_p \\ \vdots \\ z^{N-1} I_p \end{pmatrix} 	
\end{align}
which clearly only depends on $\mathbf G$.

\subsection{Main theorem} \label{subsec:mainthm}
Now we are ready for the main theorem in this section.

\begin{theorem} \label{thm:Kernelblock}
Assume the transition matrices $T_m$ are $p$-periodic and that
the above Assumptions \ref{ass:multilevel} and \ref{ass:blocksymbols} are satisfied.
Then the multi-level particle system  \eqref{eq:multilevel} 
is  determinantal with correlation
kernel $K$ given by
\begin{multline} \label{eq:Kcontour}
	\left[ K(m, px+j; m', py+i) \right]_{i,j=0}^{p-1}
	= - \frac{\chi_{m > m'}}{2\pi i} \oint_{\gamma}
				A_{m',m}(z) z^{y-x} \frac{dz}{z},
		\\
		+  \frac{1}{(2\pi i)^2}
		\oint_{\gamma} \oint_{\gamma}
		A_{m',L}(w)  R_N(w,z) A_{0,m}(z)  \frac{w^{y}}{z^{x+1} w^{M+N}} dz dw \\
				\qquad x,y \in \mathbb Z, \, 0 < m, m' < L,
\end{multline}
where $R_N(w,z)$ is the reproducing kernel \eqref{eq:RNwz} built out of
matrix valued biorthogonal polynomials associated with the weight
$W_{0,L}(z) = \frac{A_{0,L}(z)}{z^{M+N}}$ on $\gamma$.
\end{theorem}

\begin{proof}
We already know that the particle system is determinantal with 
kernel given by \eqref{eq:EMkernel}.

The first term in \eqref{eq:EMkernel} gives rise to the first term on 
the right-hand side of \eqref{eq:Kcontour} in view of  \eqref{eq:Tmnblock}
(note that $x$ and $y$ are interchanged in $T_{m',m}(y,x)$ in \eqref{eq:EMkernel}).

Let $K_0(m, x; m',y)$ be the second term in the right-hand side of 
\eqref{eq:EMkernel}. Instead of summation indices $i$ and $j$ in the
double sum we  use $p\nu + k$ and $p\nu' +k'$, with $0 \leq \nu, \nu' \leq N-1$,
$0 \leq k, k' \leq p-1$. Then from \eqref{eq:EMkernel},
\begin{multline}
	K_0(m,px+j; m', py+i) 
		= \\ \sum_{\nu, \nu'=0}^{N-1} 
		 \sum_{k, k' = 0}^{p-1} T_{m',L}(py+i, pM + p\nu' + k')
			\left[ {\mathbf G}^{-1} \right]_{p\nu' + k', p\nu+k}
		T_{0,m} (p\nu+k, px+j).
\end{multline}
Using \eqref{eq:Tmnblock} we can  write this in block form
\begin{multline} \label{eq:K0blockfactor}
	\left[ K_0(m,px+j;m', py+i) \right]_{i,j=0}^{p-1}
	=  \\
	\left( \frac{1}{2\pi i} \oint_{\gamma}
		A_{m',L}(w) w^{y-M-\nu'} \frac{dw}{w} \right)_{0 \leq \nu' \leq N-1}
		{\mathbf G}^{-1}
	\left( \frac{1}{2\pi i} \oint_{\gamma}
		A_{0,m}(z) z^{\nu-x} \frac{dz}{z} \right)^t_{0 \leq \nu \leq N-1} 
\end{multline}		
where the first factor in the right-hand side of \eqref{eq:K0blockfactor}
is a block row vector of length $N$ with $p \times p$ blocks,
and the last factor is a similar block column vector.
We combine the integrals to obtain
a double integral and then we use \eqref{eq:RNwz2} to find
 \begin{multline*}
	\left[ K_0(m,px+j; m', py+i) \right]_{i,j=0}^{p-1}
	\\
	 = \frac{1}{(2\pi i)^2} \oint_{\gamma} \oint_{\gamma}
		A_{m',L}(w)
		\begin{pmatrix} w^{N-1} I_p & \cdots & w I_p & I_p \end{pmatrix}
			 \\
		\times 
	{\mathbf G}^{-1} \begin{pmatrix} I_p \\ zI_p \\ \vdots \\ z^{N-1} I_p \end{pmatrix}
		A_{0,m}(z) w^{y-M-N} z^{-x-1} dz dw \\
	= 
	 \frac{1}{(2\pi i)^2} \oint_{\gamma} \oint_{\gamma}
		A_{m',L}(w) R_N(w,z) A_{0,m}(z) w^{y-M-N} 
	 z^{-x-1} dz dw 
\end{multline*}
which completes the proof.
\end{proof}

\subsection{Matrix valued orthogonal polynomials} \label{subsec:MVOP}
The goal of this subsection and the next is to express the reproducing kernel
 \eqref{eq:RNwz} (or \eqref{eq:RNwz2}) in terms of matrix valued orthogonal polynomials
 (MVOP) and use a Christoffel-Darboux formula for the sum \eqref{eq:RNwz}.
Such a formula is known for MVOP in various forms \cite{DPS}, \cite{GIM}, \cite{AAGMM}, \cite{AM}.
We are however in a non-standard situation, with a non-Hermitian orthogonality,
and the MVOP need not exist for every degree. Fortunately, the degree $N$ MVOP
will exist in the present situation, as we will explain in this subsection.

If we can find a factorization ${\mathbf G}^{-1} = {\mathbf Q}^t {\mathbf P}$ leading to 
matrix valued polynomials \eqref{eq:defPhij}--\eqref{eq:defPsij} 
with $P_j(z) = Q_j(z)$ and $\deg P_j = j$ for every $j=0,1, \ldots, N-1$, 
then we would have a finite sequence of MVOP that are in fact orthonormal
\begin{equation} \label{eq:MVorthogonal} 
	\frac{1}{2\pi i} \oint_{\gamma}
	P_j(z) W_{0,L}(z) P_k^t(z) dz = \delta_{j,k} I_p,
	\qquad j,k=0, \ldots, N-1, \end{equation}
and then
\begin{equation} \label{eq:RinMVOP0} 
	R_N(w,z) = \sum_{j=0}^{N-1} P^t_j(w) P_j(z). 
	\end{equation}
From \eqref{eq:RinMVOP0} it would follow that $R_N(z,w) = R_N(w,z)^t$ and this is
an identity that is not necessarily satisfied.
Thus we cannot expect that the orthonormal MVOP exist.

Instead we focus on monic MVOP. If the monic MVOP $P_j$ exists for every degree
$j$, then we have
\begin{equation} \label{eq:monicnorm}
 \frac{1}{2\pi i} \oint_{\gamma} P_j(z) W_{0,L}(z) P_k^t(z) dz = \delta_{j,k} H_j,
	\qquad j, k = 0, 1, \ldots, N-1, \end{equation}
with some matrix $H_j$. If $H_j$ is invertible for every degree then
the two sequences of matrix valued polynomials  $(H_j^{-1} P_j(z))_{j=0}^{N-1}$
and $(P_j(z))_{j=0}^{N-1}$ are biorthogonal, and thus 
\begin{equation} \label{eq:RinMVOP} 
	R_N(w,z) = \sum_{j=0}^{N-1} P^t_j(w) H_j^{-1} P_j(z). 
	\end{equation}

The orthogonality \eqref{eq:MVorthogonal} is non-Hermitian orthogonality,
and it is not associated with a positive definite scalar product. Also existence
and uniqueness of the monic MVOP is not guaranteed in general.
However, the  MVOP of degree $N$ does exists, and this is a consequence of 
the fact that $\mathbf G$ is invertible.

\begin{lemma} \label{lem:PNexists}
There is a unique monic matrix valued polynomial 
\[ P_N(z) = z^N I_p + \cdots \]
of degree $N$ such that
\begin{equation} \label{eq:PNMVOP} 
	\frac{1}{2 \pi i} \oint_{\gamma} P_N(z) W_{0,L}(z) z^k dz = 0_p, \qquad 
	k = 0, 1, \ldots, N-1. \end{equation} 
\end{lemma}
\begin{proof}
The conditions \eqref{eq:PNMVOP} give us $p^2N$ linear equations for the $p^2N$
unknown coefficients of a monic matrix valued polynomial of degree $N$. The linear 
system has matrix $\mathbf G$, provided we number the coefficients and the 
conditions appropriately, and since $\mathbf G$ is invertible, the existence and 
uniqueness of $P_N$ follows.

More explicitly, write $\ds P_N(z) = I_p z^N  + \sum_{j=0}^{N-1} C_j z^j$
with $p \times p$ matrices $C_j$ that are to be determined. The orthogonality
conditions
\[ \frac{1}{2\pi i} \oint_{\gamma} P_N(z) W_{0,L}(z) z^{N-1-k} dz = 0,
	\qquad  k=0, 1, \ldots, N-1, \]
yield
\[ \sum_{j=0}^{N-1}   C_j 
	\frac{1}{2\pi i} \oint_{\gamma} W_{0,L}(z) z^{N+j-k} \frac{dz}{z}
		= - \frac{1}{2 \pi i} \oint_{\gamma} W_{0,L}(z) z^{2N-k} \frac{dz}{z}, \]
for $k=0,1, \ldots, N-1$.
Since $W_{0,L}(z) = \frac{A_{0,L}(z)}{z^{M+N}}$ the left hand side is
\[ \sum_{j=0}^{N-1}   C_j 
	\frac{1}{2\pi i} \oint_{\gamma} \frac{A_{0,L}(z)}{z^M} z^{j-k} \frac{dz}{z}
	=
	\sum_{j=0}^{N-1} C_j G_{j,k}	 \]
where $G_{j,k}$ denotes the $j,k$th block of the block Toeplitz matrix $\mathbf G$,
see also \eqref{eq:T0Lblock}.
Varying $k = 0, 1, \ldots, N-1$, we see that 
\[ \begin{pmatrix} C_0 & \ldots & C_{N-1} \end{pmatrix} {\mathbf G}
	=  - \frac{1}{2 \pi i} \oint_{\gamma}  W_{0,L}(z) z^{2N} 
		\begin{pmatrix} I_p & z^{-1} I_p & \cdots & z^{-N+1} I_p \end{pmatrix} 
		\frac{dz}{z}. \]
The matrix $\mathbf G$ is invertible, and thus the matrix coefficients
$C_0, \ldots, C_{N-1}$ are uniquely determined, and the monic MVOP of
degree $N$ exists uniquely. 
\end{proof}

\subsection{Riemann-Hilbert problem and Christoffel-Darboux formula} \label{subsec:RHP}
The MVOP of degree $N$ is characterized by a Riemann-Hilbert problem of 
size $2p \times 2p$.
The RH problem asks for a  $2p \times 2p$ matrix valued function 
$Y : \mathbb C \setminus \gamma \to \mathbb C^{2p \times 2p}$ satisfying
\begin{itemize}
\item $Y$ is analytic,
\item $Y_+ = Y_- \begin{pmatrix} I_p & W_{0,L} \\ 0_p & I_p \end{pmatrix}$
on $\gamma$ with counterclockwise orientation,
\item $Y(z) = (I_{2p} + O(z^{-1})) 
	\begin{pmatrix} z^N I_p & 0_p \\ 0_p & z^{-N} I_p \end{pmatrix}$ as $z \to \infty$.
\end{itemize}
In the scalar valued case, i.e.\ $p=1$, the RH problem is 
due to Fokas, Its, and Kitaev \cite{FIK}. The matrix valued 
extension can be found in \cite{CM,D,GIM}. It is similar
to the RH problem for multiple orthogonal polynomials
\cite{vAGK}. 

The RH problem has a unique solution, since by Lemma \ref{lem:PNexists} the monic MVOP of degree $P_N$
exists and is unique. The solution is
\begin{equation} \label{eq:solY} 
	Y(z) =
	 \begin{pmatrix} P_N(z) & \ds \frac{1}{2\pi i} \oint_{\gamma} \frac{P_N(s) W_{0,L}(s)}{s-z} ds \\[10pt]
	 	Q_{N-1}(z) & \ds \frac{1}{2\pi i} \oint_{\gamma} \frac{Q_{N-1}(s) W_{0,L}(s)}{s-z} ds
	 	\end{pmatrix}, \quad z \in \mathbb C \setminus \gamma, 
	 	\end{equation}
where $Q_{N-1}$ is a matrix valued polynomial of degree $\leq N-1$ such that
\begin{equation} \label{eq:QNMVOP} 
	\frac{1}{2 \pi i} \oint_{\gamma} Q_{N-1}(z) W_{0,L}(z) z^k dz = 
	\begin{cases} 0_p, & 	k = 0, 1, \ldots, N-2, \\
		- I_p, & k = N-1. \end{cases} 
		\end{equation}
One can show that $Q_{N-1}$ also uniquely exists, since the conditions 
\eqref{eq:QNMVOP} give a system of $p^2N$ linear equations for the $p^2N$
coefficients of $Q_{N-1}$, and the matrix of this system can be identified 
with $\mathbf G$.
Since $\mathbf G$ is invertible, there is a unique solution. If the leading
coefficient of $Q_{N-1}$ would be invertible (which is typically the
case, but it is not guaranteed in general) 
then the monic MVOP $P_{N-1}$ 
of degree $N-1$ would exist as well and $Q_{N-1} = - H_{N-1}^{-1} P_{N-1}$,
where $H_{N-1}$ is as in \eqref{eq:monicnorm}.

In the following result we express the reproducing kernel $R_N(w,z)$
in terms of the solution of the RH problem. It can be viewed as a 
Christoffel-Darboux formula and in this form it is due to 
Delvaux \cite{D}. It is similar to the 
Christoffel-Darboux formulas for  
multiple orthogonal polynomials \cite{BK1,DK} which also
use the RH problem.
  
\begin{proposition} \label{prop:CDformula}
We have
\begin{equation} \label{eq:CDformula} 
	R_N(w,z) = 
	\frac{1}{z-w} \begin{pmatrix} 0_p & I_p \end{pmatrix}
			Y^{-1}(w)  Y(z) \begin{pmatrix} I_p \\ 0_p \end{pmatrix}.
			\end{equation}
\end{proposition}
\begin{proof}
This is due to Delvaux \cite[Proposition 1.10]{D}, see also \cite{GIM}.
Since the context and notation of \cite{D} is somewhat different 
from the present setting, we give an outline of the proof.

The right-hand side of \eqref{eq:CDformula} is a bivariate
polynomial in $z$ and $w$ of degrees $\leq N-1$ in both variables,
see Lemma 2.3 in \cite{D} for details. 
We show that it satisfies the reproducing kernel property from 
Lemma \ref{lem:RNkernel} (b),
and we follow the proof of \cite[Proposition 2.4]{D}.

Let $Q$ be a matrix valued polynomial of degree $\leq N-1$.
We replace $R_N(w,z)$ by the right-hand side of \eqref{eq:CDformula} in
the integral on the left of \eqref{eq:repkernel2} and we obtain
\eqref{eq:solY} we find
\[ \begin{pmatrix} 0 & I_p \end{pmatrix} Y^{-1}(w)
	 \frac{1}{2\pi i} \oint_{\gamma} 
	 \begin{pmatrix} P_N(z) \\ Q_{N-1}(z) \end{pmatrix} W_{0,L}(z) Q^t(z)
	 \frac{dz}{z-w}
	 \]
which is equal to
\begin{multline}
 \begin{pmatrix} 0 & I_p \end{pmatrix} Y^{-1}(w)
	  \frac{1}{2\pi i} \oint_{\gamma} 
	 \begin{pmatrix} P_N(z) \\ Q_{N-1}(z) \end{pmatrix} W_{0,L}(z)
	 \frac{Q^t(z)-Q^t(w)}{z-w} dz \\
 + \begin{pmatrix} 0 & I_p \end{pmatrix} Y^{-1}(w)
	\left( \frac{1}{2\pi i} \oint_{\gamma}  	
 	\begin{pmatrix} 
	  P_N(z) \\ 
	  Q_{N-1}(z) \end{pmatrix} W_{0,L}(z) \frac{dz}{z-w} \right)
	 Q^t(w). \label{eq:CDproof1}
 \end{multline}

For every $w$ we have that $z \mapsto \frac{Q^t(z) - Q^t(w)}{z-w}$
is a matrix valued polynomial of degree $\leq N-2$. 
The first term in \eqref{eq:CDproof1}  then vanishes because
of the orthogonality conditions \eqref{eq:PNMVOP} and \eqref{eq:QNMVOP} satisfied by $P_N$ and $Q_{N-1}$.
The second term contains the second column of $Y$, see 
\eqref{eq:solY}, and therefore \eqref{eq:CDproof1} is equal to
\[ \begin{pmatrix} 0 & I_p \end{pmatrix} Y^{-1}(w)
	Y(w) \begin{pmatrix} 0 \\ I_p \end{pmatrix} Q^t(w) \]
and this is indeed $Q^t(w)$, which proves that the right-hand side
of \eqref{eq:CDformula} has the reproducing property of Lemma \ref{lem:RNkernel} (b)
that characterizes $R_N(w,z)$ by Lemma \ref{lem:RNkernel} (c). The proposition follows.
\end{proof}

We insert \eqref{eq:CDformula} into formula
\eqref{eq:Kcontour} and find a convenient formula for the correlation
kernel in terms of the solution of the RH problem. 

A possible asymptotic analysis of the kernel would consist of two parts.
First we do an analysis of the RH problem that
would give us the asymptotic behavior of the kernel \eqref{eq:CDformula}.
Then this is followed by an asymptotic analysis of the
double contour integral in \eqref{eq:Kcontour} by means of classical
methods of steepest descent. We are able to this for the 
two periodic Aztec diamond.

\subsection{Example 1: Aztec diamond} \label{subsec:uniformAD}
Let us put a weight $1$ on a horizontal domino and a weight $q$
on a vertical domino in a tiling of the Aztec diamond of size $N$. 
This corresponds to putting weights $1$
on the horizontal edges in the Aztec diamond graph, see Figure
\ref{fig:AztecGraphWeights}, and weight $q$ on the diagonal and vertical edges.

Assumption \eqref{ass:multilevel} holds with $L=N$, $M=0$, and
transition matrices that are independent of $m$, 
\[ T(x,y) = T_m(x,y) = \begin{cases}
	  0, & y \geq x+2, \\
	  q, & y = x+1, \\
	  q^{x-y} + q^{2+x-y}, & y \leq x.
	  \end{cases} \]
Then $T$ is a Laurent matrix and the symbol is
\[ A(z) = qz + \sum_{j=-\infty}^{0} (1+q^2) q^{-j} z^j
	= \frac{z(qz+1)}{z-q}, \qquad |z| > q. \] 
Hence,
\[ A_{0,L}(z) = A^L(z) = \left( \frac{z(qz+1)}{z-q} \right)^L, \]
and since $L = N$ and $M=0$, it follows that
\begin{equation} \label{eq:AztecW1}
	W_{0,L}(z) = \frac{A_{0,L}(z)}{z^{M+N}} = \left( \frac{qz+1}{z-q} \right)^N. 
	\end{equation}
The (scalar) weight is rational with a pole at $z=q$ and a zero at $z=-1/q$,
both of order $N$.
The orthogonal polynomials are properly rescaled Jacobi polynomials, 
but with parameters $-N, N$.

The Jacobi polynomial of degree $k$ with parameters
$\alpha$ and $\beta$ may be given by the Rodrigues formula
\begin{equation} \label{eq:PkJacobi} 
	P_k^{(\alpha, \beta)}(x)
	= \frac{1}{2^k k!} (x-1)^{-\alpha} (x+1)^{-\beta}
		\frac{d^k}{dx^k} \left[ (x-1)^{k+\alpha} (x+1)^{k+\beta} \right]. 
		\end{equation}
see e.g.\ \cite[Chapter IV]{S}.
The Jacobi polynomial is usually considered with parameters $\alpha, \beta > -1$,
but the formula \eqref{eq:PkJacobi} makes sense for arbitrary parameters, and it always
gives a polynomial of degree $\leq k$. There is a reduction in the degree if and only if
$\alpha + \beta \in \{-k-1, \ldots, -2k \}$.

If $\alpha$ and $\beta$ are integers, and $\gamma$ is a circle not going through
$\pm 1$,
then for $j=0,1, \ldots$, we find by using \eqref{eq:PkJacobi} and after integrating by parts $k$ times
\begin{multline*} 
	\frac{1}{2\pi i} \oint_{\gamma} P_k^{(\alpha,\beta)}(z)
	(z-1)^{\alpha} (z+1)^{\beta} z^j dz \\
	 = \frac{1}{2^k k!} 	\frac{1}{2\pi i} \oint_{\gamma} 
		\frac{d^k}{dz^k} \left[ (z-1)^{k+\alpha} (z+1)^{k+\beta} \right] z^j
		dz \\
	 = \frac{(-1)^k}{2^k k!} 	\frac{1}{2\pi i} \oint_{\gamma} 
		(z-1)^{k+\alpha} (z+1)^{k+\beta} \left( \frac{d^k}{dz^k} z^j \right)	dz 
		\end{multline*}
Thus 
\begin{align} \label{eq:Jacobiorthogonal}
	\frac{1}{2\pi i} \oint_{\gamma} P_k^{(\alpha,\beta)}(z)
	(z-1)^{\alpha} (z+1)^{\beta} z^j dz = 0, \qquad j =0, 1, \ldots, k-1.
	\end{align}
	
The weight \eqref{eq:AztecW1} has a pole at $z=q$ and a zero at $z=-1/q$. By
an affine change of variables we  map these to $\pm 1$, and then we use
\eqref{eq:Jacobiorthogonal} with $\alpha = -N$ and $\beta=N$. 
It follows that
\[ \frac{1}{2\pi i} \oint_{\gamma} P_k^{(-N,N)} \left( \frac{2z-q + q^{-1}}{q+q^{-1}} \right)
	 W_{0,L}(z) z^j dz = 0, \qquad j =0,1, \ldots, k-1. \]
Thus the orthogonal polynomials for the weight \eqref{eq:AztecW1}
are the rescaled Jacobi polynomials
\[ P_k^{(-N,N)} \left( \frac{2z-q + q^{-1}}{q+q^{-1}} \right). \]	
For $k=N$ there is a further reduction since
\begin{equation} \label{eq:PNuniformAD} 
	P_N^{(-N,N)}(z) = \frac{1}{2^N} \binom{2N}{N} (z-1)^N.
	\end{equation}
	
The uniform Aztec diamond has been analyzed in detail 
with Krawtchouk polynomials in \cite{EKLP, J01, J05}, which are 
discrete orthogonal polynomials. We find it intriguing that
an alternative approach with Jacobi polynomials seems also
possible.

\subsection{Example 2: Hexagon tiling} \label{subsec:hexagon}
Another popular model are lozenge tilings of a hexagon. A lozenge tiling
is equivalent to a system of $N$ non-intersecting paths on the graph
with vertex set $\mathbb Z^2$ and directed edges from $(m,n)$ to $(m',n')$
if and only if $m' = m+1$ and $n' \in \{n,n+1\}$. 
The starting positions are $(0,0), \ldots, (0,N-1)$
and ending positions at $(L,M), \ldots, (L, M+N-1)$, where
the parameters $L,M,N$ are non-negative integers with $M \leq L$.
In Figure \ref{fig:PathsOnGraph2} we have $N=6$, $M=3$ and $L=8$.

\begin{figure}[t]
\begin{center}
\begin{tikzpicture}[scale=0.8]

\foreach \y in {-5,-4,-3,-2,-1,0,1,2,3}
\draw (-6,\y-0.5) -- (5,\y-0.5);

\foreach \y in {-5,-4,-3,-2,-1,0,1,2}
{ \foreach \x in {-5,-4,-3,-2,-1,-0,1, 2,3,4}
\draw (\x-0.5,\y-0.5)--(\x+0.5,\y+0.5);
}

\foreach \x in {0,1,2,3,4,5,6,7,8}
\draw (\x/2+\x/2-4.5,-6) node {\x};
\foreach \y in {0,1,2,3,4,5,6,7,8}
\draw (-6.5,\y-5.5) node {\y};

\fill (-4.5,-0.5) circle [radius = 2mm]; 
\fill (-4.5,-1.5) circle [radius = 2mm]; 
\fill (-4.5,-2.5) circle [radius = 2mm];
\fill (-4.5,-3.5) circle [radius = 2mm]; 
\fill (-4.5,-4.5) circle [radius = 2mm]; 
\fill (-4.5,-5.5) circle [radius = 2mm];  

\fill (-3.5,0.5) circle [radius = 2mm]; 
\fill (-3.5,-0.5) circle [radius = 2mm]; 
\fill (-3.5,-2.5) circle [radius = 2mm];
\fill (-3.5,-3.5) circle [radius = 2mm]; 
\fill (-3.5,-4.5) circle [radius = 2mm]; 
\fill (-3.5,-5.5) circle [radius = 2mm];  

\fill (-2.5,0.5) circle [radius = 2mm]; 
\fill (-2.5,-0.5) circle [radius = 2mm]; 
\fill (-2.5,-1.5) circle [radius = 2mm];
\fill (-2.5,-3.5) circle [radius = 2mm]; 
\fill (-2.5,-4.5) circle [radius = 2mm]; 
\fill (-2.5,-5.5) circle [radius = 2mm]; 

\fill (-1.5,1.5) circle [radius = 2mm]; 
\fill (-1.5,-0.5) circle [radius = 2mm]; 
\fill (-1.5,-1.5) circle [radius = 2mm];
\fill (-1.5,-2.5) circle [radius = 2mm]; 
\fill (-1.5,-3.5) circle [radius = 2mm]; 
\fill (-1.5,-5.5) circle [radius = 2mm]; 

\fill (-0.5,1.5) circle [radius = 2mm]; 
\fill (-0.5,0.5) circle [radius = 2mm]; 
\fill (-0.5,-1.5) circle [radius = 2mm];
\fill (-0.5,-2.5) circle [radius = 2mm]; 
\fill (-0.5,-3.5) circle [radius = 2mm]; 
\fill (-0.5,-4.5) circle [radius = 2mm]; 

\fill (0.5,1.5) circle [radius = 2mm]; 
\fill (0.5,0.5) circle [radius = 2mm]; 
\fill (0.5,-1.5) circle [radius = 2mm];
\fill (0.5,-2.5) circle [radius = 2mm]; 
\fill (0.5,-3.5) circle [radius = 2mm]; 
\fill (0.5,-4.5) circle [radius = 2mm]; 
 
\fill (1.5,1.5) circle [radius = 2mm]; 
\fill (1.5,0.5) circle [radius = 2mm]; 
\fill (1.5,-0.5) circle [radius = 2mm];
\fill (1.5,-2.5) circle [radius = 2mm]; 
\fill (1.5,-3.5) circle [radius = 2mm]; 
\fill (1.5,-4.5) circle [radius = 2mm]; 

\fill (2.5,2.5) circle [radius = 2mm]; 
\fill (2.5,1.5) circle [radius = 2mm]; 
\fill (2.5,0.5) circle [radius = 2mm];
\fill (2.5,-1.5) circle [radius = 2mm]; 
\fill (2.5,-2.5) circle [radius = 2mm]; 
\fill (2.5,-3.5) circle [radius = 2mm]; 

\fill (3.5,2.5) circle [radius = 2mm]; 
\fill (3.5,1.5) circle [radius = 2mm]; 
\fill (3.5,0.5) circle [radius = 2mm];
\fill (3.5,-0.5) circle [radius = 2mm]; 
\fill (3.5,-1.5) circle [radius = 2mm]; 
\fill (3.5,-2.5) circle [radius = 2mm]; 

\draw [line width = 1.5mm] (-4.5,-0.5)--(-3.5,0.5)--(-2.5,0.5)--(-1.5,1.5)--(-0.5,1.5)--(0.5,1.5)--(1.5,1.5)--(2.5,2.5)--(3.5,2.5);
\draw [line width = 1.5mm] (-4.5,-1.5)--(-3.5,-0.5)--(-2.5,-0.5)--(-1.5,-0.5)--(-0.5,0.5)--(0.5,0.5)--(1.5,0.5)--(2.5,1.5)--(3.5,1.5);
\draw [line width = 1.5mm] (-4.5,-2.5)--(-3.5,-2.5)--(-2.5,-1.5)--(-1.5,-1.5)--(-0.5,-1.5)--(0.5,-1.5)--(1.5,-0.5)--(2.5,0.5)--(3.5,0.5);
\draw [line width = 1.5mm] (-4.5,-3.5)--(-3.5,-3.5)--(-2.5,-3.5)--(-1.5,-2.5)--(-0.5,-2.5)--(0.5,-2.5)--(1.5,-2.5)--(2.5,-1.5)--(3.5,-0.5);
\draw [line width = 1.5mm] (-4.5,-4.5)--(-3.5,-4.5)--(-2.5,-4.5)--(-1.5,-3.5)--(-0.5,-3.5)--(0.5,-3.5)--(1.5,-3.5)--(2.5,-2.5)--(3.5,-1.5);
\draw [line width = 1.5mm] (-4.5,-5.5)--(-3.5,-5.5)--(-2.5,-5.5)--(-1.5,-5.5)--(-0.5,-4.5)--(0.5,-4.5)--(1.5,-4.5)--(2.5,-3.5)--(3.5,-2.5);

\end{tikzpicture}

\caption{Non-intersecting paths on the graph for lozenge tilings of a hexagon.
\label{fig:PathsOnGraph2}}
\end{center}
\end{figure}
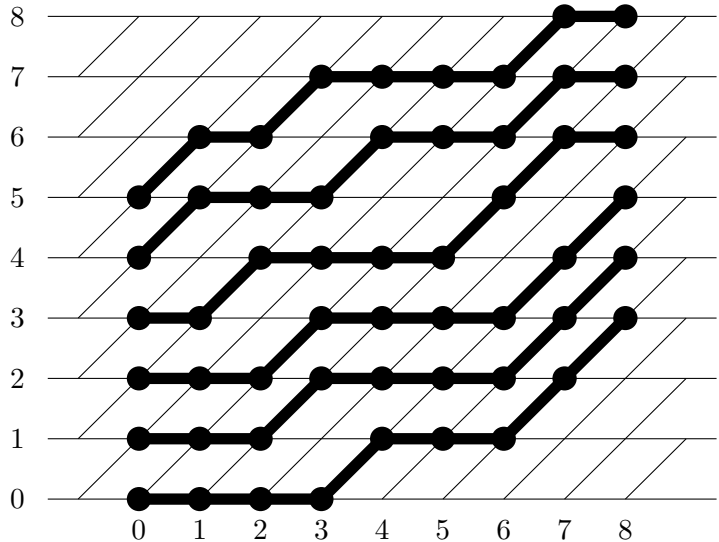

Consider the uniform case where each edge in the graph has weight $1$.
Then we are again in the situation of Assumption \ref{ass:multilevel} 
and at each level $m$ we have the transition matrices
\[ T_m(x,y) = \begin{cases} 1, & y = x \\
		1, & y = x+1 \\
		0, & \text{otherwise}. \end{cases} \]
that is independent of $m$. It is a Laurent matrix with symbol
$A(z) = z + 1$.
Then $A_{0,L}(z) = (z+1)^L$ and the weight is
\begin{equation} \label{eq:W0Lhexagon} 
	W_{0,L}(z) = \frac{A_{0,L}(z)}{z^{M+N}} = \frac{(z+1)^L}{z^{M+N}}. \end{equation}
The scalar weight is again rational with one zero and one pole. 

The orthogonal polynomials are again rescaled Jacobi polynomials, but
now with parameters $-(M+N)$ and $L$, namely for $0 \leq k < N$,
\begin{equation} \label{eq:HexagonJacobi} 
	 P_k^{(-M-N,L)}(2z+1).
\end{equation}
An asymptotic analysis of the lozenge tilings of the hexagon based on 
the Jacobi polynomials \eqref{eq:HexagonJacobi} 
seems possible, but has not been pursued yet.
See \cite{KM,MMO,MO} for an asymptotic analysis of 
Jacobi polynomials with varying non-standard parameters.

\subsection{Example 3: Aztec diamond with periodic weights}
\label{subsec:periodicAD}
	
The third example is the main interest of this paper: the two periodic
Aztec diamond of size $2N$. 

We saw in Section \ref{sec:paths} that the model gives rise to
the multi-level particle system \eqref{eq:multilevelAD}. This 
satisfies the Assumption \ref{ass:multilevel} if we take $p=2$
and $L=2N$ and $M=0$. The transition matrices are independent of $m$,
see \eqref{eq:defTmmp} and the matrix symbol is given by \eqref{eq:defA2}.
The weight matrix is $W_{0,L}(z) = \frac{A_{0,L}(z)}{z^{M+N}}
	= W^N(z)$ with
\begin{align} \nonumber 
	W(z) & = \frac{A^2(z)}{z}
	= \frac{1}{z(z-1)^2} \begin{pmatrix} 2\alpha z & \alpha(z+1) \\ 
	 	\beta z(z+1)& 2 \beta z \end{pmatrix}^2 \\
	& = \frac{1}{(z-1)^2} 
		\begin{pmatrix} (z+1)^2 + 4 \alpha^2 z & 2 \alpha(\alpha + \beta)(z+1) \\
	    2 \beta (\alpha + \beta)z(z+1) & (z+1)^2 + 4 \beta^2 z 
	    \end{pmatrix}
	    \label{eq:WAztec} 
	 	\end{align}
as in \eqref{eq:defW} and \eqref{eq:Wexplicit}. 
Observe that $W$ has  no pole at the origin.

Theorem \ref{thm:Kernelblock} applies and it gives the form
of the correlation kernel, in $2\times 2$ matrix form, that
will be stated in \eqref{eq:KNAztec} below. It is equivalent to
the form already announced in 
Section \ref{sec:results} in \eqref{eq:AztecKernelwithCDSum}.

The correlation kernel contains the reproducing kernel $R_N(w,z)$ 
with respect to the varying weight $W^N$, and that by Proposition
\ref{prop:CDformula} is expressed in terms of the RH problem for
the MVOP of degree $N$. By Lemma \ref{lem:PNexists} we conclude 
that the degree $N$ monic  MVOP with
respect to the weight $W^N$ exists. What about degrees $< N$?
While it does not matter for the rest that follows, we can
show that the MVOP of lower degrees indeed exist.

\begin{lemma}
The monic MVOP $P_k$ exists for every degree $k=0, \ldots, N-1$.
\end{lemma}
\begin{proof}
Note that $W^N(z) = \frac{A^{2N}(z)}{z^N}$, which we can also write as
\[ W^N(z) = \frac{A^{2N}(z)}{z^{M+k}}, \qquad \text{if } M = N-k. \]
For $k < N$, we consider $2k$ non-intersecting paths on the same graph
with $L=2N$ levels, starting at consecutive positions $0,1, \ldots, 2k-1$,
and ending at shifted positions $2M, 2M+1, \ldots, 2M+2k-1$.
Provided that there are such non-intersecting paths, we have a
determinantal point process as before, and we conclude by an application
of Lemma \ref{lem:PNexists}  that the
monic MVOP of degree $k$ uniquely exists. 

It is readily seen that such paths indeed exist  for the Aztec diamond graph.
For example, by letting the $2k$ paths make  a diagonal up-step $2M = 2N-2k$ times
followed by $2k$ horizontal steps, and there are no down steps.
Then these paths are indeed non-intersecting.
\end{proof}
The construction of non-intersecting paths does not work if $k \geq N+1$
and the MVOP does not exist for those degrees.

In the next section we continue with the analysis of the RH problem
and we show that the correlation kernel \eqref{eq:AztecKernelwithCDSum}
can be rewritten as \eqref{eq:theo12}.

\section{Analysis of the RH problem} \label{sec:RHP}
We consider an Aztec diamond of size $2N$ with two periodic weighting.

\subsection{Correlation kernel} \label{subsec:ADkernel}
In the two periodic Aztec diamond we find the  matrix symbol
\begin{equation} \label{eq:AAztec} 
	A(z) = \frac{1}{z-1} \begin{pmatrix} 2 \alpha z & \alpha(z + 1) \\
	\beta z(z+1) & 2 \beta z \end{pmatrix}, \qquad \alpha \beta = 1, 
	\end{equation}
and the matrix valued weight is $W^N$ with $W$ given by \eqref{eq:WAztec}. 
Note that $W^N$ is a rational function with a pole at $z=1$ only.

The contour $\gamma$ in the RH problem from Section \ref{subsec:RHP}
goes around $0$ and lies in the domain $|z| > 1$. 
By analyticity, since $W$ only has a pole at $z=1$, we are
free to deform the contour to a circle around $1$. 
We use  $\gamma_1$ to denote the circle of radius $r < 1$ around $1$. 
We obtain the following RH problem
for $Y : \mathbb C \setminus \gamma_{1} \to \mathbb C^{4 \times 4}$.
\begin{itemize}
\item $Y$ is analytic,
\item $Y$ has jump 
\begin{equation} \label{eq:Yjump} 
	Y_+(z) = Y_-(z) \begin{pmatrix} I_2 & W^N(z) \\
	0_2 & I_2 \end{pmatrix}, \qquad z \in \gamma_{1}, 
	\end{equation}
\item $Y$ has asymptotic behavior
\begin{equation} \label{eq:Yasymp} 
	Y(z) = \left(I_4 + O(z^{-1}) \right)
	\begin{pmatrix} z^N I_2 & 0_2 \\ 0_2 & z^{-N} I_2 \end{pmatrix} 
		\qquad \text{as } z \to \infty. 
		\end{equation}
\end{itemize}

Because of Theorem \ref{thm:Kernelblock} and  \eqref{eq:CDformula} 
we find the following correlation kernel for arbitrary integer
levels $m$, $m'$ with $0 < m, m' < 2N = L$ and $M=0$,
\begin{multline} \label{eq:KNAztec} 
	\begin{pmatrix} K_N(m, 2x; m', 2y) & K_N(m,2x+1; m', 2y) \\
	K_N(m,2x; m', 2y+1) & K_N(m,2x+1; m', 2y+1) \end{pmatrix} \\
	= - \frac{\chi_{m>m'}}{2\pi i} \oint_{\gamma} A^{m-m'}(z) 
	z^{y-x-1} dz + \\
	 	\frac{1}{(2\pi i)^2} \oint_{\gamma_{0,1}} \oint_{\gamma_{0,1}}
		A^{2N-m'}(w) 
	\begin{pmatrix} 0_2 & I_2 \end{pmatrix} Y^{-1}(w) Y(z) \begin{pmatrix} 
     I_2 \\ 0_2 \end{pmatrix}
	A^{m}(z)  \frac{w^y}{z^{x} w^N} \frac{dz dw}{z(z-w)}.
\end{multline}
The contour $\gamma_{0,1}$ in \eqref{eq:KNAztec} is a circle of radius
$>1+r$ around the origin, as before. The radius is  large
enough such that  $\gamma_{1}$ lies inside  $\gamma_{0,1}$.

The analysis of the correlation kernel \eqref{eq:KNAztec} consists of two parts.
First we apply a RH analysis to the RH problem for $Y$
and then we use this for an asymptotic analysis of the double integral.

The RH analysis is remarkably simple. It is not an asymptotic
analysis, since the outcome is an exact new formula for the correlation kernel.

\begin{theorem} \label{thm:KNexact} 
Assume $2y \geq m'$ and $N$ is even. Then
the correlation kernel \eqref{eq:KNAztec} is equal to
\begin{multline} 
	- \frac{\chi_{m>m'}}{2\pi i} \oint_{\gamma_{0,1}} A^{m-m'}(z) z^{y-x} 
		\frac{dz}{z} + \\
 	\frac{1}{(2\pi i)^2} \oint_{z \in \gamma_{0,1}}  \oint_{w \in \gamma_1}
		A^{N-m'}(w) F(w) A^{-N+m}(z) \frac{z^{N/2}}{w^{N/2}}
		\frac{(z-1)^N}{(w-1)^N}  \frac{w^{y}}{z^{x}} \frac{dz dw}{z(z-w)}.
		  \label{eq:theo41}
\end{multline}
where $F$ is given by \eqref{eq:defF}.
\end{theorem}

Passing from the non-intersecting path model back to the domino
tilings of the Aztec diamond, we should make the change of variables
$m \mapsto m$, $2x \mapsto m+n$, $m' \mapsto m'$ and $2y \mapsto m' + n'$,
that come from the shear transformation described in Section \ref{subsec:particles}.
Inserting these values in \eqref{eq:theo41} we obtain the correlation
kernel \eqref{eq:theo12} and so Theorem \ref{thm:formula} follows
immediately from Theorem \ref{thm:KNexact}.

The rest of Section \ref{sec:RHP} is devoted to the proof of 
Theorem \ref{thm:KNexact}.  We follow the general scheme of 
the  analysis of  RH problems, known as the 
Deift-Zhou steepest descent analysis \cite{DZ}, which was first 
applied to orthogonal polynomials in \cite{DKMVZ1,DKMVZ2}.
Extensions to larger size RH problems are for example 
in \cite{BK2,DuKu}, see also the survey \cite{Kui} and
the references therein. However, the RH analysis in this section
is not an asymptotic analysis, as it produces the
exact formula \eqref{eq:theo41}.

\subsection{Eigenvalues and eigenvectors on the Riemann surface}
We use the eigenvalues $\rho_{1,2}$ of 
$\begin{pmatrix} 2\alpha z & \alpha(z+1) \\ 
\beta z(z+1) & 2 \beta z \end{pmatrix}$ and the
eigenvalues $\lambda_{1,2}$  of $W$ as already introduced
in \eqref{eq:rho12} and \eqref{eq:lambda12}.
The corresponding eigenvectors are in the columns of the matrix
\begin{equation} \label{eq:defE} 
	E(z) = \begin{pmatrix} \alpha(z + 1) & \alpha(z + 1) \\ 
	\rho_1(z)- 2\alpha z & \rho_2(z) - 2 \alpha z \end{pmatrix}.
\end{equation} 
and we have the decompositions \eqref{eq:Adecomp}, 
\eqref{eq:Fdecomp} and \eqref{eq:Wdecomp}.

The eigenvalues and eigenvectors are defined and analytic
in the complex plane cut along the two intervals $(-\infty, -\alpha^2]$
and $[-\beta^2,0]$ where we have 
$\lambda_{1,\pm} = \lambda_{2,\mp}$, $\rho_{1,\pm} = \rho_{2,\mp}$, 
and
\begin{equation} \label{eq:Ejump} 
	E_+ = E_- \sigma_1  \qquad \text{on } (-\infty,-\alpha^2] \cup [-\beta^2,0]. 
	\end{equation}
with $\sigma_1 = \begin{pmatrix} 0 & 1 \\ 1 & 0 \end{pmatrix}$.
	
As already mentioned in Remark \ref{rem:RS}, we use 
the two sheeted Riemann surface $\mathcal R$ associated with the equation
\eqref{eq:RSeq}. 
The Riemann surface has  genus one, unless $\alpha = \beta =1$, in which case the genus is zero.

We use $z$ for a generic coordinate on $\mathcal R$, and if we 
want to emphasize
that $z$ is on the $j$th sheet, we write $z^{(j)}$, for $j=1,2$.
We write $\lambda$ for the function on $\mathcal R$,
\begin{equation} \label{eq:lambda} \lambda(z) =
	\lambda_j(z) \quad \text{ if } z = z^{(j)} \text{ is on the $j$th sheet}, 
	\end{equation}
see \eqref{eq:lambda12}, and similarly for $\rho$.
These are meromorphic functions on $\mathcal R$, namely  
\[ \rho = (\alpha + \beta) z + y, \qquad \lambda = \frac{\rho^2}{z(z-1)^2}, \]
see \eqref{eq:RSeq}, \eqref{eq:rho12} and \eqref{eq:lambda12}.

\begin{lemma} \label{lem:lambdarho}
\begin{enumerate}
\item[\rm (a)] $\rho$ has a simple zero at $z=0$, a double zero at $z = 1^{(2)}$ 
(the point $z=1$ on the second sheet), a triple pole at $z=\infty$, 
and no other zeros or poles,
\item[\rm (b)] $\lambda$ has a double zero at $z=1^{(2)}$, a 
double pole at  $z=1^{(1)}$, and no other zeros or poles,
\item[\rm (c)] The function
\begin{equation} \label{eq:rho-2az}	
	\rho(z) -2 \alpha z  =
		(\beta-\alpha) z + y 
		\end{equation}
has a zero at $z=0$, and a double zero at $z=-1^{(1)}$ (if $\alpha > \beta$).
\item[\rm (d)] $\lambda_1(z) \lambda_2(z) = 1$ for every $z$ 
and $\lambda(\infty) = 1$.
\item[\rm (e)] For real $x$ we have
\begin{align}  \label{eq:lambdareal1}
	|\lambda_{1,\pm}(x)| & = |\lambda_{2,\pm}(x)| = 1,  && x \in (-\infty,-\alpha^2] \cup [-\beta^2,0], \\  \label{eq:lambdareal2} 
  \lambda_1(x) & > 1 > \lambda_2(x) > 0, && x \in (0, \infty), \\
	\lambda_1(x) & < -1 < \lambda_2(x) < 0, && x \in (-\alpha^2,-\beta^2).
	\label{eq:lambdareal3}
	\end{align}
\item[\rm (f)] $|\lambda_1(z)| > |\lambda_2(z)|$ holds for every $z \in 
\mathbb C \setminus ((-\infty,-\alpha^2] \cup [-\beta^2,0])$.
\end{enumerate}
\end{lemma}
\begin{proof}
Parts (a), (b), and (c) are easy to verify from the definitions. We note that
part (d) comes from the fact that
\begin{equation} \label{eq:detW}
	 \det W(z) = 1 
\end{equation}
for every $z$, which follows from \eqref{eq:WAztec} by a direct calculation,
and therefore $\lambda_1(z) \lambda_2(z) = \det W(z) = 1$ for every $z$.
Also from \eqref{eq:WAztec}
\begin{equation} \label{eq:Winfty} 
	\lim_{z \to \infty} W(z) = \begin{pmatrix} 1 & 0 \\ 2\alpha(\alpha + \beta) & 1 \end{pmatrix}  =: W_{\infty}
	\qquad \text{ as } z \to \infty \end{equation}
and so for its eigenvalues we have $\lambda_{1,2}(z) \to 1 $ 
as $z \to \infty$. 

For $x \in (-\infty,-\alpha^2] \cup [-\beta^2,0]$ we have $\lambda_{1,\pm}(x) = \lambda_{2,\mp}(x)$
and $\lambda_{1,\pm}(x) = \frac{1}{\lambda_{2,\pm}(x)}$ because of part (d).
Then  the identity $\lambda_{1,+}(x) \lambda_{1,-}(x) = 
\lambda_{2,+}(x) \lambda_{2,-}(x) = 1$ follows, which gives 
\eqref{eq:lambdareal1}. 

The functions $ \log |\lambda_1|$ and $\log |\lambda_2|$ are harmonic
on $\mathbb C \setminus ((-\infty,-\alpha^2] \cup [-\beta^2,0] \cup \{1\})$,
they are both zero on $(-\infty, -\alpha^2] \cup [-\beta^2,0]$, have the value $1$
at infinity, while 
\[ \lim_{z \to 1} \log |\lambda_1(z)| = +\infty, \qquad 
	 \lim_{z \to 1} \log |\lambda_2(z)| = -\infty \]
because of part (b). 
Then by the minimum principle for harmonic functions 
$\log |\lambda_1(z)| > \log |\lambda_2(z)|$ for 
every $z \in \mathbb C \setminus ((-\infty,-\alpha^2] \cup [-\beta^2,0])$.
This establishes part (f), and also the inequalities \eqref{eq:lambdareal2}
and \eqref{eq:lambdareal3} of  part (e) since $\lambda_{1}(x)$ and $\lambda_2(x)$
are real and positive for  $x \in (0,\infty)$ and real and negative for
$x \in (-\alpha^2, -\beta^2)$, see \eqref{eq:lambda12}.
\end{proof}

\subsection{First transformation $Y \mapsto X$ of the RH problem}
We use the matrix of eigenvalues \eqref{eq:defE} in the
first transformation of the RH problem. We define
\begin{equation} \label{eq:defX}
	 X = Y \begin{pmatrix} E & 0 \\ 0 & E \end{pmatrix} 
\end{equation}
which satisfies the following RH problem.
\begin{itemize}
\item $X$ is analytic on $\mathbb C \setminus (\gamma_{1} \cup (-\infty,-\alpha^2]
	\cup [-\beta^2,0])$,
\item $X$ has jumps
\begin{equation} \label{eq:Xjump}
	 X_+ = \begin{cases} X_- \begin{pmatrix} I_2 &  \Lambda^N \\
	0_2 & I_2 \end{pmatrix} & \text{ on } \gamma_{1}, \\
	 X_- \begin{pmatrix} \sigma_1 & 0_2 \\ 0_2 & \sigma_1 \end{pmatrix}
	& \text{ on } (-\infty, -\alpha^2] \cup [-\beta^2, 0],
	\end{cases} \end{equation}
\item $X$ has asymptotic behavior 
\begin{equation} \label{eq:Xasymp} 
	X(z) = (I_4 + O(1/z)) \begin{pmatrix} z^N E(z) & 0_2 \\ 
	0_2 & z^{-N} E(z) \end{pmatrix} \qquad \text{as } z \to \infty.
	\end{equation}
\end{itemize}	
This is easy to verify from the RH problem for $Y$, the 
definition \eqref{eq:defX}, and the properties
\eqref{eq:Wdecomp} and \eqref{eq:Ejump}.

Since $\det Y(z) = 1$ and 
\begin{equation} \label{eq:detE} 
	\det E(z) = -2 \alpha (z+1) \sqrt{z(z+\alpha^2)(z+\beta^2)},
	\end{equation}
which is easy to check from  \eqref{eq:defE},
we have by \eqref{eq:defX}
\begin{equation} \label{eq:detX} 
	\det X = (\det E)^2 = 4 \alpha^2 z (z+1)^2(z+\alpha^2)(z+\beta^2). 
\end{equation}

\subsection{Second transformation $X \mapsto U$}
Remarkably, we do not need equilibrium measures or $g$-functions  
for the next transformation.

From  Lemma \ref{lem:lambdarho} (b) we know that 
both $z \mapsto (z-1)^2 \lambda_1(z)$ and $z \mapsto (z-1)^{-2} \lambda_2(z)$ have
a removable singularity at $z=1$, and hence they are
analytic in $\mathbb C \setminus ((-\infty,-\alpha^2] \cup [-\beta^2, 0])$ without any zeros. 
We recall that $N$ is even and we put 
\begin{multline} \label{eq:defU} 
	U = L X \\
	\times \begin{cases}
	\diag\left( \frac{\lambda_1^{N/2}}{(z-1)^N}, \frac{\lambda_2^{N/2}}{(z-1)^N}, 
	\frac{(z-1)^N}{\lambda_1^{N/2}}, \frac{(z-1)^N}{\lambda_2^{N/2}} \right),
	& \text{ for } |z-1| > r, \\
	\diag\left( (z-1)^N\lambda_1^{N/2}, 
		\frac{\lambda_2^{N/2}}{(z-1)^N}, 
		\frac{1}{(z-1)^N \lambda_1^{N/2}}, 
		\frac{(z-1)^N}{\lambda_2^{N/2}} \right),
	& \text{ for } |z-1| < r, \end{cases}
	 \end{multline}
where $L$ is the constant matrix
\begin{equation}  \label{eq:defL}
	L = \begin{pmatrix} W_{\infty}^{-N/2} & 0_2  \\
		0_2 & W_{\infty}^{N/2} \end{pmatrix}.
		\end{equation}
	with $W_{\infty}$ as in \eqref{eq:Winfty}.
Then $U$ is defined on 
$\mathbb C \setminus (\gamma_{1} \cup (-\infty,-\alpha^2]\cup [-\beta^2,0])$, and from the definition \eqref{eq:defU} and
 the RH problem for $X$ we obtain
\begin{itemize}
\item $U$ is analytic,
\item $U$ has the jumps
\begin{equation} \label{eq:Ujump} 
	 U_+ = \begin{cases} U_-  \begin{pmatrix} (z-1)^{2N} & 0 & 1 & 0 \\
	0 & 1 & 0 & (z-1)^{2N} \\ 0 & 0 & (z-1)^{-2N} & 0 \\
	0 & 0 & 0 & 1 \end{pmatrix}  \text{ on } \gamma_{1}, \\
	 U_- \begin{pmatrix} \sigma_1 & 0_2 \\ 0_2 & \sigma_1 
	 \end{pmatrix} \qquad \qquad \qquad	
	 \text{ on } (-\infty,-\alpha^2] \cup [-\beta^2,0].
	 \end{cases}	
	\end{equation}
\item $U$ has asymptotic behavior
\begin{align} \nonumber 
	U(z) & = L (I_4 + O(1/z)) \begin{pmatrix}  E(z) \Lambda^{N/2}(z) & 0_2 \\ 
	0_2 &  E(z) \Lambda^{-N/2}(z) \end{pmatrix} \\
    & = (I_4 + O(1/z)) \begin{pmatrix}  E(z) & 0_2 \\ 
	0_2 &  E(z) \end{pmatrix} 
	 \qquad \text{as } z \to \infty. \label{eq:Uasymp} 
	\end{align}
\end{itemize}
To obtain the jump \eqref{eq:Ujump} on $(-\infty,-\alpha^2] \cup [-\beta^2,0]$ 
we also have to use the
fact that $\lambda_{1,\pm} = \lambda_{2,\mp}$ on these cuts.

The asymptotic condition \eqref{eq:Uasymp} requires some explanation.
The first equality in \eqref{eq:Uasymp} is clear from the definition 
\eqref{eq:defU} of $U$ for $|z-1|>r$, and the asymptotic 
behavior \eqref{eq:Xjump} of $X$. By \eqref{eq:Wdecomp} we have 
$E(z) \Lambda^{\pm N/2}(z) = W^{\pm N/2}(z) E(z)$ so that by
\eqref{eq:WAztec}, \eqref{eq:Winfty}, and \eqref{eq:defL} we get that
\begin{align*}  
	L \begin{pmatrix} E(z) \Lambda^{N/2}(z) & 0_2 \\ 0_2 & E(z)
	\Lambda^{-N/2}(z) \end{pmatrix}
   &	= L \begin{pmatrix} W^{N/2}(z) E(z) & 0_2 \\ 0_2 & W^{-N/2}(z) E(z) \end{pmatrix} \\
	& = \left( I_4 + O(1/z) \right) \begin{pmatrix} E(z) & 0_2 \\ 0_2 & E(z)
	\end{pmatrix}, \end{align*}
and this leads to the second equality in \eqref{eq:Uasymp}.

\subsection{Third transformation $U \mapsto T$}
In the third transformation we turn the entries $(z-1)^{\pm 2N}$
in the jump matrix on $\gamma_{1}$ into an off-diagonal entry. It corresponds
to the opening of lenses in a steepest descent analysis. We also remove
the $24$-entry in the jump matrix on $\gamma_{1}$.

We define
\begin{equation} \label{eq:defT} 
	T(z) = \begin{cases} U(z), & \text{for } |z-1|>r, \\
	U(z) \begin{pmatrix} 0 & 0 & -1 & 0 \\
	0 & 1 & 0 & -(z-1)^{2N} \\
	1 & 0 & (z-1)^{2N} & 0 \\
	0 & 0 & 0 & 1 \end{pmatrix}, & \text{for } |z-1| < r.  
	\end{cases} 
\end{equation}
Straightforward calculations, where we just use \eqref{eq:defT}
and the RH problem for $U$,  show that $T$ satisfies
\begin{itemize}
\item $T : \mathbb C \setminus (\gamma_{1} \cup (-\infty,-\alpha^2] \cup [-\beta^2, 0])
\to \mathbb C^{4 \times 4}$ is analytic,
\item $T$ has the jumps
\begin{equation} \label{eq:Tjump} 
	 T_+ = \begin{cases} T_-  
	 \begin{pmatrix} 1 & 0 & 0 & 0 \\ 
		0 & 1 & 0 & 0 \\
		(z-1)^{-2N} & 0 & 1 & 0 \\
		0 & 0 & 0 & 1
		\end{pmatrix} \qquad \text{on } \gamma_{1}, \\
	 T_- \begin{pmatrix} \sigma_1 & 0_2 \\ 0_2 & \sigma_1 
	 \end{pmatrix} \qquad \
	 \text{ on } (-\infty,-\alpha^2] \cup [-\beta^2,0].
	 \end{cases}	
	\end{equation}
\item $T$ has asymptotic behavior
\begin{align}  
	T(z) & = (I_4 + O(1/z)) \begin{pmatrix}  E(z) & 0_2 \\ 
	0_2 &  E(z) \end{pmatrix} 
	 \qquad \text{as } z \to \infty. \label{eq:Tasymp} 
	\end{align}
\end{itemize}

\subsection{Fourth transformation $T \mapsto S$}
We next remove the jumps on the negative real axis. We use
$\begin{pmatrix} E & 0 \\ 0 & E \end{pmatrix}$ as global parametrix,
since it has the same jump on $(-\infty, - \alpha^2] \cup [-\beta^2,0]$ 
as $T$ has, see \eqref{eq:Tjump}.
We define
\begin{equation} \label{eq:defS}
   S = T \begin{pmatrix} E^{-1} & 0_2 \\ 0_2 & E^{-1} \end{pmatrix}.
\end{equation}
and then $S$ has no jump on $(-\infty,-\alpha^2) \cup (-\beta^2, 0)$,
that is, $S_+ = S_-$ on these two intervals.

Since $E$ is not invertible at $z=0$, $z=-\alpha^2$, $z=-\beta^2 $,
see also \eqref{eq:detE}, we could have introduced
singularities at these points. Therefore we look at the 
combined transformations $Y \mapsto X \mapsto U \mapsto T \mapsto S$
in order to express $S$ directly in terms of $Y$. For $z$ outside of $\gamma_{1}$ 
we have by \eqref{eq:defX}, \eqref{eq:defU}, \eqref{eq:defT} and \eqref{eq:defS},
\begin{align*} 
	S  & =  U \begin{pmatrix} E^{-1} & 0_2 \\ 0_2 & E^{-1} \end{pmatrix} \\
		& = L  X \begin{pmatrix} (z-1)^{-N} \Lambda^{N/2} E^{-1} & 0_2 \\
			0_2 & (z-1)^N \Lambda^{-N/2} E^{-1} \end{pmatrix} \\
		& = L Y \begin{pmatrix} (z-1)^{-N} E \Lambda^{N/2} E^{-1} & 0_2 \\
			0_2 & (z-1)^N E \Lambda^{-N/2} E^{-1} \end{pmatrix}.
	\end{align*}
Since $ E \Lambda E^{-1} = W$ by \eqref{eq:Wdecomp}, we simply have (recall $N$ is even)
\begin{equation} \label{eq:defSinY} 
	S = L Y \begin{pmatrix} (z-1)^{-N} W^{N/2} & 0_2 \\ 
	0_2 & (z-1)^N W^{-N/2} \end{pmatrix},
		\quad |z-1| > r.
	 \end{equation}
This shows indeed that \eqref{eq:defS} does not introduce any singularities, since 
$\det W(z) = 1$ for every  $z$, and $W(z)$ and $W^{-1}(z)$ have poles  at $z=1$ only.

Thus $S$ has analytic continuation across $(-\infty, -\alpha^2]$ 
and $[-\beta^2,0)$ and satisfies the following RH problem that we
obtain immediately from \eqref{eq:defS} and the RH problem for $T$.
\begin{itemize}
\item $S : \mathbb C \setminus \gamma_{1} \to \mathbb C^{4 \times 4}$ is analytic,
\item $S$ has the jump
\begin{equation} \label{eq:Sjump} 
	S_+(z) = S_-(z) \begin{pmatrix} I_2 & 0_2 \\ (z-1)^{-2N} F(z)  & I_2 \end{pmatrix} \quad
	\text{ for } z \in \gamma_{1},
	\end{equation}
where $F(z) = E(z) \begin{pmatrix} 1 & 0 \\ 0 & 0 \end{pmatrix} E^{-1}(z)$ 
is as in \eqref{eq:defF} and \eqref{eq:Fdecomp}.
\item $S$ has asymptotic behavior
\begin{align}  
	S(z) = I_4 + O(1/z)  \qquad \text{as } z \to \infty. \label{eq:Sasymp} 
	\end{align}
\end{itemize}
The RH problem is now normalized at infinity. Note also that 
the transformation \eqref{eq:defS} restores the property $\det S = 1$,
since $\det T = \det X = (\det E)^2$, 
see \eqref{eq:detX}.

The jump matrix in  \eqref{eq:Sjump} is lower triangular, and
the RH problem for $S$ is normalized at infinity by \eqref{eq:Sasymp}.
This means that we can solve the RH problem explicitly by a contour 
integral. We find
\begin{equation} \label{eq:Ssolution}
S(z) = \begin{pmatrix} I_2 & 0_2 \\[5pt]
	\ds \frac{1}{2\pi i} \oint_{\gamma_{1}} \frac{F(s)}{(s-1)^{2N} (s-z)} ds & I_2
	\end{pmatrix}, \qquad z \in \mathbb C  \setminus \gamma_{1}.
\end{equation}

\subsection{Proof of Theorem \ref{thm:KNexact} }

We can now give the proof of Theorem \ref{thm:KNexact}.
\begin{proof}
We analyze the effect of the transformations on the correlation kernel
\eqref{eq:KNAztec}. From \eqref{eq:defSinY} and \eqref{eq:defW} we have 
for $z,w$ outside of $\gamma_{1}$,
\begin{multline*} 
	\begin{pmatrix} 0_2 & I_2 \end{pmatrix}
	Y^{-1}(w) Y(z) \begin{pmatrix} I_2 \\ 0_2 \end{pmatrix} \\
	 = 
	(w-1)^N (z-1)^N    W^{-N/2}(w) 
	\begin{pmatrix} 0_2 & I_2 \end{pmatrix}
	S^{-1}(w) S(z) \begin{pmatrix} I_2 \\ 0_2 \end{pmatrix}
	W^{-N/2}(z) \\ = 
	(w-1)^N (z-1)^N w^{N/2} z^{N/2}   
	A^{-N}(w) 
	\begin{pmatrix} 0_2 & I_2 \end{pmatrix}
	S^{-1}(w) S(z) \begin{pmatrix} I_2 \\ 0_2 \end{pmatrix}
	A^{-N}(z).	
	\end{multline*}
Thus
\begin{multline}	\label{eq:Proof41a}
A^{2N-m'}(w) 
	\begin{pmatrix} 0_2 & I_2 \end{pmatrix} Y^{-1}(w) Y(z) \begin{pmatrix} 
     I_2 \\ 0_2 \end{pmatrix}
	A^{m}(z) w^{-N} \\
	= 
	(w-1)^N (z-1)^N w^{-N/2} z^{N/2} A^{N-m'}(w) 
	\begin{pmatrix} 0_2 & I_2 \end{pmatrix}
	S^{-1}(w) S(z) \begin{pmatrix} I_2 \\ 0_2 \end{pmatrix}
	A^{-N+m}(z)
	\end{multline}
which is part of the expression that appears in the double integral in  \eqref{eq:KNAztec}.
Because of \eqref{eq:Ssolution} and
\[ S^{-1}(w) = \begin{pmatrix} I_2 & 0_2 \\
	\ds -\frac{1}{2\pi i} \oint_{\gamma_{1}} 
		\frac{F(s)}{(s-1)^{2N}(s-w)} ds 	& I_2 \end{pmatrix} \]
we have
\begin{multline}  \label{eq:Proof41b} 
	\begin{pmatrix} 0_2 & I_2 \end{pmatrix}
	S^{-1}(w) S(z) \begin{pmatrix} I_2 \\ 0_2 \end{pmatrix}
	\\ = 
	\frac{1}{2\pi i} \oint_{\gamma_{1}} \frac{F(s)}{(s-1)^{2N}(s-z)} ds
	-\frac{1}{2\pi i} \oint_{\gamma_{1}} \frac{F(s)}{(s-1)^{2N}(s-w)} ds
	\\
	= 
	-\frac{z-w}{2\pi i} \oint_{\gamma_{1}} \frac{F(s)}{(s-1)^{2N}(s-w)(s-z)} ds \end{multline}
	
Using \eqref{eq:Proof41a} and \eqref{eq:Proof41b} we
see that the double integral in \eqref{eq:KNAztec} is equal to
\begin{multline} \label{eq:Proof41c} 
	-\frac{1}{(2\pi i)^2}
		\oint_{\gamma_{0,1}} \oint_{\gamma_{0,1}}
		A^{N-m'}(w) 
		\left(\frac{1}{2 \pi i}  
		\oint_{\gamma_{1}} \frac{F(s)}{(s-1)^{2N}(s-w)(s-z)} ds \right)
		\\
		\times
	A^{-N+m}(z) \frac{z^{N/2}}{w^{N/2}} (w-1)^N (z-1)^N 
		\frac{w^y}{z^{x+1}} dz dw.		
		 \end{multline}

We change the order of integration in \eqref{eq:Proof41c} and
evaluate the $w$-integral first. By  a residue calculation
\begin{multline} \label{eq:Proof41d} 
	\frac{1}{2\pi i} \oint_{\gamma_{0,1}}
	A^{N-m'}(w) (w-1)^N w^{y-N/2} \frac{1}{w-s} dw \\ 
	= 	A^{N-m'}(s) (s-1)^N s^{y-N/2}, \qquad s \in \gamma_{1}.
	\end{multline}
Indeed the singularities at $w=0$ and $w=1$ in the integrand 
in the left-hand side of \eqref{eq:Proof41d} are removable 
(we use \eqref{eq:AAztec}, $N$ is even, and $2y \geq m'$). 
The only singularity is at $w=s$
and \eqref{eq:Proof41d} indeed follows by Cauchy's formula
since $s \in \gamma_{1}$ lies inside $\gamma_{0,1}$.

Using \eqref{eq:Proof41d} in \eqref{eq:Proof41c}
and changing the integration variable $s$ to $w$, we obtain
the double integral in \eqref{eq:theo41}.
The single integral in \eqref{eq:theo41} is of course immediate
from \eqref{eq:KNAztec}. This completes the proof
of  Theorem \ref{thm:KNexact}.
\end{proof}

\subsection{A consistency check}

The RH analysis gives us explicit formulas, and in particular also for
the $2 \times 2$ left upper block
\[ P_N(z) = \begin{pmatrix} Y_{11}(z) & Y_{12}(z) \\ Y_{21}(z) & Y_{22}(z) \end{pmatrix}
	\]
which by \eqref{eq:solY} should be the monic MVOP of degree $N$. Following the transformations $Y \mapsto X \mapsto U \mapsto T \mapsto S$ and
the expression \eqref{eq:Ssolution} for $S$, we see that
\begin{equation} \label{eq:PN} 
	P_N(z) = (z-1)^N  W_{\infty}^{N/2} W^{-N/2}(z) 
	\end{equation}
which is indeed a monic matrix valued polynomial of degree $N$ since $N$ is even.
Note that $W(z)$ has a double pole at $z=1$, hence $W^{-N/2}(z)$ has a pole
of order $N$ at $z=1$,  and the pole is compensated by the $N$ th order
zero of $(z-1)^N$.

We then check that 
\[ P_N(z) W^N(z) = (z-1)^N  W_{\infty}^{N/2} 
	W^{N/2}(z) \] 
also has a removable singularity at $z=1$. Hence it is also a matrix polynomial
and the matrix orthogonality follows in a trivial way from Cauchy's theorem
\[ \frac{1}{2\pi i} \oint_{\gamma_1} P_N(z) W^N(z) Q(z) dz = 0_2 \]
for every matrix valued polynomial $Q$ (and not just for polynomials of degree 
$\leq N-1$).

The degree $N-1$ polynomial $Q_{N-1}$ 
is in the left lower $2 \times 2$ block of $Y$, see \eqref{eq:solY}.
This is a polynomial of degree $N-1$, but not necessarily a monic one.
From the transformation in the RH analysis and \eqref{eq:Ssolution},
we find
\begin{multline} \label{eq:QNoutside} 
	Q_{N-1}(z) = \\
	 (z-1)^N  W_{\infty}^{-N/2} 
	\left(	\frac{1}{2 \pi i} \oint_{\gamma_{1}} \frac{F(s)}{(s-1)^{2N}(s-z)} ds \right)	
	W^{-N/2}(z) \end{multline}
for $|z-1| > r$, which is indeed $O(z^{N-1})$ as $z \to \infty$. 
The analytic continuation to $|z-1| < r$, is given by the Sokhotskii-Plemelj formula,
\begin{multline*} 
	Q_{N-1}(z) = \\
	 (z-1)^N  W_{\infty}^{-N/2}
	\left(	\frac{1}{2 \pi i} \oint_{\gamma_{1}} \frac{F(s)}{(s-1)^{2N}(s-z)} ds \right)	
	W^{-N/2}(z) \\
	+ (z-1)^{-N} W_{\infty}^{-N/2} F(z) W^{-N/2}(z).
	 \end{multline*}
Note that from \eqref{eq:Fdecomp} and \eqref{eq:Wdecomp} we have 
$F(z) W^{-N/2}(z) = \lambda_1^{-N/2}(z) F(z)$ and this has a zero at $z=1$
of order $N$ because of Lemma \ref{lem:lambdarho} (b).
The zero cancels the $N$th order pole in $(z-1)^{-N}$,
and we see that the extra term is analytic for $|z-1|< r$, which
confirms that the expression defining $Q_{N-1}$ is a polynomial of degree $\leq N-1$.

Let's verify the orthogonality \eqref{eq:QNMVOP} where we take a contour $\gamma$
that lies in the exterior of $\gamma_1$. 
For an integer $k \geq 0$, we have by \eqref{eq:QNoutside}
\begin{multline*} 
	\frac{1}{2\pi i} \oint_{\gamma_1} Q_{N-1}(z) W^N(z) z^k dz = 
	W_{\infty}^{-N/2} \\
	\times 
 	\frac{1}{(2\pi i)^2}
 		\oint_{\gamma_1} (z-1)^N 
 		\left( \oint_{\gamma_{1}} \frac{F(s)}{(s-1)^{2N}(s-z)} ds \right)
 		W^{N/2}(z) z^k dz.
 		\end{multline*}
We change the order of integration and use Cauchy's formula
(the only pole is at $z=s$) to obtain
\begin{multline} \label{eq:QNMVOP1}
	\frac{1}{2\pi i} \oint_{\gamma} Q_{N-1}(z) W^N(z) z^k dz 
	=  W_{\infty}^{-N/2} \frac{1}{2\pi i}
		\oint_{\gamma_{1}} \frac{-F(s) W^{N/2}(s)}{(s-1)^{N}}   
		s^k  ds.
\end{multline}

Note that $(I_2 - F(s)) W^{N/2}(s) = (I_2- F(s)) \lambda_2^{N/2}(s)$
(again by \eqref{eq:Fdecomp} and \eqref{eq:Wdecomp}) and 
$\lambda_2^{N/2}(s)$ has a zero of order $N$ at $s=1$ by 
Lemma \ref{lem:lambdarho} (b). Thus
\[ \frac{1}{2\pi i}
		\oint_{\gamma_{1}} \frac{(I_2-F(s)) W^{N/2}(s)}{(s-1)^{N}}   
		s^k  ds = 0, \]
and combining this with \eqref{eq:QNMVOP1} leads to
\begin{equation} \label{eq:QNMVOP2}
	\frac{1}{2\pi i} \oint_{\gamma} Q_{N-1}(z) W^N(z) z^k dz 
	=  W_{\infty}^{-N/2} \frac{1}{2\pi i}
		\oint_{\gamma_{1}} \frac{-W^{N/2}(s)}{(s-1)^{N}} s^k  ds.
\end{equation}
Recall that $W^{N/2}(s)$ is a rational matrix valued function
whose only pole is at $s=1$ and it is bounded at infinity. 
Then in \eqref{eq:QNMVOP2} we move the contour $\gamma_{1}$ to infinity. There is no
contribution from infinity if $k \leq N-2$, while for $k=N-1$
there is a residue contribution at infinity and the
expression \eqref{eq:QNMVOP2} becomes $-I_2$ for $k=N-1$. 
We conclude that  \eqref{eq:QNMVOP} indeed holds.

\section{Asymptotic analysis} \label{sec:asymp}

In the final section of the paper we are analyzing the formula 
\eqref{eq:theo12} in a scaling limit where
$N \to \infty$ and the coordinates $(m,n)$ and $(m',n')$ scale linearly
with $N$. We are going to distinguish the three phases of the model,
and prove Theorems \ref{thm:gaslimit}, \ref{thm:cusplimit},
and \ref{thm:gaussianlimit}.

\subsection{Preliminaries}
We first rewrite the formula \eqref{eq:theo12} in a form that 
already contains the gas phase kernel \eqref{eq:gaskernel}
and double integrals with the phase functions $\Phi_1$ and $\Phi_2$
from \eqref{eq:defPhi}, see Corollary \ref{cor:KNrewrite}.

We may and do assume that the contour $\gamma_{0,1}$ is a contour
in $\mathbb C \setminus ((-\infty,-\alpha^2] \cup [-\beta^2,0])$
going around the interval $[-\beta^2,0]$ once in positive direction.

\begin{lemma} \label{lem:wdecay}
The integrand in the double integral in \eqref{eq:theo12}
is $O(w^{-N+n'/2 - 1/2})$ as $w \to \infty$.
\end{lemma}
\begin{proof}
From the formulas \eqref{eq:defA}-\eqref{eq:defF} we easily get
that $A(w) = O(w)$,  $F(w) = O(w^{1/2})$, $A^2(w) = O(w)$ 
and $A(w) F(w)  = O(w)$ as $w \to \infty$.
This implies that $A^{N-m'}(w) F(w) = O(w^{(N-m'+1)/2})$ as $w \to \infty$.
Use this in the integrand in \eqref{eq:theo12} and the lemma follows.
\end{proof}

Note that $n'$ could go up to $2N-1$, and then  $O(w^{-N+n'/2 - 1/2}) = O(w^{-1})$.
 However, then we are close to the boundary of the Aztec diamond, and we do not
consider this in what follows, since we focus on the gas phase. 
So we assume $n' \leq 2N-2$, and then the integrand in double integral
in \eqref{eq:theo12} is $O(w^{-3/2})$
as $w \to \infty$.

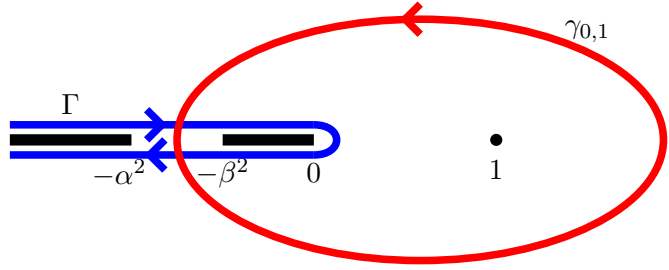
\begin{figure}[t]
\begin{center}
\begin{tikzpicture}[scale=0.8]

\draw [line width=1.5mm] (-6,0)--(-4,0);
\draw (-4.2,-0.55) node {$-\alpha^2$};

\draw [line width= 1.5mm] (-2.5,0)--(-1,0);
\draw (-2.5,-0.55) node {$-\beta^2$};
\draw (-1,-0.55) node {$0$};

\fill (2,0) circle [radius=1mm];
\draw (2,-0.5) node {$1$};

\draw [blue,line width = 1mm] (-6,0.25)--(-1,0.25);
\draw [blue,line width = 1mm] (-6,-0.25)--(-1,-0.25);
\draw [blue,line width = 1mm] (-1,-0.25) .. controls (-0.5,-0.25) and (-0.5,0.25) .. (-1,0.25);
\draw [blue,line width = 1mm] (-3.75,0.0)--(-3.5,0.25)--(-3.75,0.5);
\draw [blue,line width = 1mm] (-3.45,-0.0)--(-3.7,-0.25)--(-3.45,-0.5);
\draw (-5,0.6) node {$\Gamma$};

\draw (3.5,1.8) node {$\gamma_{0,1}$};
\draw [red,line width=1mm] (0.75,0) ellipse (4cm and 2cm);
\draw [red,line width=1mm] (0.75,1.75)--(0.5,2)--(0.75,2.25);
\end{tikzpicture} 
\end{center}

\caption{Contours $\gamma_{0,1}$ and $\Gamma$
\label{fig:Contours2}}
\end{figure}

Then for a fixed $z \in \gamma_{0,1}$,  we deform the contour $\gamma_1$ to a 
contour $\Gamma$ going around the negative real axis, starting at $-\infty$ in the 
upper half-plane and ending at $-\infty$ in the lower half-plane, as in
Figure \ref{fig:Contours2}. Since the integrand 
is $O(w^{-3/2})$, by Lemma \ref{lem:wdecay},
there is no contribution from infinity, but there is a residue
contribution from the pole at $w=z$. These residues
 combine to give the $z$-integral (we use that $F(z)$ and $A(z)$ commute)
\[ \frac{1}{2\pi i} \oint_{\gamma_{0,1}} F(z) A^{m-m'}(z) 
	\frac{z^{(m'+ n')/2}}{z^{(m+n)/2}} \frac{dz}{z} \]
Together with the single integral in \eqref{eq:theo12}
this gives the limit \eqref{eq:gaskernel} that we expect to get 
in the gas phase. We proved the following.

\begin{proposition} \label{prop:prop52}
Suppose $N$ is even and $(m,n) \in \mathcal B_N$, 
$(m',n') \in \mathcal B_N$ with $m+n$ and $m'+ n'$  even
and $n' \leq 2N-2$. 
Then 
\begin{multline} \label{eq:prop52}
	\mathbb K_N(m,n; m',n') 
	= \mathbb K_{gas}(m,n ; m', n')
	\\
	+ \frac{1}{(2\pi i)^2} \oint_{\gamma_{0,1}} \frac{dz}{z} 
		\int_{\Gamma} 
		\frac{dw}{z-w} A^{N-m'}(w) F(w) A^{-N+m}(z) \\
		\times
		\frac{z^{N/2} (z-1)^N}{w^{N/2} (w-1)^N} \frac{w^{(m'+n')/2}}{z^{(m+n)/2}}.
\end{multline}
\end{proposition}

Thus to establish Theorem \ref{thm:gaslimit} we have to prove that 
in the gas phase the double integral in \eqref{eq:prop52} tends to $0$
as $N \to \infty$ at an exponential rate.

We can rewrite \eqref{eq:prop52} where we assume that
$m = (1+\xi_1) N$, $n = (1+\xi_2)N$, $m'= (1+\xi_1')N$
and $n'= (1+\xi_2')N$. We use $\Phi_1$ and $\Phi_2$ as
in \eqref{eq:defPhi}, and to emphasize that these
functions depend on $\xi_1$ and $\xi_2$, we write
$\Phi_1(z; \xi_1, \xi_2)$ and
$\Phi_2(z; \xi_1, \xi_2)$.

\begin{corollary} \label{cor:KNrewrite}
Suppose $m = (1+ \xi_1)N$,  $n = (1+\xi_2)N$, $m'= (1+\xi_1')N$,
and $n'= (1+\xi_2')N$ with $-1 < \xi_1, \xi_2, \xi_1', \xi_2' < 1$.
Assume $N$, $m+n$ and $m'+ n'$ are even. Then 
\begin{multline} \label{eq:KNremainder}
	\mathbb K_N(m,n; m',n') 
	= \mathbb K_{gas}(m,n; m', n')
	\\
	+ \frac{1}{(2\pi i)^2} \oint_{\gamma_{0,1}} \frac{dz}{z}
	\int_{\Gamma} \frac{dw}{z-w} F(w) F(z)
		e^{N( \Phi_1(z;\xi_1,\xi_2)-\Phi_1(w; \xi_1',\xi_2'))/2} \\
	+ 
\frac{1}{(2\pi i)^2} \oint_{\gamma_{0,1}} \frac{dz}{z}
	\int_{\Gamma} \frac{dw}{z-w} F(w) (I_2- F(z))
		e^{N( \Phi_2(z; \xi_1,\xi_2)-\Phi_1(w; \xi_1',\xi_2'))/2},
		\end{multline}
with contours as in Figure \ref{fig:Contours2}.
\end{corollary}
\begin{proof}
Because of \eqref{eq:Adecomp}--\eqref{eq:Fdecomp} we have
\[ A^{N-m'}(w) F(w) = F(w) \left( \frac{\rho_{1}(w)}{w-1} \right)^{N-m'}. \]
Thus in view of \eqref{eq:lambda12} and \eqref{eq:defPhi} we get 
\begin{align} \nonumber
	A^{N-m'}(w) F(w) \frac{w^{(m'+n')/2}}{w^{N/2} (w-1)^N}
		& =  F(w) \lambda_1^{(N-m')/2}(w) \frac{w^{n'/2}}{(w-1)^N} \\
		& = F(w) e^{-N \Phi_1(w; \xi_1', \xi_2')/2},
		\label{eq:wintegrand}
		\end{align} 
and similarly, 
\begin{multline} \label{eq:zintegrand}
	A^{-N+m}(z)  \frac{z^{N/2} (z-1)^N}{z^{(m+n)/2}} \\
	= F(z) e^{N \Phi_1(z; \xi_1, \xi_2)/2}
	+ (I_2-F(z)) e^{N \Phi_2(z; \xi_1, \xi_2)/2}.
		\end{multline} 
Using  \eqref{eq:wintegrand} and \eqref{eq:zintegrand} 
in \eqref{eq:prop52} we arrive at \eqref{eq:KNremainder}.
\end{proof}

\subsection{Saddle points} \label{subsec:saddles}

The large $N$ behavior of the $z$-integrals in \eqref{eq:KNremainder}
is dominated by the factors $e^{N \Phi_1(z)}$ and 
$e^{N \Phi_2(z)/2}$  that are exponential in $N$. 
Similarly the $w$ part of the integrand is dominated by  
$e^{-N \Phi_1(w)/2}$. 

We study the saddle points, which in Definition \ref{def:saddles} 
were already introduced as 
the zeros of the meromorphic differential $\Phi'(z) dz$ from \eqref{eq:MeroDiff}
defined on the Riemann surface $\mathcal R$ associated wth \eqref{eq:RSeq}. 
Of course, $\Phi$ depends on $\xi_1, \xi_2$, and thus the saddle points depend on 
these parameters. Throughout we restrict to $-1 < \xi_1, \xi_2 < 1$.
The differential has simple poles at $1^{(1)}$, $1^{(2)}$, $0$ 
and $\infty$ with residues given in the following table. 

\begin{center}
\begin{tabular}{c|ccc|c}
	     & residue of & residue & residue  & residue  \\
	pole &  of $\frac{dz}{z-1}$ & of $\frac{dz}{z}$ & of $\frac{\lambda'}{\lambda} dz$ &
	of $\Phi' dz$ \\[5pt]  \hline 
	$1^{(1)}$ & $1$ & $0$ & $-2$ & $-2\xi_1 + 2$ \\
	$1^{(2)}$ & $1$ & $0$ & $2$  & $2\xi_1 + 2$ \\
	$0$ 	  & $0$ & $2$ & $0$  & $- 2 \xi_2-2$ \\
	$\infty$  & $-2$ & $-2$ & $0$ & $2\xi_2-2$ \\[5pt] \hline 
	\end{tabular}
\end{center}
The residues of $\frac{\lambda'}{\lambda} dz$ at $z = 1^{(1)}$ and $z = 1^{(2)}$ come from
the double pole and double zero that $\lambda$ has at these points,
see Lemma \ref{lem:lambdarho} (b). The residues add up to zero, as it should be. 

We assume $\alpha > 1$ so that the genus of $\mathcal R$ is one. Then
there are also  four zeros of $\Phi' dz$ counting multiplicities.

Recall that the real part of the Riemann surface consists of the cycles 
$\mathcal C_1$ and $\mathcal C_2$ as in \eqref{eq:RSreal}. 
\begin{proposition} \label{prop:saddleC1}
For every $\xi_1, \xi_2 \in (-1,1)$ there are at least two distinct saddle points 
on the cycle $\mathcal C_1$.
\end{proposition}
\begin{proof}
If $\mathcal C$ is a path from $P$ to $Q$ on the Riemann surface 
avoiding the poles, then by \eqref{eq:defPhi},
\[ \int_{\mathcal C} \Phi'(z) dz = 
	\left[2 \log(z-1) - (1+\xi_2) \log z + \xi_1 \log \lambda(z) \right]_P^Q \]
for a choice of  continuous branches of the logarithms along the path. 
Since $\xi_1$ and $\xi_2$ are real, it follows that the real
part is well-defined, it depends on $P$ and $Q$, but is
otherwise  independent of the path. Thus 
\begin{equation} \label{eq:Boutroux}
	\Re \left(\oint_{\mathcal C} \Phi'(z) dz \right) = 0
	\end{equation} 
	for a closed path $\mathcal C$.

Observe that that there are no poles on the cycle $\mathcal C_1$, and $\Phi'$ is real there. If there were
no two distinct zeros on $\mathcal C_1$, then there would be no sign change,
and the integral would be non-zero and real, which would contradict the
condition \eqref{eq:Boutroux}.
\end{proof}

The saddle points are explicit in case $\xi_1 = 0$, since then by \eqref{eq:MeroDiff}
\begin{equation} \label{eq:MeroDiff0} 
	\Phi'(z) dz = \left( \frac{2}{z-1} -  \frac{1+\xi_2}{z} \right) dz. 
\end{equation}
The equation $\frac{2}{z-1} = \frac{1+\xi_2}{z}$ has the unique solution
\begin{equation} \label{eq:defzc} 
	z_c(\xi_2) = - \frac{1+ \xi_2}{1-\xi_2}. 
	\end{equation}
This gives us two saddle points, namely the two points on $\mathcal R$
with \eqref{eq:defzc} as  $z$-coordinate.
The other two saddles come from the branch points $-\alpha^2$, $-\beta^2$,
which are zeros of the differential $dz$. The branch point $z=0$
is also a zero of $dz$, but this zero gets cancelled by the (double)
pole of $\frac{1+\xi_2}{z}$ in \eqref{eq:MeroDiff0}. 

For special values of $\xi_2$ the saddles at $z=z_c(\xi_2)$ coincide
with the saddle at $-\alpha^2$ or $-\beta^2$. This happens
for the values $\pm \xi_{2}^*$ with 
\begin{align} \label{eq:defxi2c}
	\xi_{2}^*  = \frac{\alpha-\beta}{\alpha + \beta} \in (0,1). 
	\end{align}
Then depending on the value of $\xi_2$, we are in the liquid or gas phase,
or on the liquid-gas transition, as defined in Definition \ref{def:phases}.

\begin{lemma} \label{lem:phases0}
Suppose $\xi_1=0$ and $-1 < \xi_2 < 1$.
\begin{enumerate}
\item[\rm (a)] $(0,\xi_2)$ cannot be in the solid phase. 
\item[\rm (b)] If $\xi_2 \in (-1, -\xi_2^*) \cup (\xi_2^*,1)$
then $(0, \xi_2) \in \mathfrak L$. 	
\item[\rm (c)] If $-\xi_2^* < \xi_2 < \xi_2^*$ then 
$(0, \xi_2) \in \mathfrak G$. 	
\item[\rm (d)] If $\xi_2 = \pm \xi_2^*$ then $(0,\xi_2)$
is on the liquid-gas transition. 
\end{enumerate}
\end{lemma}
\begin{proof}
(a) It is clear from \eqref{eq:defzc} that $z_c(\xi_2) < 0$ and so
there are no saddles on the positive real axis.

(b) If $-1  < \xi_2 < -\xi_2^*$ then $z_c(\xi_2) \in (-\beta^2,0)$,
and if $\xi_2^* < \xi_2 < 1$ then $z_c(\xi_2) \in (-\infty,-\alpha^2)$.
Even though $z_c(\xi_2)$ is real, the two saddles with $z$ coordinate
equal to $z_c(\xi_2)$ are not on the real part of
the Riemann surface, and thus in both cases we are in the liquid phase.

(c)  If $-\xi_2^* < \xi_2 < \xi_2^*$ then $z_c(\xi_2) \in (-\alpha^2, -\beta^2)$.
Then the saddles with $z$ coordinate equal to $z_c(\xi_2)$ are on the cycle
$\mathcal C_1$. The branch points $-\alpha^2$ and $-\beta^2$ are the other
two saddles and they are also on the cycle. Thus all four saddles are on 
the cycle $\mathcal C_1$ and they are distinct, and we are in the gas phase.  

(d) If $\xi_2 = -\xi_2^*$ then $z_{c}(\xi_2) = -\beta^2$ and 
if $\xi_2 = \xi_2^*$ then  $z_c(\xi_2) = -\alpha^2$. In both cases
there is a triple saddle point at one of the branch points, and we
are in the liquid-gas transition. 		
\end{proof}
		 
\subsection{Algebraic equation}
The condition of coalescing saddle points leads to an algebraic equation
for $\xi_1$ and $\xi_2$. We are able to  calculate it with the help of Maple.

First of all, the saddle point equation $\Phi'(z) dz = 0$, see \eqref{eq:MeroDiff}, 
leads us to  consider
\[ \frac{2}{z-1} - \frac{(1+\xi_2)}{z} + 
	\xi_1 \frac{\lambda'(z)}{\lambda(z)} = 0 \]
which after clearing denominators, and using 
\eqref{eq:lambda12} and \eqref{eq:rho12},  gives a polynomial equation
in $z$ and $y=\sqrt{z(z+\alpha^2)(z+\beta^2)}$. We eliminate the square
root to obtain a polynomial equation in $z$ of degree $4$, which is
\begin{multline} \label{eq:saddleEq} 
	(1-\xi_2)^2 z^4  
	+ ((\alpha^2+\beta^2) ((1-\xi_2)^2-\xi_1^2) + 2 (1-\xi_2^2 -\xi_1^2)) z^3 \\
	+ (2(\alpha^2+ \beta^2)(1-\xi_2^2-\xi_1^2) + 
	2-4\xi_1^2 +2 \xi_2^2)) z^2 \\
	+ ((\alpha^2+\beta^2)((1+\xi_2)^2-\xi_1^2) + 2 (1-\xi_2^2-\xi_1^2))z 	+  (1+\xi_2)^2 = 0
	\end{multline}
By definition, the saddles are the four zeros of the polynomial
\eqref{eq:saddleEq}.

The discriminant with respect to $z$ of \eqref{eq:saddleEq} 
is a polynomial in $\xi_1$ and $\xi_2$ that has trivial factors $\xi_1^2$ 
and $\xi_2^2$. The remaining factor is a degree $8$ polynomial, which is 
symmetric in the two variables. Setting this to zero, we obtain the following
equation for coalescing saddles: 
\begin{multline} \label{eq:AlgEq}
	(\alpha^2+1)^6 (\xi_1^8 + \xi_2^8) \\
	- 4(\alpha^2+1)^4 (\alpha^2+2\alpha-1)(\alpha^2-2\alpha-1) 
	(\xi_1^6 \xi_2^2 + \xi_1^2 \xi_2^6) \\ 
	-4 (\alpha^2+1)^4 (\alpha^4 - \alpha^2+1) (\xi_1^6 + \xi_2^6) \\
	+ 2(\alpha^2+1)^2 (3 \alpha^8 - 20 \alpha^6 +82 \alpha^4 -20 \alpha^2 +3) 
	\xi_1^4 \xi_2^4 \\
	+ 4(\alpha^2+1)^2 (\alpha^8 +17 \alpha^6 -48 \alpha^4 + 17 \alpha^2 + 1)
	(\xi_1^4 \xi_2^2 + \xi_1^2 \xi_2^4) \\
	+ 6(\alpha^4-1)^2 (\alpha^4+1)  (\xi_1^4 + \xi_2^4) \\
	+ 4(\alpha^2-1)^2 (\alpha^8 - 22 \alpha^6 -42 \alpha^4 - 22 \alpha^2 + 1) \xi_1^2 \xi_2^2 \\
	-4 (\alpha^2-1)^4  (\alpha^2+\alpha+1) (\alpha^2-\alpha+1) 
	(\xi_1^2 + \xi_2^2) \\
	+ (\alpha^2-1)^6 = 0. 
	\end{multline}
	
Up to a multiplicative constant, the equation \eqref{eq:AlgEq} coincides 
with the one given by Chhita and Johansson \cite[Appendix A]{CJ}.
See also \cite[Section 8]{Pr} for an equation 
that corresponds to \eqref{eq:AlgEq} with $\alpha = 2$
up to a change of variables.
For $\alpha = 1$, \eqref{eq:AlgEq} 
reduces  (up to a numerical factor) to
\[ (1- \xi_1^2 -\xi_2^2) (\xi_1^2 + \xi_2^2)^3 =0 \]
and the real section is the unit circle.

\begin{remark}
The discriminant of \eqref{eq:saddleEq} also vanishes for 
$\xi_1 = 0$ or $\xi_2=0$, and there is indeed a double root 
of \eqref{eq:saddleEq} for  these values of the parameters. However, 
they do not correspond to coalescing saddle points.
For $\xi_1 =0$, the double root is at  $z= z_c(\xi_2)$ from \eqref{eq:defzc}.
which corresponds to two different saddles on the Riemann surface $\mathcal R$,
unless $\xi_2 = \pm \frac{\alpha - \beta}{\alpha + \beta}$,
see Lemma \ref{lem:phases0}.  
For $\xi_2=0$, the double root turns out to be at $z=-1$, 
but the saddles are also on different sheets of $\mathcal R$. 
\end{remark}

For $\alpha > 1$, the  real section of \eqref{eq:AlgEq} has two components, 
as shown in  Figure~\ref{fig:AlgCurve}, 
both contained in the square $-1 \leq \xi_1, \xi_2 \leq 1$.
The outer component is a smooth closed curve that touches the square in
the points $(\pm 1,0)$ and $(0,\pm 1)$. 
The inner component is a closed curve with four cusps at locations
$\left(\pm \tfrac{\alpha-\beta}{\alpha+\beta}, 0\right)$, and 
$\left(0,\pm \tfrac{\alpha-\beta}{\alpha+\beta}\right)$.
It can indeed be checked that for $\xi_2=0$, the equation \eqref{eq:AlgEq}
factorizes as
\[ (\xi_1^2-1) \left( (\alpha^2+1)^2 \xi_1^2 - (\alpha^2-1)^2 \right)^3 = 0 \]
and it has solutions $\xi_1 = \pm 1$ (with multiplicity $1$)
and $\xi_1 = \pm \frac{\alpha^2-1}{\alpha^2+1} = 
\pm \frac{\alpha-\beta}{\alpha+\beta}$ (with multiplicity $3$). 

\begin{proposition} \label{prop:phases}
Let $-1 < \xi_1, \xi_2 < 1$.

\begin{enumerate}
\item[\rm (a)] $(\xi_1, \xi_2) \in \mathfrak G$ (gas phase)
if and only if  $(\xi_1, \xi_2)$ is inside the inner component  
of the algebraic curve.
\item[\rm (b)] $(\xi_1,\xi_2) \in \mathfrak L$ (liquid phase)
if and only $(\xi_1,\xi_2)$ is outside the inner component and inside the
outer component. 
\item[\rm (c)] $(\xi_1, \xi_2) \in \mathfrak S$ (solid phase)
if and only if $(\xi_1, \xi_2)$ is outside the outer component.
\end{enumerate}
\end{proposition}
\begin{proof}
If $\xi_1 = 0$ then all statements of the proposition follow
from Lemma~\ref{lem:phases0}.

The proof in the general cases follows by a continuity  argument,
since the saddles depend continuously on the parameters $\xi_1, \xi_2$,
and a saddle can only leave the real part of the Riemann surface if it 
coalesces with another saddle and then the pair can move away from 
the real part. This transition
can thus only occur for $\xi_1, \xi_2$ satisfying the algebraic
equation \eqref{eq:AlgEq}.
Note that this argument also applies to the point at infinity,
since by definition \eqref{eq:RSreal} the point at infinity 
is included in the real part. 

Combining with Lemma \ref{lem:phases0} we find that 
any point $(\xi_1, \xi_2)$ inside the inner component belongs
to the gas phase, and any point in between the inner and outer
component belongs to the liquid phase.

To treat the solid phase, we check  that any $(\xi_1, \xi_2)$
close enough to a corner point is in the solid phase.
We can see this from the equation \eqref{eq:AlgEq}.
If $\xi_1= \pm 1$ and $\xi_2 = -1$ then \eqref{eq:saddleEq} 
has solutions $z=0$, $z=1$ (and two other solutions that are on the
cycle $\mathcal C_1$), and if $\xi_1 = \pm 1$ and $\xi_2=1$ then
 \eqref{eq:saddleEq}  has solutions $z=\infty$ and $z=1$. 
Thus for each of the four corner points there are two distinct 
saddles on $\mathcal C_2$. 
 Then by continuity this continues to be the case if $(\xi_1, \xi_2)$ is 
 close enough to one of the corner points, and it continues to be so until
 the two saddles on $\mathcal C_2$ coalesce, and this happens on
 the outer component. 

This completes the proof of Proposition \ref{prop:phases}.
\end{proof}

\subsection{Gas phase: steepest descent paths}
In the gas phase all four saddles are located on the cycle $\mathcal C_1$,
and they are all simple. To prepare for the proof of Theorem \ref{thm:gaslimit},
we need more precise information on the location of the saddles.

\begin{lemma}
Suppose $(\xi_1, \xi_2) \in \mathfrak G$.
Then the function 
\begin{equation} \label{eq:RePhi} 
	\Re \Phi(z) = 2\log |z-1| -  (1+\xi_2) \log |z| + \xi_1 \log \left| \lambda(z) \right|, \quad z \in \mathcal C_1,
	\end{equation}
attains a local minimum at two of the saddles, 
say $z_{s,1}$ and $z_{s,2}$,  where $z_{s,j}$ is on the $j$th sheet for $j=1,2$.

It attains a local maximum on $\mathcal C_1$ at the other two saddles 
$z_{s,3} < z_{s,4}$.
If $\xi_1 > 0$ then $z_{s,3}$ and $z_{s,4}$ are on the first sheet
and $-\alpha^2 < z_{s,3} < z_{s,1} < z_{s,4} < -\beta^2$, 
if $\xi_1 < 0$ then $z_{s,3}$ and $z_{s,4}$ are on the second sheet and
and $-\alpha^2 < z_{s,3} < z_{s,2} < z_{s,4} < -\beta^2$, and
if $\xi_1 = 0$ then $z_{s,3} = -\alpha^2$ and $z_{s,4} = -\beta^2$ are at
the branch points.
\end{lemma}

\begin{proof}
The lemma is easy to verify if $\xi_1 = 0$, since
$z \mapsto 2\log |z-1| -  (1+\xi_2) \log |z|$ attains a local minimum at
$z = z_c(\xi_2)$ given by \eqref{eq:defzc}. Thus 
$z_{s,j} = \left(z_c(\xi_2)\right)^{(j)}$ for $j=1,2$, and the other saddles $z_{s,3}$
and $z_{s,4}$ are at the branch points. 

For $\xi_1 \neq 0$, we notice that both $\lambda_1$ and $\lambda_2$ 
are real and negative on the interval $[-\alpha^2,-\beta^2]$, see \eqref{eq:lambdareal3},
with a square root behavior at endpoints (which follows from \eqref{eq:lambda12}
and the square roots in \eqref{eq:rho12}). Thus $\lambda_{1,2}'$ become
infinite at the endpoints, and a closer inspection of \eqref{eq:rho12}, \eqref{eq:lambda12}
shows that 
\begin{equation} \label{eq:lambda1B} 
	\lim_{z \to -\alpha^2+}  \frac{\lambda_1'(z)}{\lambda_1(z)} = + \infty, \qquad 
	\lim_{z \to -\beta^2-}  \frac{\lambda_1'(z)}{\lambda_1(z)} = - \infty, 	 
	 \end{equation}
and
\begin{equation} \label{eq:lambda2B} 
	\lim_{z \to -\alpha^2+}  \frac{\lambda_2'(z)}{\lambda_2(z)} = - \infty, \qquad
	 \lim_{z \to -\beta^2-}  \frac{\lambda_2'(z)}{\lambda_2(z)} = + \infty.
	 \end{equation}
Thus if $\xi_1 \neq 0$, both functions
\begin{equation} \label{eq:dPhij} 
	\Phi_j'(z) = \frac{2}{z-1} - \frac{1+\xi_2}{z} + 
	\xi_1 \frac{\lambda'(z)}{\lambda_j(z)}, \qquad j=1,2, 
	\end{equation}
are infinite at the endpoints of the interval $[-\alpha^2, -\beta^2]$
but with opposite signs. By continuity there is an odd number of zeros 
for each of them. There are exactly four simple saddles on the cycle $\mathcal C_1$
as we are in the gas phase, and therefore one of
$\Phi_j'$, $j=1,2$, has three simple zeros and the other one has one simple zero.

We already noted that for $\xi_1 =0$ 
\begin{equation} \label{eq:RePhij} 
	\Re \Phi_j(z) =  2\log |z-1| -  (1+\xi_2) \log |z| + \xi_1 \log |\lambda_j(z)| 
	\end{equation}
attains a local minimum at an interior saddle for $j=1,2$. Because of analytic
dependence on parameters this continues to be the case for $\xi_1 \neq 0$,
and in fact, since there is no coalescence of saddle points, it remains
true for every $(\xi_1, \xi_2) \in \mathfrak{G}$. Thus saddles $z_{s,1}$ and 
$z_{s,2}$ where $\Re \Phi$ has a local minimum exist, and $z_{s,j}$ is on
the $j$th sheet.

Now, if $\xi_1 > 0$ then from \eqref{eq:lambda1B} and \eqref{eq:dPhij}
it follows that 
\[ \lim_{z \to -\alpha^2+} \Phi_1'(z) = +\infty, \qquad
	\lim_{z \to -\beta^2-} \Phi_1'(z) = -\infty \] 
and so $\Re \Phi_1$ increases on an interval $[-\alpha^2, -\alpha^2 + \delta]$ 
and decreases on $[-\beta^2-\delta, - \beta^2]$ for some $\delta > 0$. Since there is
a local minimum at $z_{s,1}$ on the first sheet, it should be 
that $\Re \Phi_1$ has two local maxima, say $z_{s,3} < z_{s,4}$, with
$-\alpha^2 < z_{s,3} < z_{s,1} < z_{s,4} < -\beta^2$.

In case $\xi_1 < 0$ we find in the same way that the local maxima
are on the second sheet, with
$-\alpha^2 < z_{s,3} < z_{s,2} < z_{s,4} < -\beta^2$.
\end{proof}

The \textbf{path of steepest descent} from the saddle $z_{s,j}$, $j=1,2$,
is the curve $\gamma_{sd,j}$ through $z_{s,j}$ where
the imaginary part of $\Phi_j$ is constant and the real part decays
if we move away from the saddle. Since $\Re \Phi_j$ on $\mathcal C_1$ 
has a local minimum at $z_{s,j}$, the path of steepest descent meets
the real line at a right angle.  

Emanating from $z_{s,j}$, $j=1,2$ are also curves $\gamma_{l,j}$ and
$\gamma_{r,j}$ where the real part is constant ($l$ stands for left, and
$r$ stands for right). The curve $\gamma_{l,j} $
emanates from  $z_{s,j}$ at angles $\pm 3\pi/4$. It consists of a
part in the upper half plane and its mirror image with respect to
the real line in the lower half plane. 
Similarly, $\gamma_{r,j}$ emanates from $z_{s,j}$ 
at angles $\pm \pi/4$, and it is also symmetric in the real line. 
Near $z_{s,j}$ we have that $\gamma_{l,j}$ is to the left of
the steepest descent path $\gamma_{sd,j}$ and $\gamma_{r,j}$ is
to the right. 

Then $\gamma_{l,j}$ and $\gamma_{r,j}$ are parts of the boundary
of the domains:
\begin{align} \label{eq:defOmegajmin}
	\Omega_j^{-} & = \{ z \in \mathbb C \mid \Re \Phi_j(z) < \Re \Phi_j(z_{s,j}) \}, \\
	\label{eq:defOmegajplus}
	\Omega_j^+ & = \{ z \in \mathbb C \mid \Re \Phi_j(z) > \Re \Phi_j(z_{s,j}) \}.
\end{align}

\begin{lemma}  \label{lem:lemma57}
Suppose $(\xi_1, \xi_2) \in \mathfrak G$, and $j =1,2$.
Then the following hold.
\begin{enumerate}
\item[\rm (a)] 
All three curves
$\gamma_{sd,j}$, $\gamma_{l,j}$ and $\gamma_{r,j}$ are
simple closed curves enclosing the interval $[-\beta^2,0]$.
\item[\rm (b)]
The steepest descent path $\gamma_{sd,j}$ intersects the positive real line at $1$,
$\gamma_{l,j}$ intersects the positive real line at a value $> 1$
while $\gamma_{r,j}$ intersects the positive real line at a value $< 1$.
\item[\rm (c)]
$\Omega_j^-$ is a  bounded open set with at most three connected components.
One component (which we call the main component) 
contains $\gamma_{sd,j} \setminus \{ z_{s,j} \}$.
There are at most two other connected components, 
namely a  component containing $-\beta^2$ (if $\Re \Phi_j(-\beta^2) < \Re \Phi_j(z_{s,j})$)
and a component containing $-\alpha^2$ (if $\Re \Phi_j(-\alpha^2) < \Re \Phi_j(z_{s,j})$).
The other components (if they exist) are at positive distance from the main component.
\item[\rm (d)] $\Omega_j^+$ is an open set
with an unbounded component that contains a contour $\Gamma_{1,j}$
that goes around $(-\infty,-\alpha^2]$ 
and with a bounded component that contains a contour $\Gamma_{2,j}$
going around the interval $[-\beta^2,0]$.
\end{enumerate}
\end{lemma}

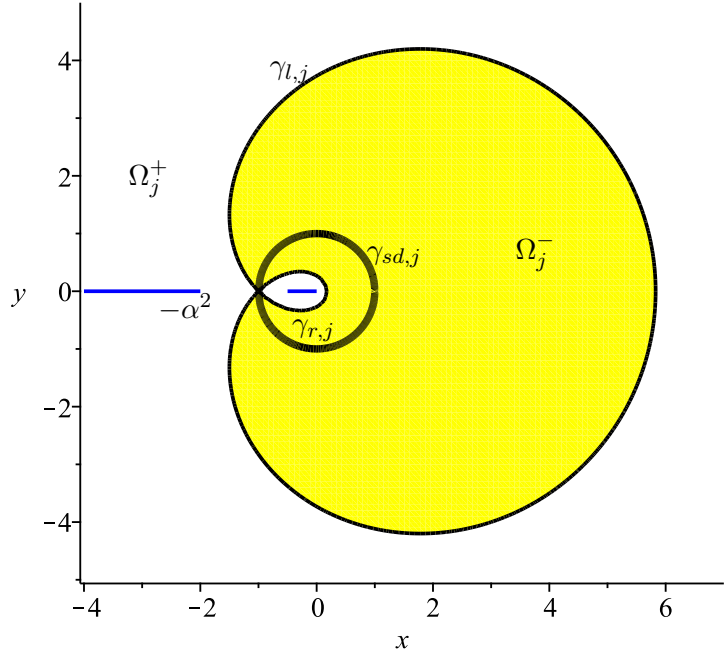
\begin{figure}[t]
 \begin{center}
 \begin{overpic}[scale=0.5]{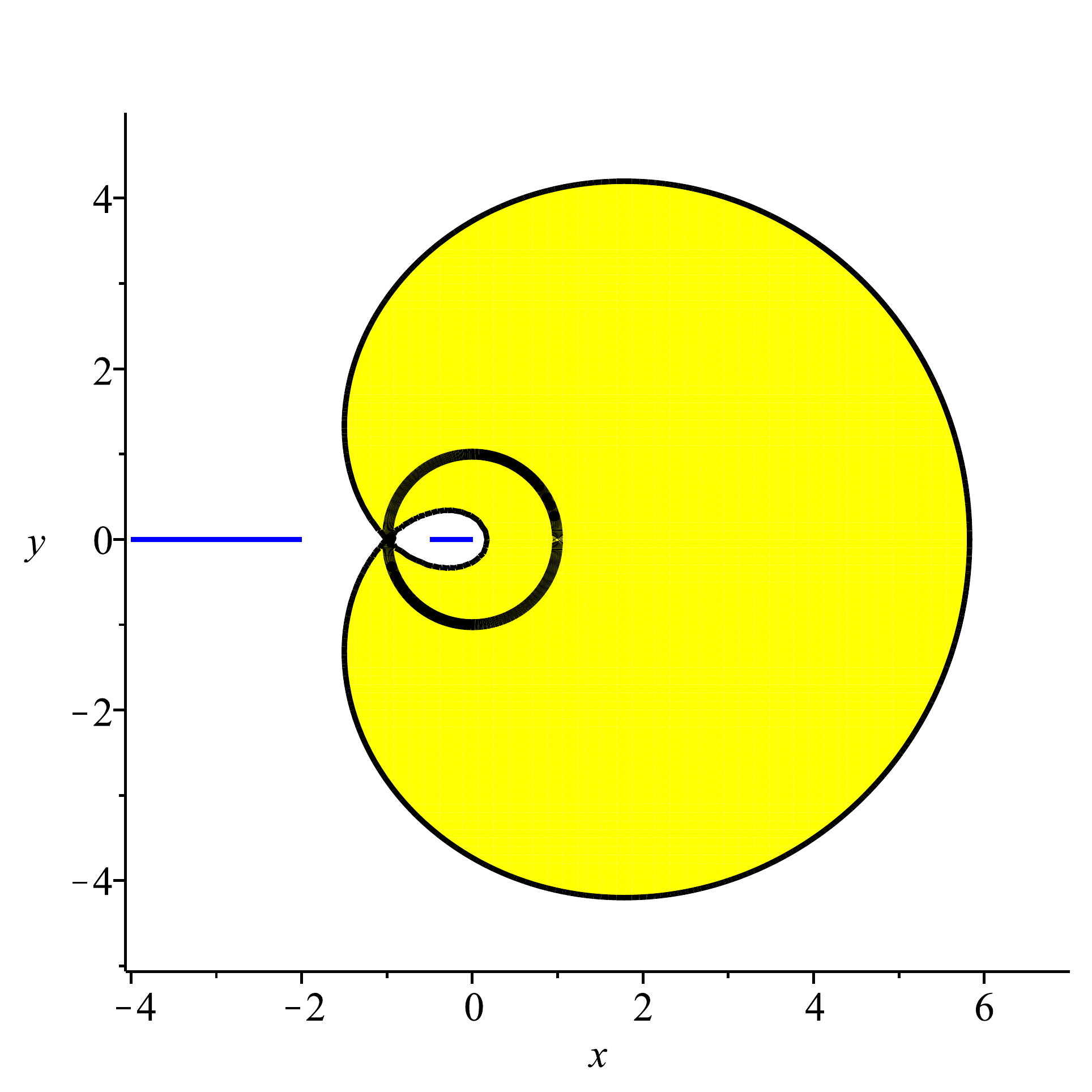}
 	\put(50,55){$\gamma_{sd,j}$}
 	\put(40,45){$\gamma_{r,j}$}
 	\put(37,80){$\gamma_{l,j}$}
 	\put(22,47){$ -\alpha^2$}
 	\put(70,55){$\Omega_j^-$}
 	\put(18,65){$\Omega_j^+$}
 \end{overpic} 
\caption{Illustration for Lemma \ref{lem:lemma57}. 
The figure shows the steepest descent curve $\gamma_{sd,j}$ (dark curve) from the saddle $z_{s,j}$ and the level lines $\gamma_{l,j}$ and $\gamma_{r,j}$ of $\Re \Phi_j$
that enclose the domain $\Omega_j^-$ for the case $\xi_1=0$.
In this case $\Omega_j^-$ has one component, but for $\xi_1 \neq 0$,
there could be a component containing $-\alpha^2$, and a component containing
$-\beta^2$. \label{eq:Omegaj}}
\end{center}
\end{figure}
 
\begin{proof}
The lemma is straightforward to verify if $\xi_1 = 0$, since in that case
\begin{equation} \label{eq:RePhi12} 
	\Re \Phi_1(z) = \Re \Phi_2(z) = 2 \log |z-1| - (1+\xi_2) \log |z| 
	\end{equation}
and $z_{s,1} = z_{s,2} = z_c(\xi_2) = 
- \frac{1+\xi_2}{1-\xi_2}$ which is in $(-\alpha^2,-\beta^2)$ since
we are in the gas phase.
Then $\gamma_{l,j}$, $\gamma_{sd,j}$, and $\gamma_{r,j}$ 
are independent of $j$ and we have a situation as in Figure \ref{eq:Omegaj}. 
The curves $\gamma_{l,j}$ and $\gamma_{r,j}$ enclose the domain $\Omega_j^-$ 
that is shaded in the figure, and $\Omega_j^-$ has only one component in this case.

The non-shaded domain $\Omega_j^+$ has an unbounded component with boundary
$\gamma_{l,j} \cup (-\infty,-\alpha^2]$ and a bounded component with boundary 
$\gamma_{r,j} \cup [-\beta^2,0]$. The contours $\Gamma_{1,j}$ and $\Gamma_{2,j}$
can be taken in $\Omega_j^+$ as specified in part (d) of the lemma.

From \eqref{eq:RePhi12} it is easy to see that $\Re \Phi_j$ is strictly decreasing on 
$(-\infty, z_c(\xi_2)]$,  strictly increasing on $(z_c(\xi_2),0]$ with
value $+\infty$ at $0$, then again strictly decreasing on $[0,1]$ with value
$-\infty$ at $1$, and finally again strictly increasing on $[1, \infty)$
with limiting value $+\infty$ at $\infty$. 

For $\xi_1 \neq 0$ the behavior on the positive real axis is the same:
for both $j=1$ and $j=2$, we have that $\Phi_j(x)$ is real for $x > 0$ and 
it decreases from $+\infty$
to $-\infty$ on the interval $[0,1]$ and increases from $-\infty$ to $+\infty$
on $[1,\infty)$. This is so because otherwise there would be a zero of the
derivative, which would be a saddle point, but all saddle points are on
the cycle $\mathcal C_1$ since we are in the gas phase.

The behavior of $\Re \Phi_j$ on the cuts $(-\infty, -\alpha^2] \cup [-\beta^2,0]$
is exactly the same as what we have for $\xi_1 = 0$. Indeed 
it does not depend on $\xi_1$ at all, since $|\lambda_{j,\pm}(x)| = 1$ for 
$x$ on the cuts, see \eqref{eq:lambdareal1}  and \eqref{eq:RePhij}. Thus
\begin{equation} \label{eq:RePhijincrease}
\begin{aligned}
	&\Re \Phi_j(x) \text{ is strictly decreasing for } x \in (-\infty,-\alpha^2], \\
	&\Re \Phi_j(x) \text{ is strictly increasing for } x \in [-\beta^2,0].
	\end{aligned}
	\end{equation}

Now let's follow the paths $\gamma_{l,j}$ and $\gamma_{r,j}$, where the real part 
$\Re \Phi_j$ is constant,  as they move away from $z_{s,j}$ into the upper half plane.
These paths remain bounded, since $\Re \Phi_j(z) \to +\infty$ as $|z| \to \infty$, 
which follows from \eqref{eq:RePhij}, Lemma \ref{lem:lambdarho} (c), and the
fact that $\xi_2 < 1$. 
The two paths cannot come together in the upper half plane, since then
they would enclose a domain on which $\Re \Phi_j$ is harmonic and constant on the
boundary, which violates the maximum/minimum principle of harmonic functions. 
For the same reason they cannot meet at a point on the positive real axis.

Suppose now one of the paths comes to the cut $(-\infty,-\alpha^2]$
at a point $q$. Then this path, together with its mirror image
in the real line, encloses a bounded domain $D$ and $\Re \Phi_j$ has
the constant value $\Re \Phi_j(z_{s,j})$ on its boundary. 
The value of $\Re \Phi_j$ is smaller on the interval $[q,-\alpha^2]$ because of  \eqref{eq:RePhijincrease}.
Also $\Re \Phi_j$ is harmonic on $D \setminus [q,-\alpha^2]$ and it
follows from the maximum principle for harmonic functions that
\[ \Re \Phi_j(z) < \Re \Phi_j(z_{s,j}) \qquad \text{ for every } z \in D. \]
This is a contradiction, since $z_{s,j}$ is a saddle at which 
$\Re \Phi_j$ attains a local minimum when restricted to the cycle $\mathcal C_1$.

We arrive at a similar contradiction if one of the paths $\gamma_{l,j}$
and $\gamma_{r,j}$ comes to the cut $[-\beta^2,0]$.

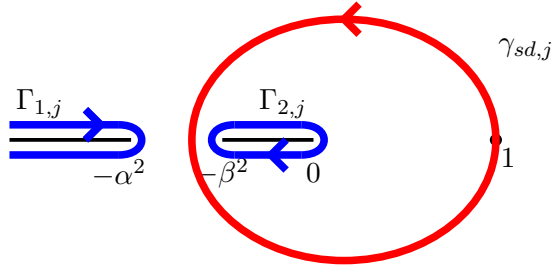
\begin{figure}[t]
\begin{center}
\begin{tikzpicture}[scale=0.8]

\draw [line width=0.5mm] (-6,0)--(-4,0);
\draw (-4.2,-0.55) node {$-\alpha^2$};

\draw [line width= 0.5mm] (-2.5,0)--(-1,0);
\draw (-2.5,-0.55) node {$-\beta^2$};
\draw (-1,-0.55) node {$0$};

\fill (2,0) circle [radius=1mm];
\draw (2.2,-0.3) node {$1$};

\draw (2.5,1.5) node {$\gamma_{sd,j}$};
\draw [red,line width=1mm] (-0.5,0) ellipse (2.5cm and 2cm);
\draw [red,line width=1mm] (-0.25,1.75)--(-0.5,2)--(-0.25,2.25);

\draw [blue,line width = 1mm] (-6,0.25)--(-4.2,0.25);
\draw [blue,line width = 1mm] (-6,-0.25)--(-4.2,-0.25);
\draw [blue,line width = 1mm] (-4.2,-0.25) .. controls (-3.7,-0.25) and (-3.7,0.25) .. (-4.2,0.25);
\draw [blue,line width = 1mm] (-4.75,0.0)--(-4.5,0.25)--(-4.75,0.5);

\draw [blue,line width = 1mm] (-2.3,0.25)--(-1.2,0.25);
\draw [blue,line width = 1mm] (-2.3,-0.25)--(-1.2,-0.25);
\draw [blue,line width = 1mm] (-2.3,-0.25) .. controls (-2.8,-0.25) and (-2.8,0.25) .. (-2.3,0.25);
\draw [blue,line width = 1mm] (-1.2,-0.25) .. controls (-0.7,-0.25) and (-0.7,0.25) .. (-1.2,0.25);
\draw [blue,line width = 1mm] (-4.75,0.0)--(-4.5,0.25)--(-4.75,0.5);

\draw [blue,line width = 1mm] (-1.45,-0.0)--(-1.7,-0.25)--(-1.45,-0.5);
\draw (-5.5,0.6) node {$\Gamma_{1,j}$};
\draw (-1.5,0.6) node {$\Gamma_{2,j}$};

\end{tikzpicture} 
\end{center}

\caption{Steepest descent contour $\gamma_{sd,j}$ and contours
$\Gamma_{1,j}$ and $\Gamma_{2,j}$ in case $\Omega_{j}^-$ has
only one component.
\label{fig:Contours3}}
\end{figure}

Thus the two paths can only leave the upper half-plane via the positive
real axis. Since $\Phi_j$ is strictly decreasing on $[0,1]$ and 
strictly increasing on $[1,\infty]$, we must conclude that 
$\gamma_{l,j}$ comes to the unique value in $(1,\infty)$
and $\gamma_{r,j}$ to the  unique value in $(0,1)$ 
where $\Phi_j$ is equal to $\Re \Phi_j(z_{s,j})$.
Together with their mirror images in the real axis, they both form
simple closed curves that go around the cut $[-\beta^2,0]$.
They enclose a domain where $\Re \Phi_j$ is smaller
than $\Re \Phi_j(z_{s,j})$, i.e., it is contained 
in $\Omega_j^-$, see \eqref{eq:defOmegajmin}.

The path of steepest descent $\gamma_{sd,j}$ lies in $\Omega_j^-$,
and $\Re \Phi_j$ decreases along $\gamma_{sd,j}$ in the upper half plane.
It will meet the positive real axis at $1$, where $\Re \Phi_j$ is $-\infty$,
since if it would meet the positive real axis at some other point, this point
would be a saddle, but there is no saddle on $\mathcal C_2$.

We now proved parts (a) and (b) of Lemma \ref{lem:lemma57}.
We also proved that the component of $\Omega_j^-$ that is bounded
by $\gamma_{l,j}$ and $\gamma_{r,j}$ contains the steepest
descent curve $\gamma_{sd,j} \setminus \{z_{s,j}\}$.
We also see that $\Omega_j^-$ is bounded since $\Re \Phi_j(z) \to +\infty$ 
as $|z| \to \infty$.

Any other connected component of $\Omega_j^-$ has to intersect with the
branch cuts $(-\infty,-\alpha^2] \cup [-\beta^2,0]$, since otherwise
we have again a contradiction with the minimum principle for harmonic
functions. 

If a component of $\Omega_j^-$ intersects $(-\infty, -\alpha^2]$ then it will do
so along an interval $(q,-\alpha^2]$ for some $q < -\alpha^2$, 
because of \eqref{eq:RePhijincrease}.
Hence there can be at most one such component, and this exists if and only if
if $\Re \Phi_j(-\alpha^2) < \Re \Phi_j(z_{s,j})$.
Similarly, there is at most one component that intersects $[-\beta^2,0]$,
and thus in total there are at most three components.

Finally, since $z_{s,j}$ is a simple saddle, and $\Re \Phi_j$ attains a minimum
at $z_{s,j}$ when we restrict to the cycle $\mathcal C_1$, there is 
$\delta > 0$ such that
\[ (z_{s,j}-\delta, z_{s,j}) \cup  (z_{s,j}, z_{s,j} +\delta)
	\subset \Omega_j^+. \]
Thus the boundary of another component (if it exists) intersects 
the interval $(-\alpha^2, -\beta^2)$
at a point different from $z_{s,j}$, and thus the component stays at
positive distance from the main component. 
This proves part (c) of the lemma.

The component of $\Omega_j^-$ that contains  $-\alpha^2$ (if it exists)
is bounded and belongs to the exterior of the closed contour $\gamma_{l,j}$. However,
by part (c), it remains at a positive distance from $\gamma_{l,j}$,
and therefor we can find a contour $\Gamma_{1,j}$ as in the statement
of part (d) that goes around the interval $(-\infty,-\alpha^2]$ and
avoids the closure of the domain $\Omega_j^-$ while staying to the left of
$\gamma_{l,j}$.
Similarly, we can find $\Gamma_{2,j}$ as in part (d).  See Figures \ref{fig:Contours3}
and \ref{fig:Contours4} for plots of $\gamma_{sd,j}$ 
and $\Gamma_{1,j}$ and $\Gamma_{2,j}$ in the situations where $\Omega_j^-$
has no component containing $-\alpha^2$ or $-\beta^2$ (Figure \ref{fig:Contours3})
and where $\Omega_j^-$ has a component containing $-\alpha^2$ that $\Gamma_{1,j}$
should avoid (Figure \ref{fig:Contours4}). 

This completes the proof of Lemma \ref{lem:lemma57}.
\end{proof}

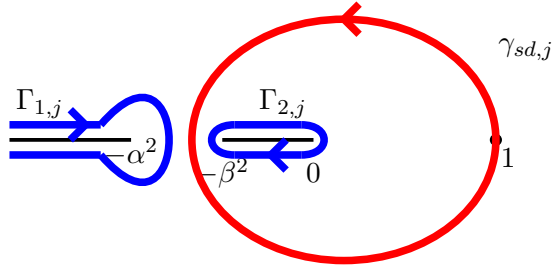
\begin{figure}[t]
\begin{center}
\begin{tikzpicture}[scale=0.8]

\draw [line width=.5mm] (-6,0)--(-4,0);
\draw (-4,-0.2) node {$-\alpha^2$};

\draw [line width=.5mm] (-2.5,0)--(-1,0);
\draw (-2.5,-0.55) node {$-\beta^2$};
\draw (-1,-0.55) node {$0$};

\fill (2,0) circle [radius=1mm];
\draw (2.2,-0.3) node {$1$};

\draw (2.5,1.5) node {$\gamma_{sd,j}$};
\draw [red,line width=1mm] (-0.5,0) ellipse (2.5cm and 2cm);
\draw [red,line width=1mm] (-0.25,1.75)--(-0.5,2)--(-0.25,2.25);

\draw [blue,line width = 1mm] (-6,0.25)--(-4.5,0.25);
\draw [blue,line width = 1mm] (-6,-0.25)--(-4.5,-0.25);
\draw [blue,line width = 1mm] (-4.5,-0.25) .. controls (-3,-2) and (-3,2) .. (-4.5,0.25);
\draw [blue,line width = 1mm] (-5,0.0)--(-4.75,0.25)--(-5,0.5);

\draw [blue,line width = 1mm] (-2.3,0.25)--(-1.2,0.25);
\draw [blue,line width = 1mm] (-2.3,-0.25)--(-1.2,-0.25);
\draw [blue,line width = 1mm] (-2.3,-0.25) .. controls (-2.8,-0.25) and (-2.8,0.25) .. (-2.3,0.25);
\draw [blue,line width = 1mm] (-1.2,-0.25) .. controls (-0.7,-0.25) and (-0.7,0.25) .. (-1.2,0.25);
\draw [blue,line width = 1mm] (-1.45,-0.0)--(-1.7,-0.25)--(-1.45,-0.5);
\draw (-5.5,0.6) node {$\Gamma_{1,j}$};
\draw (-1.5,0.6) node {$\Gamma_{2,j}$};

\end{tikzpicture} 
\end{center}

\caption{Steepest descent contour $\gamma_{sd,j}$ and contours
$\Gamma_{1,j}$ and $\Gamma_{2,j}$ in case $\Omega_{j}^-$ has 
a component containing $-\alpha^2$, but no component containing
$-\beta^2$.
\label{fig:Contours4}}
\end{figure}

\subsection{Gas phase: proof of Theorem \ref{thm:gaslimit}}
\label{subsec:gaslimit}

We are now ready to prove Theorem \ref{thm:gaslimit}.

\begin{proof}
To establish Theorem \ref{thm:gaslimit} we are going to show
that in the gas phase the two double integrals in 
\eqref{eq:KNremainder} tend to $0$ as $N \to \infty$
at an exponential rate.

In the situation of Theorem \ref{thm:gaslimit}
we have $m = (1+\xi_1)N + o(N)$, $n=(1+\xi_2) N+o(N)$,
$m' = (1+\xi_1)N + o(N)$, $n'=(1+\xi_2) N+o(N)$,
so that in \eqref{eq:KNremainder} 
we should replace $\xi_1$ and $\xi_1'$ by $\xi_1 + o(1)$,
and $\xi_2$ and $\xi_2'$ by $\xi_2 + o(1)$. 
The $o(1)$ terms give only a subleading contribution,
and it is enough to prove that both integrals
\begin{equation} \label{eq:KNremainder1} 
\frac{1}{(2\pi i)^2} \oint_{\gamma_{0,1}} \frac{dz}{z}
	\int_{\Gamma} \frac{dw}{z-w} F(w) F(z)
		e^{N(\Phi_1(z) - \Phi_1(w))/2} 
	\end{equation}
	and
\begin{equation} \label{eq:KNremainder2} 
\frac{1}{(2\pi i)^2} \oint_{\gamma_{0,1}} \frac{dz}{z}
	\int_{\Gamma} \frac{dw}{z-w} F(w) (I_2- F(z))
		e^{N(\Phi_2(z) - \Phi_1(w))/2} 
	\end{equation}
are exponentially small as $N \to \infty$. Here we write  again
$\Phi_{1}$ and $\Phi_2$ instead of $\Phi_1(\cdot; \xi_1,\xi_2)$
and $\Phi_2(\cdot; \xi_1, \xi_2)$.

We move the contour $\gamma_{0,1}$ in \eqref{eq:KNremainder1} to the
steepest descent path $\gamma_{sd,1}$ through $z_{s,1}$, and we have
\begin{equation} \label{eq:RePhi1z} 
	\Re \Phi_1(z) \leq \Re \Phi_1(z_{s,1}), \qquad z \in \gamma_{sd,1}. 
	\end{equation}
The contour $\Gamma$ in \eqref{eq:KNremainder1} traverses the interval 
$(-\alpha^2, -\beta^2)$ twice, but in opposite directions. The integrand has no branching
on this interval, and therefore the two contributions cancel out.
We can thus replace $\Gamma$ by two contours $\Gamma_{1,1}$ and $\Gamma_{2,1}$
as in Lemma \ref{lem:lemma57} (d), where 
$\Gamma_{1,1}$ goes around $(-\infty, -\alpha^2]$, $\Gamma_{2,1}$ goes around
$[-\beta^2,0]$, and
\begin{equation} \label{eq:RePhi1w} 
	\Re \Phi_1(w) > \Re \Phi_1(z_{s,1}), \qquad w \in \Gamma_{1,1} \cup \Gamma_{2,1} 
	\end{equation}
since the contours are in $\Omega_1^+$, see \eqref{eq:defOmegajplus},
and $\Re \Phi_1(w) = (1-\xi_2) \log|w| + O(1)$ as $w \to \infty$.

By \eqref{eq:RePhi1z}, \eqref{eq:RePhi1w} the
factor $e^{N(\Phi_1(z)-\Phi_1(w))/2}$ is $O(e^{-cN})$ for some $c > 0$, uniformly 
for  $(z,w) \in \gamma_{sd,1} \times (\Gamma_{1,1} \cup \Gamma_{2,1})$. 
Thus \eqref{eq:KNremainder1} tends to $0$ at an exponential rate as $N \to \infty$.

We are going to apply a similar argument to \eqref{eq:KNremainder2}
and deform $\gamma_{0,1}$ to the steepest descent contour 
$\gamma_{sd,2}$
passing through $z_{s,2}$. 
We again replace $\Gamma$ by the union of two contours,
one going around $(-\infty,-\alpha^2]$ and the other around $[-\beta^2,0]$,
but now we are going to move these contours to the cuts, where we note
that $\lambda_{1,\pm} = \lambda_{2,\mp}$
and $\Phi_{1,\pm} = \Phi_{2,\mp}$ plus a purely imaginary constant
(depending on the precise branches of the logarithms that we choose in \eqref{eq:defPhi}).
We also note that there is no pole at $w=0$, and that $I_2-F(w)$
is the analytic continuation of $F(w)$ to the second sheet. 
Then we deform the contours further to $\Gamma_{1,2}$ and $\Gamma_{2,2}$ as in Lemma \ref{lem:lemma57} (d),
and we see that \eqref{eq:KNremainder2} is equal to
\begin{multline} \label{eq:KNremainder3}
- \frac{1}{(2\pi i)^2} \oint_{\gamma_{sd,2}} \frac{dz}{z}
	\int_{\Gamma_{1,2} \cup \Gamma_{2,2}} \frac{dw}{z-w}
		(I_2- F(w)) (I_2- F(z))
		e^{N( \Phi_2(z)-\Phi_2(w))/2}.
		\end{multline}
The minus sign comes since we use the orientation on $\Gamma_{1,2}$
and $\Gamma_{2,2}$ as in Figure \ref{fig:Contours4}.

Analogous to \eqref{eq:RePhi1z} and \eqref{eq:RePhi1w} we now have
\[ \Re \Phi_2(z) \leq \Re \Phi_2(z_{s,2}) <
	\Re \Phi_2(w), \qquad z \in \gamma_{sd,2}, \quad  
	w \in \Gamma_{1,2} \cup \Gamma_{2,2}, \]
and the integral \eqref{eq:KNremainder3} tends to $0$ at 
an exponential rate as $N \to \infty$.

This completes the proof of Theorem \ref{thm:gaslimit}.
\end{proof}

\subsection{Cusp points: proof of Theorem \ref{thm:cusplimit}}
\label{subsec:cusplimit}

For the proof of Theorem \ref{thm:cusplimit} we are going
to use \eqref{eq:KNremainder} once more.

In the scaling \eqref{eq:scaling1} of the parameters we have that the 
saddle points coalesce to a triple saddle point at $-\alpha^2$. 
We are going to deform $\gamma_{0,1}$ so that it comes close to $-\alpha^2$.
The main contribution to the integrals in \eqref{eq:KNremainder}
then comes from the triple saddle at $-\alpha^2$, as the integrands are
exponentially small if $w$ and/or $z$ are outside of a small neighborhood
of $-\alpha^2$. This follows as in the proof of the gas phase.

Hence, our task is to investigate how the entries in the integrals in \eqref{eq:KNremainder}
behave for $z$ and $w$ close to $-\alpha^2$. We do this in the following
two lemmas and their corollaries. 
Besides the constants $c_1$ and $c_2$ from \eqref{eq:constants1} we also use
\begin{equation} \label{eq:constants2}
	c_0 = \frac{\alpha + \beta}{\sqrt{2}}.
\end{equation}
\begin{lemma} \label{lem:Phi12}
We have
\begin{multline} \label{eq:Phi12nearsaddle}
	 \Phi_{1,2}(z; \xi_1, \xi_2) \\
 	= 2 \log(\alpha^2+1) - (1+ \xi_2) \log(\alpha^2)
 	+ (1-\xi_2 \pm \xi_1) \pi i \sgn(\Im z)
 	\\
 	\pm \frac{2 \sqrt{\alpha-\beta}}{\sqrt{\alpha+\beta}}
 		\xi_1 \left(\frac{z+\alpha^2}{\alpha^2}\right)^{1/2}
 	+ (\xi_2 - \xi_2^*) \left( \frac{z+\alpha^2}{\alpha^2} \right)
 	+  \frac{1}{(\alpha+\beta)^2} \left(\frac{z+\alpha^2}{\alpha^2}\right)^2 \\
 	+ \xi_1 O\left(z+\alpha^2\right)
 	+ (\xi_2 - \xi_2^*) O\left((z+\alpha^2)^2 \right)
 	+ O\left((z+\alpha^2)^3\right). 
 	\end{multline}
The sign $\pm$ means that $+$ applies to $\Phi_1$ 
and $-$ applies  to $\Phi_2$.
\end{lemma}
\begin{proof}
We recall 
\begin{equation}  \label{eq:Phi1sum}
\begin{aligned}
  \Phi_1(z; \xi_1, \xi_2) 
	& = 2 \log(z-1) - (1+\xi_2) \log z + \xi_1 \log \lambda_1(z) \\
	& = \Phi_1(z; 0, \xi_2^*) - (\xi_2 - \xi_2^*) \log z
		+ \xi_1 \log \lambda_1(z). 
\end{aligned} \end{equation}
All terms are multi-valued because of the logarithms. 
We take principal branches of $\log(z-1)$ and $\log z$, and the analytic 
branch of $\log \lambda_1(z)$ in  $\mathbb C \setminus ((-\infty,0] \cup [1, \infty))$
that is real and positive on $(0,1)$ (recall $\lambda_1(x) > 1$ for $x \in (0,1)$).

Since $\Phi_1(z;0,\xi_2^*)$ has a critical point at $z=-\alpha^2$
with second derivative equal to
\[ \frac{d^2}{dz^2} \Phi_1(z;0,\xi_2^*) = \frac{2}{\alpha^4 (\alpha+ \beta)^2} \]
we have as $z \to -\alpha^2$,
\begin{multline} \label{eq:term1} 
	\Phi_1(z;0,\xi_2^*) 
	=  
	2\log(\alpha^2+1) - (1+\xi_2^*) \log (\alpha^2) +
	(1-\xi_2^*) \pi i \sgn(\Im z) \\
		+ \frac{1}{(\alpha+\beta)^2} \left(\frac{z+\alpha^2}{\alpha^2} \right)^2 
			+ O\left((z+\alpha^2)^3\right). 
\end{multline}
By an expansion of $\log z$ around $z= -\alpha^2$,
\begin{equation} \label{eq:term2}
	\log z
	=  \log (\alpha^2) + \pi i \sgn(\Im z) - 
	\frac{z+\alpha^2}{\alpha^2} + O\left((z+\alpha^2)^2\right). 
		\end{equation}
	
From the formula \eqref{eq:rho12} for $\rho_{1}$
\[ \rho_{1}(z) = - \alpha^2(\alpha + \beta)
	- \alpha \sqrt{\alpha^2-\beta^2} (z+\alpha^2)^{1/2}
		+ O\left(z+ \alpha^2\right) \]
and then from the formula \eqref{eq:lambda12} for $\lambda_{1}$ we get
\[ \lambda_{1}(z) = -1 -
	2\frac{ \sqrt{\alpha-\beta}}{\sqrt{\alpha+\beta}}
	\left(\frac{z+\alpha^2}{\alpha^2} \right)^{1/2} + O(z+ \alpha^2), \]
which confirms the fact that $\lambda_{1}(x) < -1$
for $x \in (-\alpha^2, -\beta^2)$ as stated in \eqref{eq:lambdareal3} of 
Lemma \ref{lem:lambdarho}~(e). 
Our choice of $\log \lambda_1(z)$ is such that 
$\log \lambda_1(z) \to \pm \pi i$ as $z \to -\alpha^2$ with $ \pm \Im z > 0$,
Hence
\begin{equation} \label{eq:term3} 
	\log \lambda_{1}(z)
	= \pi  i \sgn(\Im z) + 2
	\frac{\sqrt{\alpha-\beta}}{\sqrt{\alpha+\beta}} 
	\left( \frac{z+\alpha^2}{\alpha^2} \right)^{1/2} 
	+ O(z+ \alpha^2) \end{equation}
as $z \to \infty$.
Using \eqref{eq:term1}, \eqref{eq:term2}, and \eqref{eq:term3}
in \eqref{eq:Phi1sum} we find the expansion \eqref{eq:Phi12nearsaddle} for $\Phi_1$.

The expansion for $\Phi_2$ also follows, since
$\lambda_2 = 1/\lambda_1$ by Lemma \ref{lem:lambdarho}~(d), which means that
$\log \lambda_2 = - \log \lambda_1$, and so
$\Phi_2$ is obtained from $\Phi_1$ by simply changing the
sign of $\xi_1$.
\end{proof}

\begin{corollary} \label{cor:Phi12scaling}
Suppose
\begin{equation} \label{eq:scaling2}
	\frac{z+\alpha^2}{\alpha^2} = c_0 s N^{-1/2}, \quad
 	\frac{w+\alpha^2}{\alpha^2} = c_0 t N^{-1/2}, 
\end{equation}
with constant $c_0$ as in \eqref{eq:constants2}. Then under
the scaling assumptions \eqref{eq:scaling1}, \eqref{eq:constants1}, 
we have 
\begin{multline} \label{eq:expPhizPhiw}
	\exp \left(N( \Phi_{1,2}(z; \xi_1, \xi_2) 
	- \Phi_1(w;\xi_1', \xi_2'))/2 \right) \\
	= (-1)^{(\Delta n-\Delta m)/2} (\pm 1)^{m} \alpha^{\Delta n} 
	  \frac{e^{\pm u s^{1/2} + \frac{1}{2} vs + \frac{1}{4} s^2}}
	  {e^{u't^{1/2} +\frac{1}{2} v't + \frac{1}{4} t^2}}
		\left(1  + O(N^{-1/4}) \right)
 \end{multline}
as $N \to \infty$.
\end{corollary}

\begin{proof}
From Lemma \ref{lem:Phi12} and the scalings \eqref{eq:scaling1} and \eqref{eq:scaling2} 
we find after straighforward calculations 
\begin{multline} \label{eq:Phiz11}
 \Phi_{1,2}(z; \xi_1, \xi_2)
	= 2 \log (\alpha^2+1) - (1+\xi_2) \log(\alpha^2)
		+ (1-\xi_2 \pm \xi_1) \pi i \sgn(\Im z) \\
		+ \left(\pm 2 u s^{1/2} 
		+ vs + \frac{1}{2} s^2 \right) N^{-1} + O(N^{-5/4})
 \end{multline}
and similarly for $\Phi_1(w; \xi_1', \xi'_2)$. Thus for $\Im z > 0$,
$\Im w > 0$,
\begin{multline} \label{eq:PhizPhiw}
 \Phi_{1,2}(z; \xi_1, \xi_2)  - \Phi_1(w;\xi_1', \xi_2') \\
	=  (\xi_2'- \xi_2) \log(\alpha^2)
		+ (-\xi_1' \pm \xi_1 + \xi_2'-\xi_2)  \pi i \\
		+ \left(2 (\pm u s^{1/2} - u't^{1/2}) 
		+ vs - v't + \frac{1}{2} (s^2-t^2) \right) N^{-1} + O(N^{-5/4}).
 \end{multline}
We note
\begin{equation} \label{eq:prefactor1} 
	e^{\frac{N}{2}(\xi_2'- \xi_2) \log(\alpha^2)}
	= e^{\frac{1}{2}(n'-n) \log (\alpha^2)}
	= \alpha^{\Delta n} 
	\end{equation}
since $\xi_2 = n/N-1$ and $\xi_2' = n'/N-1$, and
\begin{align} \label{eq:prefactor2}
	e^{\frac{N}{2}(-\xi_1' \pm \xi_1 + \xi_2'-\xi_2) \pi i}
	= e^{\frac{1}{2}(-m' \pm m + n'-n)) \pi i}
\end{align}
since also $\xi_1 = m/N-1$ and $\xi_2' = m'/N-1$.
Recall $m+n$ and $m'+n'$ are even, and $-m' \pm m +n' -n$
is even as well. Thus 
\begin{equation} \label{eq:prefactor3}
e^{\frac{N}{2}(-\xi_1' \pm \xi_1 + \xi_2'-\xi_2)} 
	= (-1)^{(\Delta n - \Delta m)/2} (\pm 1)^m.
	\end{equation}
The expansion \eqref{eq:expPhizPhiw} is immediate from
\eqref{eq:PhizPhiw}, \eqref{eq:prefactor1}, and \eqref{eq:prefactor3}.

We derived \eqref{eq:expPhizPhiw} for $\Im z > 0$ and $\Im w > 0$.
Similar calculations give \eqref{eq:expPhizPhiw} for
other combinations of signs of $\Im z$ and $\Im w$. 
\end{proof}

We also need the behavior of $F(z)$ near the saddle $z=-\alpha^2$.
\begin{lemma}
We have the following.
\begin{enumerate}
\item[\rm (a)] As $z \to -\alpha^2$,
\begin{multline} \label{eq:FwFz0}
 F(z) =   \frac{\sqrt{\alpha-\beta}}{2 
	\sqrt{\alpha+\beta}} \begin{pmatrix} 1 & 1 \\ -1 & -1 \end{pmatrix}
	\left( \frac{z+\alpha^2}{\alpha^2} \right)^{-1/2}
		+ \frac{1}{2} I_2 + O\left((z+\alpha^2)^{1/2}\right). 
\end{multline}
\item[\rm (b)] As $w,z \to -\alpha^2$,
\begin{multline} \label{eq:FwFz1}
	F(w) F(z) \\
	=  \frac{\sqrt{\alpha-\beta}}{4 \sqrt{\alpha+\beta}} 
	 \begin{pmatrix} 1 & 1 \\ -1 & -1 \end{pmatrix}
	 \left( \left( \frac{w + \alpha^2}{\alpha^2} \right)^{-1/2} + 
	 \left(\frac{z+\alpha^2}{\alpha^2} \right)^{-1/2} \right) 
		+ O(1). \end{multline}
\item[\rm (c)] As $w,z \to -\alpha^2$,
\begin{multline} \label{eq:FwFz2}
	F(w) (I_2 -  F(z)) \\
	= 
	 \frac{\sqrt{\alpha-\beta}}{4 \sqrt{\alpha+\beta}} 
	 \begin{pmatrix} 1 & 1 \\ -1 & -1 \end{pmatrix}
	 \left( \left(\frac{w + \alpha^2}{\alpha^2} \right)^{-1/2} 
	 - \left(\frac{z+\alpha^2}{\alpha^2} \right)^{-1/2} \right) 
		+ O(1). 
		\end{multline}
\end{enumerate}	
\end{lemma}
	
\begin{proof}
Part (a) follows from \eqref{eq:defF} where we have to take
into account that  $\sqrt{z(z+\beta^2)}$ is the negative square
root for $z=-\alpha^2$.

Parts (b) and (c) follow from part (a) since 
$F_0 = \begin{pmatrix} 1 & 1 \\ -1 & -1 \end{pmatrix}$ is a 
nilpotent matrix, i.e., $F_0^2 = 0_2$.
\end{proof}

\begin{corollary} \label{cor:FwFzscaling}
Under the same scaling  \eqref{eq:scaling2} 
as in Corollary \ref{cor:Phi12scaling} we have
\begin{equation} \label{eq:FwFz3}
	F(w) F(z) 
	=  \frac{2^{1/4} \sqrt{\alpha-\beta}}{4(\alpha+\beta)} 
	 \begin{pmatrix} 1 & 1 \\ -1 & -1 \end{pmatrix}
	 \left( s^{-1/2} + t^{-1/2} \right) N^{1/4}  
		+ O(1) \end{equation}
and
\begin{equation}  \label{eq:FwFz4}
	F(w) (I_2- F(z)) 
	=  \frac{2^{1/4} \sqrt{\alpha-\beta}}{4(\alpha+\beta)} 
	 \begin{pmatrix} 1 & 1 \\ -1 & -1 \end{pmatrix}
	 \left(- s^{-1/2} + t^{-1/2} \right) N^{1/4}  
		+ O(1) \end{equation}
		as $N \to \infty$.
\end{corollary}
\begin{proof}
This follows from inserting \eqref{eq:scaling2} into 
\eqref{eq:FwFz1} and \eqref{eq:FwFz2} and taking note of the
value \eqref{eq:constants2} for $c_0$.
\end{proof}

Now we are ready for the proof of Theorem \ref{thm:cusplimit}.

\begin{proof}[Proof of Theorem \ref{thm:cusplimit}] 
We deform the contour $\gamma_{0,1}$ in \eqref{eq:KNremainder} 
to the one shown  in Figure \ref{fig:defascdesc}. 
It consists of the steepest descent contour $\gamma_{sd,1}$ for
$\Phi_1( \cdot; 0, \xi_2^*)$  through $-\alpha^2$, 
but with a slight indentation around $-\alpha^2$ of size $O(N^{-1/2})$.
The steepest descent curve $\gamma_{sd,1}$ is actually a perfect 
circle passing through $-\alpha^2$ and $1$, since $\xi_1=0$ in the present situation.
Also $\Phi_1 = \Phi_2$, since $\xi_1=0$.
 
The contour $\Gamma$ in \eqref{eq:KNremainder} is split into two components,
as in the proof of Theorem \ref{thm:gaslimit}. We denote them by $\Gamma_1$
and $\Gamma_2$ as  shown in Figure \ref{fig:defascdesc}.
We reverse the orientation on both $\Gamma_1 \cup \Gamma_2$ and $\gamma_{0,1}$,
which does not change the values of the double integrals in \eqref{eq:KNremainder}.

	\begin{figure}[t]
		\begin{center}
			\begin{tikzpicture}
			\filldraw (0,0) circle (2pt);
			\filldraw (1,0) circle (2pt);
			\filldraw (-.75,0) circle (2pt);
			\filldraw[black] (-2,0) circle (2pt);
			\draw[very thick] (-.75,0)--(0,0);
			\draw[very thick] (-4,0)--(-2,0);
			\draw[help lines] (-4,0)--(4,0);
			\draw [red, line width=1.5mm,domain=0:163] plot ({1.5*cos(\x)-0.5}, {1.5*sin(\x)});
			\draw [red, line width=1.5mm,domain=197:360] plot ({1.5*cos(\x)-0.5}, {1.5*sin(\x)});
			\draw [red,->,line width = 1.5mm] (-.5,1.5)--(-0.45,1.5);
			\draw [red,line width = 1.5mm,domain=-85:85] plot ({.4*cos(\x)-2}, {.4*sin(\x)});
			\draw [blue,line width = 1.5mm,domain=-90:90] plot ({.2*cos(\x)-2}, {.2*sin(\x)});
			
			\draw[line width = 1.5mm, blue,->] (-2,.2)--(-3,.2);
			\draw[line width = 1.5mm, blue,-] (-4,.2)--(-3,.2);
			\draw[line width = 1.5mm, blue,-] (-2,-.2)--(-3,-.2);
			\draw[line width = 1.5mm, blue,->] (-4,-.2)--(-3,-.2);
			
			\draw [blue,line width=1.5mm,domain=90:270] plot ({.2*cos(\x)-.75}, {.2*sin(\x)});
			\draw [blue, line width=1.5mm,domain=-90:90] plot ({.2*cos(\x)}, {.2*sin(\x)});
			
			\draw[line width=1.5mm, blue,-] (0,.2)--(-.5,.2);
			\draw[line width=1.5mm, blue,-] (-.75,.2)--(-.3,.2);
			\draw[line width =1.5mm, blue,-] (-.3,-.2)--(0,-.2);
			\draw[line width=1.5mm, blue,->] (-0.75,-.2)--(-.15,-.2);
			
			\draw (-3.5,0.5) node {$\Gamma_1$};
			\draw (-0.5,0.5) node {$\Gamma_2$};
			\draw (1,1.1) node {$\gamma_{0,1}$};
			\draw (-2.3,-0.5) node {$-\alpha^2$};
			\draw (-0.9,-0.5) node {$-\beta^2$};
			\draw (0,-0.5) node {$0$};
			\draw (1.5,-0.3) node {$1$};
			\end{tikzpicture}
		\end{center}
		\caption{The  contours $\Gamma_1$ and $\Gamma_2$ and $\gamma_{0,1}$
		with the reversed orientation. 
		\label{fig:defascdesc} } 
	\end{figure}
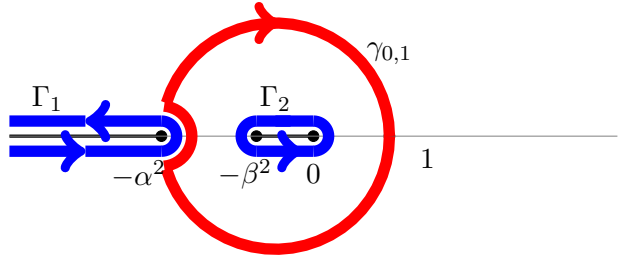
	
As in the proof of Theorem \ref{thm:gaslimit} we then have
\[ \Re \Phi_{1,2}(z)  < \Re \Phi_1(-\alpha^2) < \Re \Phi_1(w),
	 \]
whenever $z \in \gamma_{0,1}$ and 
$w \in (\Gamma_1 \setminus \{-\alpha^2\}) \cup \Gamma_2$.
From these inequalities we find that it is only a neighborhood of $-\alpha^2$ 
that contributes, and in particular the component $\Gamma_2$ that goes  
around $[-\beta^2,0]$ will not contribute to the limit.
Near $-\alpha^2$  we introduce the change of variables
\begin{equation} \label{eq:zwmapstost} 
	z= -\alpha^2 + \alpha^2 c_0 s N^{-1/2}, \quad 
	w = -\alpha^2 + \alpha^2 c_0 t N^{-1/2} 
\end{equation}
with $c_0$ as in \eqref{eq:constants2}. Then 
\begin{align} 
	\frac{dz dw}{z(z-w)}
	& = \frac{\alpha^2 c_0}{z}  \frac{ds dt}{s-t} \, N^{-1/2}
	\approx \frac{\alpha+\beta}{\sqrt{2}} \frac{ds dt}{t-s} \, N^{-1/2} 
	  \label{eq:dsdt} 
		\end{align}
where we used $z \sim -\alpha^2$ and the value \eqref{eq:constants2} for $c_0$.

From \eqref{eq:expPhizPhiw},  \eqref{eq:FwFz3}-\eqref{eq:FwFz4}
and \eqref{eq:dsdt}, we find that the leading order behavior for
the sum of the two double integrals in \eqref{eq:KNremainder} as 
$N \to \infty$ is  equal to 
\begin{multline} \label{eq:integralsSigma1Sigma2} 
	(-1)^{(\Delta n-\Delta m)/2}  \alpha^{\Delta n} 
	\frac{\sqrt{\alpha-\beta}}{4 \cdot 2^{1/4}} 
	\begin{pmatrix} 1 & 1 \\ - 1 & - 1 \end{pmatrix} N^{-1/4}
		\\	\times
	\left( \frac{1}{(2\pi i)^2} 
	\int_{s \in \Sigma_1} \int_{t \in \Sigma_2}
		\frac{e^{u s^{1/2} + \frac{1}{2} vs + \frac{1}{4} s^2}}
		{e^{u't^{1/2} + \frac{1}{2} v't + \frac{1}{4}t^2}}
			(s^{-1/2} + t^{-1/2})
			\frac{ds dt}{t-s} \right. \\
	\left. + (-1)^m
	\frac{1}{(2\pi i)^2} 
	\int_{s \in \Sigma_1} \int_{t \in \Sigma_2}
		\frac{e^{-u s^{1/2} + \frac{1}{2} vs + \frac{1}{4} s^2}}
		{e^{u't^{1/2} + \frac{1}{2} v't + \frac{1}{4}t^2}}
			(-s^{-1/2} + t^{-1/2})
			\frac{ds dt}{t-s} \right)
	\end{multline}
where $\Sigma_1$ and $\Sigma_2$ are as in Figure \ref{fig:PearceyContours2}.
$\Sigma_2$ is the image of $\Gamma_1$ under the mapping \eqref{eq:zwmapstost}.
It is a contour going around the negative real axis. $\Sigma_1$ arises from
applying \eqref{eq:zwmapstost} to the part of $\gamma_{0,1}$ in a small
disk around $-\alpha^2$, and deforming  and extending it so that 
$\Sigma_1$ is the imaginary axis with an indentation around $0$.

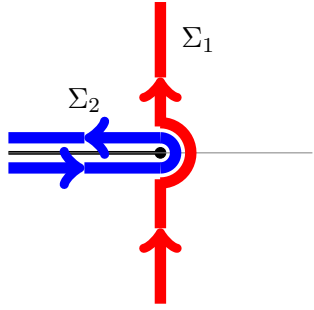
\begin{figure}[t]
\begin{center}
	\begin{tikzpicture}
	\filldraw[black] (-2,0) circle (2pt);
	\draw[line width = 0.5mm] (-4,0)--(-2,0);
	\draw[help lines] (-4,0)--(0,0);
	\draw [red,line width = 1.5mm,domain=-90:90] plot ({.4*cos(\x)-2}, {.4*sin(\x)});
	\draw [red,line width = 1.5mm,->] (-2,.35)--(-2,1);
	\draw [red, line width = 1.5mm,-] (-2,1)--(-2,2);
	\draw [red, line width = 1.5mm,->] (-2,-2)--(-2,-1);
	\draw [red, line width = 1.5mm,-] (-2,-1)--(-2,-.35);
	\draw [blue,line width = 1.5mm,domain=-90:90] plot ({.2*cos(\x)-2}, {.2*sin(\x)});
	\draw (-1.5,1.5) node {$\Sigma_1$};
	\draw (-3,.7) node {$\Sigma_2$};
	\draw[line width = 1.5mm, blue,->] (-2,.2)--(-3,.2);
	\draw[line width = 1.5mm, blue,-] (-4,.2)--(-3,.2);
	\draw[line width = 1.5mm, blue,-] (-2,-.2)--(-3,-.2);
	\draw[line width = 1.5mm, blue,->] (-4,-.2)--(-3,-.2);
	\end{tikzpicture}
\caption{The contours $\Sigma_1$ and $\Sigma_2$} \label{fig:PearceyContours2}
		\end{center}
	\end{figure}

We make a further change of variables by mapping 
$t = (t')^2$, $t' \in i \mathbb R$ in both integrals, and 
$s  = (s')^2$, $s' \in \Sigma$ in the first integral, but 
$s = (s')^2$, $s' \in (-\Sigma)$ in  the second double integral. 
Recall that $\Sigma$ and $-\Sigma$ are as in Figure \ref{fig:PearceyIntegrals}.
We find a sum of two double integrals, but the integrands are exactly the same.
We pick up a Jacobian $4s't'$ from the change  of variables,
and we use 
\[ \left((s')^{-1} + (t')^{-1}\right)  \frac{s' t' }{(t')^2-(s')^2} = \frac{1}{t'-s'} \]
to simplify the expression.
Then the integrand on the right-hand side of \eqref{eq:PearceyLimit1}
and \eqref{eq:PearceyLimit2} follows 
(provided we drop the accents from $s'$ and $t'$).

This completes the proof of Theorem \ref{thm:cusplimit}.
\end{proof}

\subsection{Cusp points: proof of Theorem \ref{thm:gaussianlimit}}
\label{subsec:gausslimit}

\begin{proof}
The proof of Theorem \ref{thm:gaussianlimit} is based on the 
integral representation  \eqref{eq:gaskernel2} for 
$\mathbb K_{gas}$. 

We already noted in Remark \ref{rem:decay}, that 
$\mathbb K_{gas}(m,n;m',n')$ tends to zero at
an exponential rate whenever $|\Delta m| + |\Delta n| \to \infty$.
In the situation of Theorem \ref{thm:gaussianlimit} we have
\begin{align*} 
	\Delta m & = (\xi_1'- \xi_1)N = c_1(u'-u) N^{1/4} + o(N^{1/4}), \\
	\Delta n & = (\xi_2'- \xi_2)N = c_2(v'-v) N^{1/2} + o(N^{1/2}), 
	\end{align*}
as $N \to \infty$.	Thus \eqref{eq:gaussianlimit} holds
if $v < v'$ or if $v = v'$ and $u \neq u'$.

So we assume $v > v'$. Then $\mathbb K_{gas}(m,n; m',n')$ 
still decays to $0$ as $N \to \infty$, but 
$\alpha^{-\Delta n}$ increases since $\Delta n < 0$. 
and the limit in 
\eqref{eq:gaussianlimit} will exist, as we shown next.

We argue similar to the proof of Theorem \ref{thm:cusplimit}. 
We start by deforming the contour $\gamma$ to a contour
going around the interval $(-\infty, -\alpha^2]$ as before. 
That is, we deform $\gamma$ to $\Gamma_1$ as in Figure \ref{fig:defascdesc} but with opposite direction. 
We can do this since for $\Delta n < 0$ there is enough decay 
of the integrand in \eqref{eq:gaskernel2} as $z \to \infty$.

It is then helpful to rewrite \eqref{eq:gaskernel2} to 
\begin{multline*}
\mathbb K_{gas} (m,n;m',n') \\
	=\begin{cases} \ds
 \frac{1}{2 \pi i} \int _{\gamma} F(z) e^{N (\Phi_1(z;\xi_1,\xi_2)-\Phi_1(z;\xi_1',\xi_2'))/2} \frac{dz}{z}, & 
 \text{if } \Delta m \geq 0, \\[10pt]
	\ds \frac{1}{2 \pi i} \int _{\gamma} (F(z)-I_2) e^{N (\Phi_2(z;\xi_1,\xi_2)-\Phi_2(z;\xi_1',\xi_2'))/2} \frac{dz}{z}, &
	\text{if } \Delta m <0.
\end{cases}
\end{multline*}
After a change of variables \eqref{eq:zwmapstost}, using the expansions \eqref{eq:Phiz11} and \eqref{eq:FwFz0}, and finally changing $s$ to $s^2$, all as in the proof of Theorem \ref{thm:cusplimit}, we find 
\begin{multline}
N^{1/4}(-1)^{(\Delta n-\Delta m)/2} \alpha^{-\Delta n}  \mathbb K_{gas}(m,n;m',n')\\
= 
\begin{cases} \ds
\frac{\sqrt{\alpha-\beta}}{2^{1/4} }\begin{pmatrix} 1& 1\\ -1 & -1\end{pmatrix} \frac{1}{2 \pi i} 
\int_{-i 	\infty}^{i \infty} e^{\frac12 s^2 (v-v')+ s (u-u')} ds +o(1),& \Delta m \geq 0,\\[10pt] 
\ds
\frac{(-1)^{\Delta m}
\sqrt{\alpha-\beta}}{2^{1/4}} \begin{pmatrix} 1& 1\\ -1 & -1\end{pmatrix} \frac{1}{2 \pi i}  
\int_{-i 	\infty}^{i \infty} e^{\frac12 s^2 (v-v')- s (u-u')} ds +o(1), &\Delta m<0,
\end{cases}
\end{multline}
as $N\to \infty$. 
The limit \eqref{eq:gaussianlimit} then follows after 
observing that the integral is a well-known representation 
of the gaussian 
\[
\int_{-i \infty}^{i \infty} e^{\frac12 s^2 (v-v')\pm s (u-u')} ds= \frac{\sqrt{2 \pi } i}{\sqrt{v-v'}}e^{-\frac{(u-u')^2}{v-v'}},
\qquad \text{ for } v > v'.
\]
This concludes the proof of Theorem \ref{thm:gaussianlimit}.
\end{proof}

\section{Acknowledgements}
We thank Sunil Chhita and Kurt Johansson for many
fruitful discussions. We are also grateful to Christophe Charlier for a careful reading of a preliminary manuscript. 

Most of this work was done when Arno Kuijlaars was a visiting
professor at KTH in spring semester of 2017, funded by
the Knut and Alice Wallenberg Foundation. 

Maurice Duits is supported by the grants 2012.3128 and
2016.05450 of the Swedish Research Council and by the 
G\"oran Gustafsson foundation. 

Arno Kuijlaars is supported by long term structural 
funding-Methusalem grant of  the Flemish Government, 
and via support of FWO Flanders through projects  
G.0864.16 and EOS G0G9118N.

\end{document}